\documentclass[11pt]{amsart}
\usepackage[utf8]{inputenc}
\usepackage{amsmath, palatino, mathpazo, amsfonts, amssymb, mathtools}
\usepackage[mathscr]{eucal}
\usepackage[all]{xy}
\usepackage{datetime}
\usepackage{amsthm}
\usepackage{comment}
\usepackage{multicol}
\usepackage{graphicx}
\usepackage{tikz-cd}
\usepackage{subcaption}
\usepackage{enumitem}
\usetikzlibrary{calc}
\usepackage[backref=page]{hyperref}

\usetikzlibrary{decorations.markings, arrows.meta, calc}
\usepackage{marginnote}

\newtheorem{theorem}{Theorem}[section]
\newtheorem{lemma}[theorem]{Lemma}
\newtheorem{proposition}[theorem]{Proposition}
\newtheorem{corollary}[theorem]{Corollary}

\theoremstyle{remark}
\newtheorem{conjecture}[theorem]{Conjecture}
\newtheorem{remark}[theorem]{Remark}
\newtheorem{example}[theorem]{Example}
\newtheorem{definition}[theorem]{Definition}

\numberwithin{equation}{section}

\newcommand{\Sym}{{\rm Sym}}

\title{A theory of generalized Lam\'e curves}
\author{You-Cheng Chou}
\address{June E Huh Center for Mathematical Challenge}
\email{bensonchou72@gmail.com, bensonchou@kias.re.kr}

\author{Chin-Lung Wang}
\address{Department of Mathematics, National Taiwan University}
\email{dragon@math.ntu.edu.tw}

\author{Po-Sheng Wu}
\address{Department of Mathematics, National Taiwan University}
\email{d10221002@ntu.edu.tw}

\date{\today}
\begin{document}

\begin{abstract}
We study the generalized Lam\'e equation (GLE) on an elliptic curve $E_\tau$ with multiple regular singularities $\mathbf{p} = (p_i)_{i = 1}^r$ of weights $\mathbf{n} = (n_i)_{i = 1}^r$, $n_i \in \mathbf{C}$. We construct two fundamental algebraic curves:

(i) The generalized Lam\'e curve (GLC): For total weight $n:=\sum\nolimits_{i=1}^r n_i\in\mathbb Z_{\geq 0}$, by analyzing the quasi-periodic solutions, we construct $\mathcal{Y}_{\mathbf{n}, \mathbf{p}}(\tau)$, which lies in an affine bundle over $\Sym^n E_\tau$ and parametrizes the generalized Hermite--Halphen (GHH) ansatz. 

(ii) The log-free curve: For each $i$, the constraint $n_i\in\tfrac{1}{2}\mathbb{N}$ gives a polynomial equation in the accessory parameters. This leads to a non-complete intersection variety $V_{\mathbf{n}, \mathbf{p}}(\tau)$ when all $n_i\in\tfrac{1}{2} \mathbb{N}$. We prove that $V_{\mathbf{n},\mathbf{p}}(\tau)$ is a reduced curve via its relation with GLC and by analyzing the reducedness of the special fiber in a flat limit, confirming and strengthening a conjecture made by the second author in \cite{Wang_2020}.

The central technical achievement of this work is the structural analysis of the GLC as an algebraic family over the pole configuration space $\mathbf{p}$. We control this global geometry via the addition map
\[
\sigma_{\mathbf{n},\mathbf{p}} \colon \Sym^n E_\tau \longrightarrow E_\tau,
\]
yielding a generically finite, universal degree formula. Remarkably, the geometry of boundary degenerations under pole collisions perfectly mirrors the tensor algebra of $\mathfrak{sl}_2(\mathbb{C})$-modules within the BGG category $\mathcal{O}$. We explicitly obtain the special fibers of this family by fixing a generic image $c$ and the degree of the addition map during specialization—a construction we formalize as the degeneration theorem. Furthermore, we show that these pointwise limits lie within GLC components of lower weights, and that these boundary components enhance to establish the global flatness of the GLC over both the pole configuration and weight spaces.

To study the monodromy of the GLE, we develop a framework of twisted isomonodromic deformations, where the ``twist'' accounts for the monodromy representation being defined only up to an explicit scalar. Generalizing the theory of pre-modular forms \cite{Lin_Wang_2017}, we construct $(\mathbf{n}, \mathbf{p})$-deformed pre-modular forms parameterized by this twisted monodromy data $(t,s)$. Their vanishing solves the monodromy problem, and they factorize along boundary strata via the degeneration theorem. Iterating these deformations through the boundary allows any arbitrary configuration to be continuously deformed to the classical Lam\'e equation.

Finally, we apply the asymptotic scaling technique, originally developed for the study of the addition map, to general symmetric elliptic equilibrium systems. This yields a rational limit that extracts the intersection multiplicities of local root clusters, and this structural rigidity gives a complete extension of Treibich's generic counting \cite{Treibich_2025} for $r=2$ symmetric pairs to all smooth elliptic curves.
\end{abstract}

\maketitle

\tableofcontents

\maketitle

\sloppy

\vspace{1em}
\noindent \textbf{Notation and conventions.} Throughout this paper, we adopt the following standard notations:
\begin{itemize}[leftmargin=*]
    \item Let $\Lambda_\tau = \mathbb{Z} + \mathbb{Z}\tau$ be the lattice with $\tau \in \mathbb{H} = \{ \tau \in \mathbb{C} \ \mid \ \Im (\tau) >0 \}$ and $E_\tau = \mathbb{C}/\Lambda_\tau$ be the associated elliptic curve. $[p] \in E_\tau$ denotes the canonical projection of $p \in \mathbb{C}$.
    \item We suppress $\tau$ in the Weierstrass functions: 
    \[
    \wp(z) := \wp(z;\tau), \quad \zeta(z) := \zeta(z;\tau), \quad \sigma(z) := \sigma(z;\tau), \quad \eta_i := 2 \zeta( \tfrac{1}{2} \omega_i ).
    \]
    \item We use boldface letters (e.g., $\mathbf{a}, \mathbf{n}, \mathbf{p}$) to denote finite tuples, with their components written in standard font (e.g., $a_\mu, n_i, p_i$).
    \item $\{ \dots \} \in \Sym^n X$ denotes an unordered symmetric class. In particular, a point in $\Sym ^n E_{\tau}$ is denoted by $\{[\mathbf{a}]\}$. For elements $\{ \mathbf{a} \} \in \Sym^n X$ and $\{ \mathbf{b} \} \in \Sym^m X$, their natural multiplication is denoted by $\{ \mathbf{a} \} \times \{ \mathbf{b} \} \in \Sym^{n+m} X$. 
\end{itemize}

\section{Introduction}

The classical Lam\'e equation on an elliptic curve occupies a distinguished position at the intersection of differential equations, algebraic geometry, and integrable systems. Its solutions exhibit special algebraic and arithmetic properties reflected in the geometry of auxiliary algebraic curves, commonly referred to as Lam\'e or spectral curves. These features persist, in a more intricate form, for differential equations on elliptic curves with multiple singularities.

\subsection{Generalized Lam\'e equations}
We develop a systematic algebro-geometric theory of generalized Lam\'e equations (GLE) on elliptic curves and their associated generalized Lam\'e curves (GLC). Let $E_\tau=\mathbb{C}/\Lambda_\tau$ be an elliptic curve. A generalized Lam\'e equation is a second-order differential equation
\begin{equation}\label{eq:GLE}
\frac{d^2 w}{dz^2}-Q(z)w=0,
\end{equation}
whose potential is an elliptic function with prescribed principal parts
\begin{equation}\label{eq:potential}
Q(z)=\sum_{i=1}^r n_i(n_i+1)\wp(z-p_i)+\sum_{i=1}^r A_i\zeta(z-p_i)+B,
\qquad \sum_{i=1}^r A_i=0.
\end{equation}
Here, $p_1,\dots,p_r \in \mathbb{C}$ are regular-sigularities, $n_i\in\mathbb{C}$ are fixed weights, and $(A_i,B)\in \mathbb{C}^{r+1}$ are accessory parameters. In the literature, this equation is also regarded as a meromorphic projective connection of Fuchsian type \cite{Iwasaki_1991, Kawai_1996}.

While it suffices to view the singularities $p_i$ on $E_\tau$ for a fixed equation, studying moduli problems---such as isomonodromic deformations---requires lifting them to the universal cover $\mathbb{C}$ to track their continuous trajectories.

Our primary focus is on quasi-periodic solutions, which take the exact form of the generalized Hermite--Halphen (GHH) ansatz (see \eqref{eq:GHH} below). As a direct consequence of Theorem~\ref{t:double_multi_sol}, we impose the integrality condition
\[
n := \sum_{i=1}^r n_i \in \mathbb{Z}_{\geq 0}.
\]
This assumption is essential for the existence of the multi-valued ansatz globally defined on the complex plane.

\subsubsection{Log-free varieties}

When some $n_i\in\frac12\mathbb{N}$, the local Frobenius solutions of \eqref{eq:GLE} may contain logarithmic terms. The vanishing of the corresponding logarithmic obstructions imposes polynomial conditions on the accessory parameters and defines the log-free locus. These log-free GLEs play a central role in monodromy problems and arise naturally in the study of singular Liouville equations on elliptic curves. 

Specifically, assign $\deg A_j = 1$ for all $j$ and $\deg B = 2$. For each $\ell_i := 2n_i \in \mathbb{N}$, the corresponding obstruction polynomial has degree $\ell_i + 1$ and takes the form (let $\ell = \ell_i$):
\[
F_\ell(A, B) = q_\ell(A, B) + \dots = \frac{(-1)^\ell}{(\ell !)^2}\prod_{j = 0}^\ell (A - (\ell - 2j)B^{1/2}) + \dots
\]
The top-degree term explicitly exhibits the branching structure of $\mathfrak{sl}_2(\mathbb{C})$ representations. While originally discovered in \cite{Wang_2020} via direct computation, a more concise $\mathfrak{sl}_2(\mathbb{C})$ representation-theoretic proof is given in Proposition \ref{p:topterm}, which is the starting point of the BGG category study in \ref{ss:BGG}. 

An affine log-free variety $V_{\mathbf{n},\mathbf{p}}(\tau) \subset \mathbb{C}^{r + 1}$ is then defined by the equations $F_{\ell_i}(A, B) = 0$ for indices where $n_i \in \tfrac12 \mathbb{N}$. In the extreme case where $n_i \in \tfrac12 \mathbb{N}$ for all $i = 1, \ldots, r$, coupled with the constraint $\sum A_i = 0$, this yields exactly $r + 1$ equations in $r + 1$ variables. Remarkably, it is only in this extreme case that these equations may fail to intersect transversely (cf.~Remark~\ref{r:CI}), obscuring the actual geometric structure.

It was conjectured in \cite{Wang_2020} that $V_{\mathbf{n},\mathbf{p}}(\tau)$ for $\mathbf{n} \in (\frac12 \mathbb{N})^r$ should contain curve components. Proving this conjecture was one of the original motivations for initiating this project around 2017. The solution obtained here is far more involved than originally anticipated. Resolving this conjecture requires constructing the moduli space of the GHH ansatz---which natively forms a curve---and proving the existence of a non-constant (generically 2:1) morphism from this moduli space to $V_{\mathbf{n},\mathbf{p}}(\tau)$.

\subsubsection{Generalized Hermite--Halphen ansatz}
A key conceptual contribution of this work is the identification of the generalized Hermite--Halphen (GHH) ansatz as the structural mechanism governing log-free generalized Lam\'e equations. The ansatz is defined for any total weight $n \in \mathbb{Z}_{\geq 0}$. For strictly positive weight $n>0$, it takes the form
\begin{equation}\label{eq:GHH}
w_{\mathbf{n},\mathbf{p} ; \mathbf{a}}(z)=
\exp\!\left(
\frac{z}{n}\sum_{i=1}^r n_i\sum_{\mu=1}^n\zeta(a_\mu-p_i)
\right)\,
\frac{\prod_{\mu=1}^n\sigma(z-a_\mu)}
{\prod_{i=1}^r\sigma(z-p_i)^{n_i}},
\end{equation}
where $\mathbf{a} = (a_1,\dots,a_n) \in \mathbb{C}^n$ whose projections define an unordered symmetric class $\{[\mathbf{a}]\} := \{[a_1], \dots, [a_n]\} \in \Sym^n E_\tau$. In the boundary case of zero total weight ($n=0$), the numerator product is empty, and the ansatz simplifies to a form parameterized by a free constant $h \in \mathbb{C}$:
\begin{equation}\label{eq:GHH_zero}
w_{\mathbf{n},\mathbf{p}; h}(z)= \frac{\exp(hz)} {\prod_{i=1}^r\sigma(z-p_i)^{n_i}}.
\end{equation}

Proposition~\ref{exist_ansatz_sol} shows that the GHH ansatz is not merely a source of explicit solutions but is in fact \emph{necessary}: any generalized Lamé equation with $\mathbf{n} \in (\frac12 \mathbb{N})^r$ which admits only log–free solutions must possess a solution of the form \eqref{eq:GHH}, possibly after lowering the total weight. This result already encodes the compactification behavior required for the construction of generalized Lamé curves and provides the conceptual bridge between analytic log–free conditions and algebraic geometry on $\Sym^n E_\tau$.

\subsubsection{Generalized Lam\'e curves and deformation}
Motivated by this observation, we define $\mathcal{Y}_{\mathbf{n},\mathbf{p}}(\tau)$ to be the moduli space of GHH data corresponding to log-free generalized Lam\'e equations, and we denote by $\overline{\mathcal{Y}}_{\mathbf{n}, \mathbf{p}}(\tau)$ its natural compactification. These spaces define the \emph{generalized Lam\'e curves}. A central theme of this work is the study of $\overline{\mathcal{Y}}_{\mathbf{n},\mathbf{p}}(\tau)$ as an algebraic family as the poles $p_i$ vary and collide, with the aim of understanding degenerations of generalized Lam\'e equations through algebro-geometric methods.

Associated to the GHH data is the $(\mathbf{n},\mathbf{p})$-shifted addition map
\begin{equation}\label{eq:summation_intro}
\sigma_{\mathbf{n},\mathbf{p}} \colon \Sym^n E_\tau \longrightarrow E_\tau, \quad \{[\mathbf{a}]\} \longmapsto \Big[\sum_{\mu=1}^n a_\mu - \sum_{i=1}^r n_i p_i \Big],
\end{equation}
which plays a fundamental role in controlling the global geometry of generalized Lam\'e curves. The behavior of this map under deformation and degeneration provides the main geometric input for the results stated below.

\subsection{Main results} 
Here we summarize the principal results. The overarching strategy is to establish the geometric structure of generalized Lam\'e equations on the generic fiber, and then use the degeneration theorem to explicitly analyze their properties via pole collisions.

\subsubsection{Existence of GLC for general parameters}
For fixed weights $\mathbf{n} = (n_i)$ with total weight $n \in \mathbb{Z}_{\geq 0}$, and a pole configuration $\mathbf{p} = (p_i) \in \mathbb{C}^r$, we construct the moduli space $\mathcal{Y}_{\mathbf{n},\mathbf{p}}(\tau)$ of the GHH ansatz. 

Recall the coefficient of $z$ in the exponential factor:
\[
h = \frac{1}{n}\sum_{i=1}^r \sum_{\mu=1}^n n_i\zeta(a_\mu-p_i).
\]
Strictly speaking, the full GHH data is captured by the equivalence class $[(\mathbf{a}, h)]$, which naturally lies in an affine bundle over $\Sym^n E_\tau$. While much of the geometric structure can be studied purely via $\{[\mathbf{a}]\} \in \Sym^n E_\tau$, the introduction of $h$ is required not only to study the compactification of the moduli space within the total space of the corresponding projective bundle, but also to ensure that the morphism $\mathcal{Y}_{\mathbf{n},\mathbf{p}}(\tau) \to V_{\mathbf{n},\mathbf{p}}(\tau)$ extends to this compactification.

The natural projection $\pi \colon \overline{\mathcal{Y}}_{\mathbf{n},\mathbf{p}}(\tau) \to \overline{Y}_{\mathbf{n},\mathbf{p}}(\tau)$, which forgets the $h$ parameter, is an isomorphism on the interior and finite on the boundary. The interior base space $Y_{\mathbf{n},\mathbf{p}}(\tau) \subset \Sym^n E_\tau$ is explicitly defined as the quasi-projective variety cut out by the system of residue equations. For $\mu = 1, \dots, n$:
\begin{equation}
\sum_{i=1}^{r} \sum_{\substack{\nu=1 \\ \nu \neq \mu}}^{n} n_{i} \Big( \zeta(a_{\mu} - a_{\nu}) - \zeta(a_{\mu} - p_{i}) + \zeta(a_{\nu} - p_{i}) \Big) = 0.
\end{equation}

\begin{theorem}[cf.~Theorem~\ref{t:GLC}]
$Y_{\mathbf{n}, \mathbf{p}}(\tau)$ and $\overline{Y}_{\mathbf{n}, \mathbf{p}}(\tau)$ both consist of a finite union of curves.
\end{theorem}

We refer to this base projection as the \emph{underlying generalized Lam\'e curve (GLC)}, distinguishing it from the total \emph{generalized Lam\'e curve} $\overline{\mathcal{Y}}_{\mathbf{n}, \mathbf{p}}(\tau)$. In the boundary case of zero total weight ($n=0$), the base space $Y_{\mathbf{n}, \mathbf{p}}(\tau)$ degenerates to a point, and $\mathcal{Y}_{\mathbf{n},\mathbf{p}}(\tau)$ is entirely parameterized by the free parameter $h$.

\subsubsection{Global flatness and the BGG category $\mathcal{O}$ dictionary} \label{ss:BGG}
A central geometric achievement of this work is establishing that the generalized Lam\'e curves assemble into a universally flat family.

\begin{theorem}[cf.~Theorem~\ref{thm:universal_flatness}]
There exists a globally flat family
\[
    \overline{\mathcal{Y}} \longrightarrow \mathbb{C}^r \times \Big\{ \mathbf{n} \in \mathbb{C}^r \;\Big|\; \sum_{i=1}^r n_i = n \Big\}
\]
whose generic fiber over $(\mathbf{n},\mathbf{p})$ is canonically isomorphic to the generalized Lam\'e curve $\overline{\mathcal{Y}}_{\mathbf{n}, \mathbf{p}}(\tau)$.
\end{theorem}

We control this global geometry via the $(\mathbf{n},\mathbf{p})$-shifted addition map
\[
\sigma_{\mathbf{n},\mathbf{p}} \colon \Sym^n E_\tau \longrightarrow E_\tau, \quad \{[\mathbf{a}]\} \longmapsto \Big[\sum_{\mu=1}^n a_\mu - \sum_{i=1}^r n_i p_i \Big].
\]
For a generic target $c \in E_\tau$, its restriction to $\overline{Y}_{\mathbf{n}, \mathbf{p}}(\tau)$ is generically finite of a universal combinatorial degree (cf.~Theorem~\ref{thm_deg}). We denote the corresponding fibers over a fixed target $c \in E_\tau$ by
\[
    \overline{Y}_{\mathbf{n},\mathbf{p}}^c(\tau) := \overline{Y}_{\mathbf{n},\mathbf{p}}(\tau) \cap \sigma^{-1}_{\mathbf{n},\mathbf{p}}(c), \quad \text{and} \quad \overline{\mathcal{Y}}_{\mathbf{n},\mathbf{p}}^c(\tau) := \overline{\mathcal{Y}}_{\mathbf{n},\mathbf{p}}(\tau) \cap \pi^{-1}\sigma^{-1}_{\mathbf{n},\mathbf{p}}(c).
\]
The conservation of this degree enforces rigid structural formulas at the boundary. By specializing parameters, we uncover that the special fibers translate perfectly into the formal Grothendieck ring $R[q]$ of the BGG category $\mathcal{O}$.

To each singular pole $p_i$ with weight $n_i$, we associate a highest-weight module. Let $[M_{2n_i}]$ denote the class of the Verma module (for generic $n_i$), and $[L_{2n_i}]$ denote its finite-dimensional irreducible quotient (for $n_i \in \frac{1}{2}\mathbb{Z}_{\geq 0}$). The formal variable $q$ tracks collision multiplicity, where $q^k$ corresponds geometrically to a collision cluster $\{p^k\}$.

The boundary behavior of the generalized Lam\'e curve is governed by two rules (cf.~Definition~\ref{def:BGG}):

\vspace{0.5em}
\noindent \textbf{1. Decomposition rule ($\mathbf{n}$-direction):} For $n_1 \in \frac{1}{2}\mathbb{Z}_{\geq 0}$, the geometric splitting of the curve corresponds exactly to the algebraic BGG resolution:
\[
    [M_{2n_1}] = [L_{2n_1}] + q^{2n_1+1} [M_{-2n_1-2}].
\]
Geometrically, the moduli space decomposes into a primary component of weight $n$ and a boundary stratum where $2n_1+1$ roots collide at $p_1$:
\begin{theorem}[Decomposition theorem, cf.~Theorem~\ref{thm_decom}]
Assume $n_1 \in \frac{1}{2}\mathbb{Z}_{\geq 0}$, and $n_i$ are generic for $i \ge 2$. We have:
\[
\begin{split}
    \lim_{\delta \rightarrow 0}\overline{\mathcal{Y}}_{(n_1+\delta,n_2-\delta,n_3,\dots,n_r),\mathbf{p} }^c(\tau) &= \overline{\mathcal{Y}}_{\mathbf{n},\mathbf{p} }^c(\tau) \\
    &\quad \cup \ \{p_1^{2n_1+1}\} \times \overline{\mathcal{Y}}^{c'}_{(-n_1-1,n_2,\dots,n_r),\mathbf{p}} (\tau),
\end{split}
\]
where $c' := c - (2n_1+1) [p_1]$.
\end{theorem}

\vspace{0.5em}
\noindent \textbf{2. Generic tensor rule ($\mathbf{p}$-direction):} For general parameters, the collision of poles $p_1 \to p_2$ corresponds to the tensor product of Verma modules:
\[
    [M_{2n_1}] \otimes [M_{2n_2}] = \sum_{k=0}^\infty q^k [M_{2n_1+2n_2-2k}].
\]
During pole collision, the root configuration partitions into two scales: a local rational limit zooming into the collision, $\{p_{12}^k\}$, and the remaining coordinates converging to a lower-weight GLC, $\overline{\mathcal{Y}}_{\mathbf{n}',\mathbf{p}'}(\tau)$. To be precise,
define the shift operator 
\[
S_v \colon [(\mathbf{a}, h)] \mapsto [(\mathbf{a}, h+v)]
\]
on $\overline{\mathcal{Y}}_{\mathbf{n},\mathbf{p}}(\tau)$. For $\omega \in \Lambda_\tau$, let $\mathbf{p}_{12, \epsilon}^\omega := (p_{12} + \omega + \epsilon, p_{12} - \epsilon, p_3, \dots, p_r)$. As $\epsilon \to 0$, the limit cleanly fractures into a union of boundary strata shifted by the quasi-period $-n_1\eta(\omega)$:
\begin{theorem}[Degeneration theorem, cf.~Theorem~\ref{thm:dege_generic}]
\[
\lim_{\epsilon \to 0}\overline{\mathcal{Y}}_{\mathbf{n},\mathbf{p}_{12, \epsilon}^\omega}^c(\tau) =
\bigcup_{k=0}^{\infty} \{p_{12}^k\} \times S_{-n_1 \eta(\omega)} \Big( \overline{\mathcal{Y}}^{c_k}_{\mathbf{n}_{12}^{(k)},\mathbf{p}_{12}}(\tau) \Big),
\]
where $c_k := c - k[p_{12}]$, $\mathbf{p}_{12}=(p_{12},p_3,\dots,p_r)$ and $\mathbf{n}_{12}^{(k)} := (n_1+n_2-k,n_3,\dots,n_r)$.
\end{theorem}

The proof of the global flatness theorem is constructive. By iteratively applying these two rules at boundary strata, we explicitly build local power series deformations to establish flatness over the higher-dimensional parameter space. The primary significance of this result is that it dictates the exact factorization of the pre-modular form along the special fibers, which is essential to solving the monodromy problem.

\subsubsection{Curveness proof of log-free variety}
For the exceptional case $\mathbf{n} \in (\tfrac{1}{2}\mathbb{N})^r$, we solve a conjecture proposed by Wang \cite{Wang_2020}, which states that the log-free variety $V_{\mathbf{n},\mathbf{p}}(\tau)$ contains one-dimensional curve components. 

Our proof relies on the natural projection from the generalized Lam\'e curve:
\[
    \pi \colon \mathcal{Y}_{\mathbf{n},\mathbf{p}}(\tau) \longrightarrow V_{\mathbf{n},\mathbf{p}}(\tau).
\]
By tracking its asymptotic behavior near the infinity divisor $B=\infty$, we establish that $\pi$ maps the (Type-I) boundary points of the GLC surjectively onto all possible combinatorial branches of $V_{\mathbf{n},\mathbf{p}}(\tau)$. This exhaustive geometric realization not only proves Wang's conjecture in full generality, but parameterizes the branches up to a natural involution. This complete characterization of the boundary directly enables the explicit computation of the Hilbert quasi-polynomial for the compactified log-free stack $\overline{V}_{\mathbf{n},\mathbf{p}}(\tau)$.

Finally, we significantly strengthen the result by eliminating the possibility of zero-dimensional components:
\begin{theorem}[cf.~Theorem~\ref{thm:log_free_curveness}]
$\overline{V}_{\mathbf{n},\mathbf{p}}(\tau)$ is a reduced curve without isolated or embedded points.
\end{theorem}
To achieve this, we construct a weighted 1-parameter flat degeneration of the system. Under this degeneration, the ideal specializes to a highly degenerate central fiber defined by the top-degree terms of the log-free conditions. By using Stanley's result on the Weak Lefschetz Property for Artinian complete intersections, we verify that this naive limit scheme is exactly the initial Gr\"obner ideal, forming a reduced curve with no embedded origins. The principle of flat families then guarantees that the generic fiber $V_{\mathbf{n},\mathbf{p}}(\tau)$ inherently assumes this exact reduced curve structure.

\subsubsection{Application I: Pre-modular form, factorization, and twisted isomonodromy deformation}

We study the isomonodromic deformations of the GHH ansatz solutions over the configuration space $\mathbf{p} \in \mathbb{C}^r$. For general non-integer weights $\mathbf{n}$, the monodromy eigenvalues $\lambda_{\gamma_i}$ factorize into a path-dependent winding constant $e^{\Theta(\gamma_i)}$ and an underlying pair of parameters $(t,s) \in (\mathbb{C}/\mathbb{Z})^2$, which we call the \emph{twisted monodromy}. For an ansatz solution $[(\mathbf{a}, h)] \in \overline{\mathcal{Y}}_{\mathbf{n},\mathbf{p}}(\tau)$, this twisted data $(t,s)$ is uniquely determined by the addition map $\widetilde{\sigma}_{\mathbf{n},\mathbf{p}}:\Sym^n \mathbb{C} \rightarrow \mathbb{C}$ on the universal cover via the linear system:
\[
\begin{cases}
    \tilde{t} \omega_1 + \tilde{s} \omega_2 &= \widetilde{\sigma}_{\mathbf{n},\mathbf{p}}(\mathbf{a}), \\
    \tilde{t} \eta_1 + \tilde{s} \eta_2 &= h.
\end{cases}
\]
where $(\tilde{t}, \tilde{s}) \in \mathbb{C}^2$ project to the equivalence class $(t,s)$.

Finding the existence of solutions with prescribed twisted monodromy data $(t,s)$ poses a highly transcendental problem, as the conformal structure $\tau$ must deform non-trivially alongside the configuration $\mathbf{p}$. To systematically track this interdependence, we generalize the pre-modular form theory developed in \cite{Lin_Wang_2017} to construct an $(\mathbf{n}, \mathbf{p})$-deformed version, denoted $\Phi_{\mathbf{n},\mathbf{p};t,s}(\tau)$. A crucial feature of this construction is that the degeneration of the special fibers translates directly into an explicit algebraic factorization:
\begin{corollary}(cf.~Corollary~\ref{cor:modular_decomp})
Let $\omega = m_1\omega_1 + m_2 \omega_2 \in \Lambda_\tau$. Then$$    \lim_{\epsilon \rightarrow 0} \Phi_{\mathbf{n},\mathbf{p}_{12, \epsilon}^\omega;t,s}(\tau) = \prod_{k=0}^{n} \Big( \Phi_{\mathbf{n}_k,\mathbf{p}_{12};t',s'}(\tau) \Big)^{m_k},$$where $\mathbf{n}_{12}^{(k)}$ and $\mathbf{p}_{12}^{\omega}$ are defined in Theorem~\ref{thm:dege_generic}, $m_k$ is the multiplicity of the corresponding component, and $t' = t + m_1 n_1, \ s'= s + m_2 n_1.$\end{corollary}

The factorization of this deformed pre-modular form under the degeneration theorem explicitly governs the system as $\mathbf{p} \to \mathbf{0}$. By implicitly differentiating $\Phi_{\mathbf{n},\mathbf{p};t,s}(\tau) = 0$, we extract the differential equation governing the twisted isomonodromic deformation. For generic twisted monodromy data (avoiding a specific locus detailed later), lifting a generic path $\mathbf{p} \to \mathbf{0}$ along this deformation leads directly to the following reduction:
\begin{theorem}[cf.~Corollary~\ref{cor:isomo_deform}]
For generic $(t,s)$, fix an initial point $[(\mathbf{a}_0, h_0)] \in \overline{\mathcal{Y}}^{t,s}_{\mathbf{n}, \mathbf{p}_0}(\tau_0)$. Following the twisted isomonodromic deformation along a generic path $\mathbf{p}_0 \to \mathbf{0}$ analytically continues this initial state to a well-defined limit point $[(\mathbf{a}^*, h^*)] \in \overline{\mathcal{Y}}^{t,s}_{k,0}(\tau^*)$ in the special fiber, reducing the system to classical case for some $0 \leq k \leq n$ and $\tau^* \in \mathbb{H}$.
\end{theorem}
\subsubsection{Application II: Completing Treibich's counting formula with multiplicities}

Treibich \cite[Theorem~5.15]{Treibich_2025} proved that, for a generic
smooth elliptic curve and any $(n_0,n_1,n_2,n_3)\in\mathbb Z_{\geq0}^4$,
there are exactly $27$ finite-gap potentials of the form
\[
 I_{\mathbf p}(z)=
 \sum_{i=0}^3 n_i(n_i+1)\wp\Big(z-\frac{w_i}{2}\Big)
 +2\sum_{\mu=1}^2\bigl(\wp(z-p_\mu)+\wp(z+p_\mu)\bigr),
\]
with the four moving poles $\{\pm p_1,\pm p_2\}$ distinct and away from the half-periods. We complete the counting statement by proving that the total is $27$ (with multiplicity) on every smooth elliptic curve. 

We deform the elliptic equilibrium equations to generic constants and compute the local contributions near the half-periods. This approach may be viewed as an elliptic Bézout-type counting method and suggests a framework for more general symmetric elliptic systems. In this paper, we restrict to two moving pole pairs. By showing that no solutions escape to the boundary as the deformation parameters tend to zero, we obtain the following count with multiplicities.

\begin{theorem}[Treibich's counting formula, cf.~Corollary~\ref{cor:Treibich}]
For every smooth elliptic curve and every
$(n_0,n_1,n_2,n_3)\in\mathbb Z_{\geq0}^4$,
\[
 T(n_0,n_1,n_2,n_3;2)=27,
\]
where $T$ counts the number of finite-gap potentials $I_{\mathbf{p}}(z)$ with multiplicity.
\end{theorem}
For a generic elliptic curve, this agrees with
Treibich's count of distinct potentials. The extension is that the multiplicity count holds for every smooth elliptic curve.

\subsection{Relation to previous works}
Special cases of generalized Lamé equations have been extensively studied, though typically with restrictive constraints on the pole configurations.

\subsubsection{Finite-gap theory} In the context of integrable systems, classical Lamé equations and their generalizations have been deeply investigated by Airault, McKean, Moser, Treibich, Verdier, Gesztesy, Weikard, and others \cite{Airault_McKean_Moser_1977, Treibich_Verdier_1992, Gesztesy_Weikard_Picard_1996,Gesztesy_Unterkofler_Karl_2006}. These works primarily focus on integer weights at special fixed locations (often half-periods) or assume completely generic configurations, emphasizing the spectral properties of the KdV hierarchy. Treibich and Verdier investigate pole collisions in the moduli of finite gap potential, where each collision cluster has triangular multiplicity and the total pole multiplicity is preserved throughout the collision process. By relaxing the finite gap property, our moduli space with pole collisions can continuously relate configurations with different total weights.

\subsubsection{Isomonodromic deformations on tori} In the realm of symplectic geometry and Painlevé equations, the study of Fuchsian equations on elliptic curves was pioneered by Okamoto, who formulated Painlevé VI as an isomonodromic deformation on a torus. This was systematically generalized by Kawai \cite{Kawai_1996}, who derived the full symplectic form (cf.~\ref{ss:Kawai_Goldman} below) governing the case of moving poles. However, Kawai's analysis operates strictly over the generic configuration space $E_\tau^r \setminus \Delta$, intentionally avoiding the profound analytic singularities that occur when poles merge. Our dimension formula of the log free variety $V_{\mathbf{n},\mathbf{p}}(\tau)$ as a complete intersection when some $n_i \not\in \tfrac12 \mathbb{N}$ (cf.~Remark \ref{r:CI}) also follows from Iwasaki's work \cite{Iwasaki_1991}. However the most interesting case in our study when $n_i \in \tfrac12 \mathbb{N}$ for all $i$ is not included.  

\subsubsection{Geometric analysis} In geometric analysis, elliptic equations directly related to the generalized Lam\'e equations arise in the study of singular Liouville (mean-field) equations  \begin{equation} \label{eq:Liouville}
\Delta u \pm e^u = 8\pi \sum_{i=1}^r n_i \delta_{p_i}.
\end{equation}
In the classical Lam\'e equations, namely the case $r = 1$, C.-S.~Lin and the second author \cite{Lin_Wang_2010, Lin_Wang_2017} initiated the systematic study of unitary monodromy conditions (the case with $+$ sign in \eqref{eq:Liouville}) via pre-modular forms. This was partially extended to the generalized Lam\'e setting by the second author in \cite{Wang_2020} and it was conjectured there the existence of curve components in the associated log-free varieties, aiming at setting up a key step toward solving the unitary monodromy problem. The present work fundamentally unifies and transcends these approaches.

\subsection{Future perspectives}

While the present work establishes the algebraic foundations for the degeneration of the GLE, this structural control opens direct pathways to global problems in the geometry of moduli spaces. 

\subsubsection{The unitary monodromy problem and non-abelian Hodge theory} 
A notoriously difficult problem in the study of Fuchsian equations is characterizing the exact locus in the accessory parameter space $(\mathbf{A}, B)$ where the monodromy representation is unitary ($SU(2)$). While the Lamé equation is naturally formulated in the algebraic de Rham realization as an oper, its continuous analytic deformations are deeply connected to the Dolbeault realization via the non-abelian Hodge (NAH) correspondence. In the Dolbeault setting, the corresponding Higgs field $\theta$ lies in the Hitchin section and takes the explicit trace-free companion matrix form:
\begin{equation} \label{eq:Higgs_field}
\theta = \begin{pmatrix} 0 & 1 \\ Q(z) & 0 \end{pmatrix} dz.
\end{equation}
The accessory parameters dictate the determinant of this Higgs field, yielding the global meromorphic quadratic differential $Q(z) dz^2 = -\det(\theta)$. 

The bridge to the monodromy character variety is established by the existence of a harmonic metric $h$ solving Hitchin's self-duality equation, $F_{h} + [\theta, \theta^{*h}] = 0$. When restricted to the specific data of the Lamé oper, the off-diagonal terms of the metric equations vanish, and the system collapses into the scalar Liouville-type PDE \eqref{eq:Liouville}. For the split real form $SL(2, \mathbb{R})$, this yields the $-$ sign where the maximum principle applies, and standard algebraic stability guarantees a unique metric. In stark contrast, the unitary $SU(2)$ monodromy corresponds to the $+$ sign. Here, the maximum principle violently fails. Solutions can concentrate and bubble, destroying uniqueness and causing the continuous parametrization map from the Dolbeault side to the character variety to fold over itself. The generalized Lamé curves constructed here provide the explicit algebraic framework necessary to track these continuous real-analytic bifurcations rigorously, supplying the structural rigidity required precisely where classical NAH stability conditions break down.

\subsubsection{Symplectic geometry and isomonodromic deformations} \label{ss:Kawai_Goldman}
The accessory parameters $(\mathbf{A}, B)$ and pole positions $\mathbf{p}$ serve as canonical coordinates in symplectic geometry. By Kawai's theorem, the Riemann--Hilbert map to the character variety is a symplectic local diffeomorphism, and the monodromy-preserving deformations of the GLE on $E_\tau$ are governed by the Kawai--Goldman 2-form. In our notation, this form is given by:
\begin{equation*} 
\omega_{KG} = 2 \sum_{i=1}^r dA_i \wedge dp_i + \frac{1}{\pi i} dB \wedge d\tau - \frac{\eta_1(\tau)}{\pi i} \sum_{i=1}^r \Big(p_i dA_i \wedge d\tau + A_i dp_i \wedge d\tau\Big).
\end{equation*}
Because our framework handles the general case allowing poles to move ($dp_i \neq 0$), the full dynamics of this symplectic structure are activated. The roots of the pre-modular forms $\Phi_{\mathbf{n},\mathbf{p};t,s}(\tau)$ introduced in this theory dictate these isomonodromic flows. Most importantly, because our summation map $\sigma_{\mathbf{n}, \mathbf{p}}$ is completely flat over the unconstrained configuration space $E_\tau^r$, it imposes rigid geometric constraints on how these Hamiltonian systems and their associated symplectic leaves degenerate or foliate when poles collide ($p_i \to p_j$).

\medskip

We plan to study both the NAH unitary aspect as well as the symplectic/isomonodromic aspect in a subsequent work.

\subsection*{Acknowledgement} We are grateful to Martin Guest for discussing aspects on non-abelian Hodge theory, and to Ching-Li Chai for calling our attention on recent applications of the classical BGG complex in Hodge conjecture.

\section{Log-free condition and ansatz solution}
Given $\mathbf{n}=(n_1,\dots,n_r)\in \mathbb{C}^r $, we consider the generalized Lam\'e equation 
\begin{equation} \label{e:GLE}
 \frac{d^2 w}{d z^2} - \Big( \sum_{i=1}^{r} n_i(n_i + 1) \wp(z - p_i) + \sum_{i=1}^{r} A_i \zeta(z - p_i) +B   \Big) w =0,
\end{equation}
on $\mathbb{C}$, where $A_i,B, p_i \in \mathbb{C}$. We assume that 
\begin{itemize}
\item[(i)] $\sum_{i=1}^r A_i=0$ so that \eqref{e:GLE} descends to an equation in $E_{\tau}$.
\item[(ii)] The points $[p_i] \in E_\tau$, for $i = 1, \dots, r$, are pairwise distinct. 
\end{itemize}
Although the potential $Q(z)$ is invariant under symmetry $n_i \mapsto -n_i-1$, we do not globally impose restrictions such as $\Re(n_i) \geq -1/2$. This flexibility is because a single equation can admit different types of generalized Hermite--Halphen ansatz solutions (Section~\ref{ss:ansatz}) depending on the chosen branch of $n_i$.

For simplicity, we will also write \eqref{e:GLE} as 
$
w'' = Qw.
$

\subsection{Log-free conditions}
Equation (\ref{e:GLE}) has local exponents $-n_i$ and $n_i + 1$ at $z = p_i$.  
For $n_i \notin \mathbb{Z}/2$ there is no logarithmic solution. For $n_i = -\frac{1}{2}$, there is always a solution with logarithmic term. For $n_i =0$, there is a solution with logarithmic term if $A_i \neq 0$. 
When $n_i \in \tfrac{1}{2}\mathbb{N}$, to derive the log-free condition, let
\[
w(z) = \sum_{k \geq 0} c_k (z-p_i)^{-n_i +k}, \quad c_0 = 1,
\]
be a solution corresponding to the exponent $-n_i$ at $p_i$. We compare the coefficients of $(z - p_i)^{-n_i + l}$, $l \ge -2$, of equation (\ref{e:GLE}) after plugging $w(z)$. Coefficients of $(z - p_i)^{-n_i - 2}$ term matches automatically. For $(z-p_i)^{-n_i - 1}$ term, we have equation
\[
c_1(-n_i +1)(-n_i) = c_1n_i(n_i + 1) + A_i
\]
and hence $c_1=-\frac{A_i}{2n_i}$.

For $(z - p_i)^{-n_i + 0}$ term, we have equation
\begin{equation}
\begin{split}
&c_2(-n_i + 2)(-n_i + 1) \\
&\quad = c_2n_i(n_i + 1) + c_1 A_i + \sum_{j = 1, \ne i}^rn_j (n_j + 1) \wp_{ij} + \sum_{j = 1, \neq i}^r A_j \zeta_{ij} + B.
\end{split}
\end{equation}
In the primitive case $n_i=\frac{1}{2}$, the $c_2$ terms cancel out and it gives rise to the following quadratic (compatibility) equation:
\begin{equation} \label{e:log-free}
A_i^2 = B + \sum_{j = 1,\neq i}^r ( A_j \zeta_{ij} + \tfrac{3}{4}\wp_{ij}).
\end{equation}
Here $\wp_{ij} = \wp(p_i-p_j)$ and $\zeta_{ij} = \zeta(p_i-p_j)$.

For general $n_i \in \tfrac{1}{2}\mathbb{N}$, we get a polynomial with degree $\ell_i + 1 = 2n_i +1$. Indeed, continuing the process we see that for the $(z - p_i)^{-n_i + k}$ term with $k \ge 0$,
\begin{equation}
\begin{split}
&c_{k + 2}(-\tfrac{\ell_i}{2} + k + 2)(-\tfrac{\ell_i}{2} + k + 1) \\
&\quad = c_{k + 2}\tfrac{\ell_i}{2}(\tfrac{\ell_i}{2} + 1) + c_{k + 1} A_i + c_k B\\
&\qquad + \sum_{m = 0}^k c_{k - m}\Big(\sum_{j \ne i} \tfrac{\ell_j}{2} (\tfrac{\ell_j}{2} + 1) \wp_{ij}^{(m)} + \sum_{j \ne i} A_j \zeta_{ij}^{(m)} \Big).
\end{split}
\end{equation}
By assigning degree 1 to $A_i$ and degree 2 to $B$, we see inductively that $c_k$ has degree $k$. The recursion leads to a compatibility equation $F_{\ell_i}(\mathbf{A}, B) = 0$ on $c_{k + 2}$ when $k = \ell_i - 1$. Hence $F_{\ell_i} (\mathbf{A}, B)$ has degree $\ell_i + 1$. 

Notice that the highest degree terms $q_{\ell_i}(A_i, B)$ depend only on $A_i$ and $B$ and the expression $q_\ell(A, B)$ is determined by the compatibility equation of the following simple recursion  
\begin{equation} \label{e:q-rec}
(k + 2)((k + 1) - \ell) \bar c_{k + 2} = A \bar c_{k + 1} + B \bar c_k, \quad k \ge -1
\end{equation}
with $\bar c_{-1} = 0$ and $\bar c_0 = 1$. Consider the generating function $f(z) = \sum_{k = 0}^\infty \bar c_k z^k$. Then \eqref{e:q-rec} is equivalent to the confluent hypergeometric ODE
\begin{equation} \label{e:CHG-f}
zf'' = \ell f' + (A + Bz) f, \qquad f(0) = 1,\quad f'(0) = -\frac{A}{\ell}.
\end{equation}
Under $u = z^{\ell/2}f$ this is brought into the canonical form
\begin{equation} \label{e:CHG}
u'' + \Big( \frac{A}{z} + B - \frac{\ell(\ell + 2)}{4z^2}\Big) u = 0.
\end{equation}
In \cite[\S 3.3]{Wang_2020}, \eqref{e:CHG} was solved using explicit Kummer solutions (cf.~\cite[Ch.XVI]{Whittaker_Watson_1996}) and a ``residue trick'', which leads to the exact expression of $q_\ell(A, B)$:

\begin{proposition} \cite[Lemma 3.3]{Wang_2020} \label{p:topterm}
The RHS of \eqref{e:q-rec} at step $k = \ell - 1$ equals
\begin{equation} \label{e:q-ell}
q_{\ell}(A, B) = \frac{(-1)^\ell}{(\ell !)^2}\prod_{j = 0}^\ell (A - (\ell - 2j)B^{1/2}).
\end{equation}
\end{proposition}

\begin{proof}
We present an alternative and shorter proof using the representation theory of $\mathfrak{sl}_2(\mathbb{C})$. Write \eqref{e:CHG-f} as $Lf = 0$ where 
\[ 
L = z \frac{d^2}{dz^2}  - \ell \frac{d}{dz} - (A + B z). 
\]
We apply a gauge transformation to isolate the finite-dimensional dynamics. Let $f(z) = e^{\lambda z} P(z)$ with $\lambda^2 = B$. Then $\widetilde{L} P(z) = 0$ for
\[ 
\widetilde{L} = e^{-\lambda z} L e^{\lambda z} = z \frac{d^2}{dz^2} + (2\lambda z - \ell) \frac{d}{dz} - (A + \ell \lambda). 
\]

Consider the standard differential realization of the Borel subalgebra of $\mathfrak{sl}_2(\mathbb{C})$ acting on polynomials:
\[
 J_- = \frac{d}{dz}, \qquad J_0 = z \frac{d}{dz} - \tfrac12 \ell. 
 \]
The invariant subspace is the representation $V_\ell = \mathrm{Sym}^\ell(\mathbb{C}^2) \cong \mathbb{C}[z]_{\le \ell}$. We express $\widetilde{L}$ in the universal enveloping algebra via these generators by
\[ 
\widetilde{L} = (J_0 - \tfrac12 \ell )  J_- + 2\lambda J_0 - A. 
\]

Since $J_0$ and $J_-$ stabilize $V_\ell$, $\widetilde{L}$ restricts to an endomorphism of $V_\ell$. In the standard basis $\{1, z, z^2, \dots, z^\ell\}$, the lowering operator $J_-$ is strictly upper-triangular, while the Cartan element $J_0$ is diagonal with spectrum $\{-\frac{\ell}{2}, -\frac{\ell}{2}+1, \dots, \frac{\ell}{2}\}$. Consequently, $\widetilde{L}$ is upper-triangular on $V_\ell$. 

The formal series $f(z)$ truncates (equivalently, the recursion \eqref{e:q-rec} encounters no obstruction at step $k=\ell$) if and only if $\widetilde{L}$ possesses a polynomial solution in $V_\ell$, which occurs precisely when the determinant of $\widetilde{L}|_{V_\ell}$ vanishes. The eigenvalues of $\widetilde{L}$ are its diagonal entries:
\[ 
d_k = 2\lambda (k - \tfrac12 \ell ) - A, \qquad 0 \le k \le \ell. 
\]
Setting $\lambda = \sqrt{B}$, the determinant is $0$ if and only if $A = (2k - \ell)\sqrt{B}$. Equivalently, $A \in \{(\ell - 2j)\sqrt{B} \mid 0 \le j \le \ell\}$. Therefore, the compatibility polynomial $q_\ell(A, B)$ must be proportional to this determinant:
\[ q_\ell(A, B) = C_\ell \prod_{j=0}^\ell \big(A - (\ell - 2j)\sqrt{B}\big). \]

To determine $C_\ell$, we consider the recursion \eqref{e:q-rec} by setting $B = 0$. The recursion simplifies to $(k+2)(k+1-\ell)\bar c_{k+2} = A \bar c_{k+1}$, yielding 
\[ \bar c_\ell = \frac{A^\ell}{\ell! (-\ell)(1-\ell)\cdots (-1)} = \frac{(-1)^\ell}{(\ell!)^2} A^\ell. \]
Since the compatibility equation at step $\ell-1$ is exactly $q_\ell(A, B) = A \bar c_\ell + B \bar c_{\ell-1} = 0$, the leading term is $A \bar c_\ell = \frac{(-1)^\ell}{(\ell!)^2} A^{\ell+1}$ hence $C_\ell = \frac{(-1)^\ell}{(\ell!)^2}$.
\end{proof}

\begin{remark}
In fact, an equivalent but more sophisticated form of the log-free conditions was studied in \cite{Wang_2020}. In that setup one first obtained a quadratic recursion and then transformed it into \eqref{e:CHG}.  
\end{remark}

\begin{definition} The variety defined by log-free conditions
\[
V_{\mathbf{n},\mathbf{p}}(\tau) := \Big\{(A_1,\dots,A_r,B)\ | \ F_{\ell_i} (\mathbf{A}, B)=0 \ \mbox{if $n_i \in \frac12\mathbb{N}$}, \sum_{i=1}^r A_i=0 \Big\}
\]
is called \emph{log-free variety}.

Consider $B$ as a variable of degree $2$ and $A_i$ of degree $1$. Define its compatification, called \emph{projective log-free stack}, in the weighted projective space:
\[
    \overline{V}_{\mathbf{n},\mathbf{p}}(\tau) \subset \mathbb{P}^{r+1}(1,\dots,1,2). 
\]
Note that the ambient space $\mathbb{P}^{r+1}(1,\dots,1,2)$ has a unique orbifold point $[0:\dots:0:1]$ with a nontrivial stabilizer group $\mathbb{Z}/2\mathbb{Z}$. We will see how the orbifold point affects the Hilbert quasi-polynomial computation in section 5.
\end{definition}

The log-free variety for the case when $\mathbf{n} \in \frac{1}{2} \mathbb{N}^r$ is of special interest. 
In \cite[Conjecture 5.7]{Wang_2020}, it was conjectured that $V_{\mathbf{n},\mathbf{p}}(\tau)$ always has curve components. It is nontrivial since it implies that log-free conditions are not complete intersections. We give a proof in section 5.

\begin{remark} \label{r:CI}
The non-complete intersection phenomenon happens only when $\mathbf{n} \in \tfrac{1}{2}\mathbb{N}^r$.  We show that in all other cases the log-free conditions are complete intersection. 

It suffices to prove that using only $\sum_i A_i=0$ and the first $r-1$ equations $F_{l_i}$, the resulting variety is a curve (without higher dimension components). The general case follows directly from the projective dimension theorem. 

Let $H$ be the infinity divisor (defined by $\{A_0=0\}$). Note that 
\[
\begin{aligned}
    &\Big\{ \sum A_i=0, F_{l_i}(\mathbf{A},B) =0, \text{ for } 1 \leq i \leq r-1 \Big\} \cap H 
    \\
    &= \left\{ [0:A_1:\dots:A_r:B] \in \mathbb{P}(1^{r+1},2) \;\middle|\; 
    \begin{aligned}
        &\sum_{i=1}^r A_i = 0, \quad
        q_{\ell_i}(A_i, B) = 0 \\
        &\qquad \qquad \text{for } 1 \le i \le r-1
    \end{aligned}
    \right\}.
\end{aligned}
\]
By Proposition~\ref{p:topterm}, it is easy to see that the intersection only consists of points and B\'ezout theorem on weighted projective space gives the exact counting:
\[
    \frac{1}{2} (2n_1+1) \dots (2n_{r-1}+1).  
\]
The points are indeed characterized by
\[
    \left\{ [0:c_1:\dots:c_r:1] \in\mathbb{P}(1^{r+1},2)\:\middle|\; \substack{\displaystyle c_r = -\sum_{i=1}^{r-1} c_i \\  \displaystyle c_i \in \{ -2n_i, -2(n_i-2), \dots, 2n_i \}  \\ \displaystyle \text{ for }1\leq i \leq r-1 } \right\} \Big/ \sim, 
\]
where $\mathbf{c} \sim - \mathbf{c}$. 

It is clear that $\dim (\overline{V}_{\mathbf{n},\mathbf{p}}(\tau)) \ge 1$. By projective dimension theorem:
\[  
    \dim (\overline{V}_{\mathbf{n},\mathbf{p}}(\tau)) \leq \dim (\overline{V}_{\mathbf{n},\mathbf{p}}(\tau) \cap H) +1 =1. 
\]
Hence $\overline{V}_{\mathbf{n},\mathbf{p},}(\tau)$ has pure dimension one and then a complete intersection.

 Iwasaki (\cite[Theorem 2.5]{Iwasaki_1991}) also observed this generic behavior, proving that the log-free variety has pure dimension $m$, where $m>0$ is the number of non-half-integral weights. 
\end{remark}

\subsection{Ansatz solution} \label{ss:ansatz}
We start with a general observation.
\begin{theorem} \label{t:double_multi_sol}
Suppose $f(z)$ is a multi-valued function on $\mathbb{C}$ satisfying the following conditions:
\begin{enumerate}
\item Quasi-periodicity: There exist constants $\mu_1, \mu_2 \in \mathbb{C}^*$ such that $f(z+\omega_k) = \mu_k f(z)$ for $k=1,2$.
\item Local behavior: The only singular or vanishing points of $f(z)$ lie in $S + \Lambda_\tau$, where $S = \{p_1, \dots, p_r\}$ is a finite set in a fundamental domain. At each $p_i \in S$, $f(z)$ admits the local form $f(z) = (z-p_i)^{\alpha_i} h_i(z)$, where $\alpha_i \in \mathbb{C}$ and $h_i(z)$ is holomorphic and non-vanishing at $p_i$.
\end{enumerate}
Then
$
\sum_{i=1}^r \alpha_i = 0,
$
and $f(z)$ is uniquely given by the formula:
\[
    f(z) = C e^{\kappa z} \prod_{i=1}^r \sigma(z-p_i)^{\alpha_i}, \quad \kappa \in \mathbb{C}.
\]
\end{theorem}
\begin{proof}
Let $g(z) = \frac{f'(z)}{f(z)}$. The quasi-periodicity of $f$ implies that $g(z)$ is an elliptic function. The local behavior of $f$ implies that $g$ has poles on $p_i$ with residue $\alpha_i$. We conclude that 
\[
    g(z) = \kappa + \sum_{i=1}^r \alpha_i \zeta(z-p_i), \quad \kappa \in \mathbb{C}.
\]
The residue theorem implies that $\sum_{i=1}^N \alpha_i =0$, and 
\[
    f(z) =  C \exp\Big\{ \int_z g(w)dw  \Big\} = C e^{\kappa z} \prod_{i=1}^r \sigma(z-p_i)^{\alpha_i}. 
\]
\end{proof}

We study the GLE with solution satisfying quasi-periodicity. It takes the following form. 

Given $\mathbf{n} \in (\mathbb{C}\setminus \{0\})^r$ with total weight $\sum_i n_i = n \in \mathbb{Z}_{ \geq  0}$, the \textit{generalized Hermite--Halphen ansatz} (GHH ansatz) of type $\mathbf{n}$ consists of the following two cases. 

For $n\in \mathbb{N}$, 
\begin{equation} \label{e:GHH}
\begin{split}
w_{\mathbf{n},\mathbf{p}; \mathbf{a}}(z) = \exp \{hz\}\frac{\displaystyle\prod_{\mu = 1}^n \sigma(z - a_\mu)}{\displaystyle\prod_{i = 1}^{r}\sigma(z - p_i)^{n_i}},
\end{split}
\end{equation}
where $\mathbf{a} = (a_\mu) \in \mathbb{C}^n$ such that the points $[a_\mu] \in E_\tau$ are pairwise distinct, $[a_\mu] \neq [p_i]$ for all $\mu, i$, and
\begin{equation} \label{e:AP}
h = \frac{1}{n}\zeta_{\mathbf{n},\mathbf{p}}(\mathbf{a}) := \frac{1}{n} \sum_{\mu = 1}^n \sum_{i = 1}^{r} n_i \zeta(a_\mu - p_i). 
\end{equation}

When the total weight $n$ is $0$, consider 
\begin{equation} \label{e:GHH_zero}
\begin{split}
w_{\mathbf{n},\mathbf{p};h}(z) = \frac{\exp\{ hz \}}{\displaystyle\prod_{i = 1}^{r}\sigma(z - p_i)^{n_i}},
\end{split}
\end{equation}
where $h \in \mathbb{C}$ is a free parameter. 

\begin{remark}
When $r=1$ and $\mathbf{n} = n \in \mathbb{N}$, the generalized Hermite--Halphen ansatz reduces to 
\[
w_{\mathbf{n}, \mathbf{p} ;  \mathbf{a} }(z) = w_{\mathbf{a}}(z)
= \exp \Big\{ z \sum_{\mu = 1}^n \zeta(a_\mu) \Big\}  
\frac{\prod_{\mu = 1}^n  \sigma(z - a_\mu)}{\sigma(z)^n},
\]
which is the classical Hermite--Halphen ansatz for Lam\'e equation.
\end{remark}

To verify that the ansatz is a valid solution candidate for equation \eqref{e:GLE}, we compute its logarithmic derivative 
\[
\frac{w'}{w} = (\log w)' = h + \sum_{\mu = 1}^n \zeta(z - a_\mu) - \sum_{i = 1}^{r} n_i \zeta(z - p_i),
\]
which is an elliptic function since $\sum_{i = 1}^r n_i = n$. Substituting this into the identity $\frac{w''}{w} = \big( \frac{w'}{w} \big)' + \big( \frac{w'}{w} \big)^2$ yields the constraint:
\begin{equation} \label{e:constraint}
\begin{split}
&\sum_{i=1}^{r} n_i (n_i + 1)\wp (z - p_i) + \sum_{i = 1}^r A_i \zeta(z - p_i) + B \\
&\qquad= \sum_{i = 1}^{r} n_i \wp(z - p_i) - \sum_{\mu = 1}^n \wp(z - a_\mu) \\
& \qquad \quad + \Big( h +
 \sum_{\mu = 1}^n \zeta(z - a_\mu) - \sum_{i = 1}^{r} n_i \zeta(z - p_i)
 \Big)^2.
\end{split}
\end{equation}
By convention, for $n=0$ empty sums over $\mathbf{a}$ vanishes identically, the value of $h$ is determined by parameters $(\mathbf{A},B)$. For $n>0$, the equation is regular at $z= a_\mu$, we have
\[
    0 = {\rm Res}_{z=a_\mu} \Big(\frac{w''}{w} \Big) = 2 \Big( h + \sum_{\nu = 1, \neq \mu}^n \zeta(a_\mu - a_\nu) - \sum_{i = 1}^{r} n_i \zeta(a_\mu - p_i) \Big).
\]
Summing over $\mu$ gives the equation 
\[
    n h - \sum_{\mu=1}^n\sum_{i=1}^r n_i \zeta(a_\mu-p_i) =0
\]
which is consistent with the definition of GHH ansatz.

The double pole at $z = p_i$ on the right-hand side has coefficient $n_i + n_i^2 = n_i(n_i + 1)$, precisely matching the left-hand side. The remaining steps to justify that the GHH ansatz is indeed a solution are to equate all the residue
terms and the constant terms.

To simplify notation, we use vector evaluation shorthand for sums over the roots $\mathbf{a}$, and retaining the index shorthand for the fixed poles $\mathbf{p}$:
\begin{equation}
\begin{split}
\zeta(p_i - \mathbf{a}) &:= \sum_{\mu = 1}^n \zeta(p_i - a_\mu), \qquad \wp(p_i - \mathbf{a}) := \sum_{\mu = 1}^n \wp(p_i - a_\mu), \\
\zeta_{ij} &:= \zeta(p_i - p_j), \qquad \wp_{ij} := \wp(p_i - p_j).
\end{split}
\end{equation}
For the weight $n=0$ case, by convention, $\zeta(p_i - \mathbf{a}) = 0$ and $\wp(p_i - \mathbf{a}) = 0$.

\begin{theorem} \label{Ansatz}
The generalized Hermite--Halphen ansatz is a solution to equation \eqref{e:GLE} for some $(\{A_i\}_{i = 1}^r, B)$ provided that, when $n>0$, the roots $\mathbf{a}$ satisfy
\begin{equation} \label{Gene-Lame-Curve}
\sum_{i = 1}^{r} \sum_{\nu = 1, \neq \mu}^n 
n_i\Big(
\zeta(a_\mu - a_\nu) - \zeta(a_\mu - p_{i}) +  \zeta(a_\nu - p_{i})
\Big) = 0
\end{equation}
for all $\mu = 1, \ldots, n$. 

For $n \geq 0$, the corresponding coefficients $A_i$ and $B$ are given by:
\begin{equation} \label{A-a}
\frac{A_i}{\ell_i} = - \zeta(p_i - \mathbf{a}) + \sum_{j = 1, \ne i}^r n_j \zeta_{ij} - h,
\end{equation}
\begin{equation} \label{B-a}
B = \frac{A_i^2}{\ell_i^2} - \sum_{j = 1, \ne i}^r n_j(n_j + \ell_i) \wp_{ij} - \sum_{j = 1, \neq i}^{r} A_j \zeta_{ij} + (\ell_i - 1) \wp(p_i - \mathbf{a}).
\end{equation}
\end{theorem}

\begin{proof}
For $n>0$, recall the residue computation:
\[
    0 = {\rm Res}_{z=a_\mu} \Big(\frac{w''}{w} \Big) = 2 \Big( h + \sum_{\nu = 1, \neq \mu}^n \zeta(a_\mu - a_\nu) - \sum_{i = 1}^{r} n_i \zeta(a_\mu - p_i) \Big).
\]
Plug in $h = \frac{1}{n}\zeta_{\mathbf{n},\mathbf{p}}(\mathbf{a})$ gives \eqref{Gene-Lame-Curve}.

For $n\geq 0$, at the fixed poles $z=p_i$, the logarithmic derivative $w'/w$ has residue $ -n_i$. Its constant term $C_i$ and linear term $D_i$ are explicitly given by:
\[
    C_i = h + \zeta(p_i - \mathbf{a}) - \sum_{j=1, \neq i}^r n_j \zeta_{ij}, \qquad
    D_i = -\wp(p_i - \mathbf{a}) + \sum_{j=1, \neq i}^r n_j \wp_{ij}.
\]

Matching the coefficients of $(z-p_i)^{-1}$ and $(z-p_i)^0$ on both sides of \eqref{e:constraint} yields, respectively:
\[
\begin{split}
    A_i &=- 2n_i C_i, \\
    B + \sum_{j\neq i} n_j(n_j + 1)\wp_{ij} + \sum_{j\neq i} A_j \zeta_{ij} &= C_i^2 + (1+2\nu_i) D_i.
\end{split}
\]
Since $\ell_i = 2n_i$, substituting $C_i$ into the first equation directly gives \eqref{A-a}. Substituting both $C_i$ and $D_i$ into the second gives \eqref{B-a}. By our empty-sum convention, this coefficient matching holds uniformly for all $n \ge 0$.
\end{proof}

It is clear that $w_{\mathbf{n}, \mathbf{p}; \mathbf{a}}$ is log-free. If $n_i>0$, the expansion of $w_{\mathbf{n}, \mathbf{p}; \mathbf{a}}$ at $p_i$ starts with the lower exponent $c_i (z-p_i)^{-n_i} + \cdots$  which ensures that there is no log solution at $p_i$. Hence for $\mathbf{n} \in (\tfrac{1}{2}\mathbb{N})^r$, the existence of $w_{\mathbf{n}, \mathbf{p}; \mathbf{a}}$ gives the log-free conditions for all $p_i$. The global constraints $\sum_i A_i=0$ is automatic by equation~\eqref{A-a}. All together give the log-free conditions of equation~\eqref{e:GLE}.

We have the converse statement. It will be crucial for the later study of log-free variety. 

\begin{proposition} \label{exist_ansatz_sol}
If the generalized Lam\'e equation $w'' = Qw$ as in \eqref{e:GLE} with $\mathbf{n} \in (\tfrac{1}{2} \mathbb{N})^r$ and total weight $n \in \mathbb{N}$ has only log-free solutions, then one of its solutions $w$ is given by the generalized Hermite--Halphen ansatz, possibly with lower total weight.
\end{proposition}

\begin{proof}
Let $P(z) := \sigma(z)^n / \prod_{i=1}^r \sigma(z-p_i)^{n_i}$ (which may be multi-valued if some $n_i$ are half-integers) and $s := P'/P$. Writing $w = Pu$ and substituting into $w'' = Qw$ yields the equation
\[
Lu := u'' + 2s u' + (s' + s^2 - Q)u = 0,
\]
which has elliptic coefficients. Since the original equation is log-free, $Lu = 0$ also has only log-free solutions. 
The local exponents of $u$ at $p_i$ are $0$ and $2n_i+1 \in \mathbb{N}$. By a standard argument (cf.~\cite[p.376]{Ince_1956}), this guarantees a quasi-periodic solution $u$. In the decomposition $w = Pu$, all principal parts at $p_i$ are absorbed by $P$, hence the only poles of $u$ are at the origin $[0] \in E_\tau$ with order $n$. 

Consequently, $u$ has $n$ zeros $\{[a_\mu]\}_{\mu=1}^n$ in $E_\tau$. Because $u$ is a solution to a second-order ODE, any zeros such that $[a_\mu] \neq [p_i]$ must be pairwise distinct in $E_\tau$. The logarithmic derivative of $u$ must take the form
\[
\frac{u'}{u} = -n \zeta(z) + \sum_{\mu = 1}^n \zeta(z - a_\mu) + h
\]
for some constant $h$. Equivalently, $u(z) = e^{hz} \prod_{\mu = 1}^n \sigma (z - a_\mu)/\sigma(z)^n$.

If a zero coincides with a fixed pole $[p_i]$, the local exponent requires $(2n_i+1)$ zeros to coincide there. We define a new weight vector $\mathbf{n}'$ with total weight $n'$: $n'_i = -n_i-1$ if zeros coincide at $[p_i]$, and $n'_i = n_i$ otherwise. Let $\mathbf{a}'$ denote the remaining zeros.

As derived previously from \eqref{e:constraint}, summing the residues forces the relation $n'h = \zeta_{\mathbf{n}',\mathbf{p}}(\mathbf{a}')$. For $n'>0$, this gives $h = \frac{1}{n'}\zeta_{\mathbf{n}',\mathbf{p}}(\mathbf{a}')$. For $n'=0$, the constant $h$ is unconstrained. Both cases are included in the definition of GHH ansatz.

With $u$ and $h$ as determined above, $w=Pu$ exactly matches the GHH ansatz of weight $\mathbf{n}'$ and roots $\mathbf{a}'$. This finishes the proof.
\end{proof}

\section{Geometry of generalized Lam\'e curves}
\subsection{Underlying generalized Lam\'e curves}
Let $\mathbf{n} \in (\mathbb{C} \setminus \{0\})^r$ with total weight $n \in \mathbb N$.
From the residue computation of Theorem~\ref{Ansatz}, we define the quasi-projective variety:
\[
\begin{split}
&Y_{\mathbf{n},\mathbf{p}}(\tau) := \Big\{  
\{[\mathbf{a}] \} \in {\rm Sym}^n E_\tau \Big|, 
\\
&\sum_{i = 1}^{r} \sum_{\nu = 1, \neq \mu}^n 
n_i\Big(
\zeta(a_\mu - a_\nu) - \zeta(a_\mu - p_{i}) +  \zeta(a_\nu - p_{i})
\Big) = 0, \text{ for }\mu=1,\dots,n \Big\}.
\end{split}
\]
We call its scheme-theoretic closure in $\Sym^n E_\tau$, denoted $\overline{Y}_{\mathbf{n},\mathbf{p}}(\tau)$, the underlying generalized Lam\'e curve of type $\mathbf{n}$.

\begin{theorem} \label{t:GLC}
$Y_{\mathbf{n},\mathbf{p}}(\tau)$ and $\overline{Y}_{\mathbf{n},\mathbf{p}}(\tau)$ are unions of finite number of curves. 
\end{theorem}

\begin{proof}
It suffices to prove the result for $Y_{\mathbf{n},\mathbf{p}}(\tau)$, since taking closure preserve the curveness property.
The defining equations of $Y_{\mathbf{n},\mathbf{p}}(\tau)$ sum to zero, implying any irreducible component has dimension $\geq 1$. It suffices to show equality holds. 

Using the identity $\zeta(v-w) - \zeta(v) + \zeta(w) = \frac{1}{2}\frac{\wp'(v) + \wp'(w)}{\wp(v)-\wp(w)}$, we rewrite the defining equations as
\begin{equation} \label{e:Lame_alg_form}
    \sum_{i=1}^{r} \mathbf{A}_i \mathbf{y}_i = \mathbf{0},
\end{equation}
where $\mathbf{y}_i := (y_{1,i}, \dots, y_{n,i})^T$ with $x_{\mu,i} := \wp(a_\mu - p_i)$ and $y_{\mu,i} := \wp'(a_\mu - p_i)$, and $\mathbf{A}_i$ is the $n \times n$ matrix defined as:
\[
\mathbf{A}_i := n_i
\left(
\begin{matrix}
\sum_{k=2}^n \frac{1}{x_{1,i} - x_{k,i}}  & \frac{1}{x_{1,i}-x_{2,i}} & \cdots & \frac{1}{x_{1,i}-x_{n,i}} \\
\frac{1}{x_{2,i}-x_{1,i}} & \sum_{k \neq 2} \frac{1}{x_{2,i} - x_{k,i}} & \cdots & \frac{1}{x_{2,i}-x_{n,i}} \\
\vdots & \vdots & \ddots & \vdots \\
\frac{1}{x_{n,i}-x_{1,i}} & \frac{1}{x_{n,i}-x_{2,i}} & \cdots & \sum_{k=1}^{n-1} \frac{1}{x_{n,i} - x_{k,i}}
\end{matrix}
\right).
\]
Since the entries of $\mathbf{A}_i$ contain fractions of the form $\frac{1}{x_{\mu,i} - x_{\nu,i}}$, the formulation becomes undefined if any $x_{\mu,i} = x_{\nu,i}$. We analyze the generic case and the degenerate case separately.

\vspace{0.5em}
\textbf{The Generic case ($x_{\mu, i} \neq x_{\nu, i}$ for all $\mu \neq \nu$).}
It is a known result that $\mathrm{rank}(\mathbf{A}_i) = n-1$ (see \cite[Proof of Proposition 5.8.3]{Chai_Lin_Wang_2015}). We provide a brief alternate proof whose techniques extend to the degenerate case. 

We claim the left null space 
\[
v\mathbf{A}_i = \mathbf{0} \Leftrightarrow \sum_{\nu=1,\neq \mu}^n \frac{v_\mu - v_\nu}{x_{\mu,i} - x_{\nu,i}} =0, \text{ for }\mu=1,\dots,n.
\]
is spanned by $\mathbf{1} = (1,\dots,1)$. Let $f(x)$ be the unique polynomial of degree $\leq n-1$ with $f(x_{\mu,i}) = v_\mu$. The condition $v\mathbf{A}_i = \mathbf{0}$ implies $(Lf)(x_{\mu,i}) = f'(x_{\mu,i})$ for all $\mu$, where
\[
    (Lf)(x) := \sum_{\nu=1}^n \frac{f(x) - f(x_{\nu,i})}{x - x_{\nu,i}}.
\]
Since $(Lf)(x) - f'(x)$ is a polynomial of degree $\leq n-2$ that evaluates to zero at $n$ distinct points, it is identically zero. This shows $f(x)$ is constant, so $v = c\mathbf{1}$, yielding $\mathrm{rank}(\mathbf{A}_i) = n-1$.

We now count the dimension of the solution space. The entire configuration is parameterized by $n+r-1$ independent degrees of freedom (e.g., $x_{1,1}, \dots, x_{1,r}$ and $x_{2,1}, \dots, x_{n,1}$) via the elliptic curve addition laws. Fixing $\mathbf{p}$ subject to $\sum n_i p_i = 0$ gives $r-1$ constraints. Since the matrix imposes $n-1$ independent constraints, the remaining dimension is 
\[
(n+r-1) - (r-1) - (n-1) = 1.
\]

\vspace{0.5em}
\textbf{The Degenerate case ($x_{\mu, i} = x_{\nu, i}$ for some pairs).}
Assume $k$ pairs coincide: $x_{\mu_l,i} = x_{\nu_l,i}$ for $l=1,\dots,k$. The formulation in \eqref{e:Lame_alg_form} is modified to a reduced system
\[
    \mathbf{B}_i \mathbf{y}_i = \mathbf{c},
\]
where $\mathbf{B}_i$ is obtained from $\mathbf{A}_i$ by setting $x_{\mu_l,i} = x_{\nu_l,i}$ and forcing the undefined terms $\frac{1}{x_{\mu_l,i} - x_{\nu_l,i}}$ to zero. The vector $\mathbf{y}_i$ is subject to $k$ constraints $y_{\mu_l,i} = -y_{\nu_l,i}$, and the constant vector $\mathbf{c}$ has entries where $c_{\mu_l} = -c_{\nu_l} = \zeta(a_{\mu_l}-a_{\nu_l}) - \zeta(a_{\mu_l} - p_i) + \zeta(a_{\nu_l} - p_i)$. Because $\mathbf{y}_i$ already has $k$ constraints, it suffices to prove $\mathrm{rank}(\mathbf{B}_i) \geq n-1-k$.

Let $\mathbf{B} = \mathbf{B}_i$ with $x_\mu = x_{\mu,i}$, and define the index set $I := \{\mu_1, \nu_1, \dots, \mu_k, \nu_k\}$. The left null space condition $v\mathbf{B} = \mathbf{0}$ decouples into:
\[
\begin{split}
    &\sum_{\nu \notin I} \frac{v_\mu-v_\nu}{x_\mu-x_\nu} + \sum_{l=1}^k \Big( \frac{v_\mu - v_{\mu_l}}{x_\mu - x_{\mu_l}} + \frac{v_\mu - v_{\nu_l}}{x_\mu - x_{\nu_l}} \Big) = 0, \quad \text{for } \mu \notin I, \\
    &\sum_{\nu \neq \mu_l, \nu_l} \frac{v_{\mu_l} - v_\nu}{x_{\mu_l} - x_\nu} = 0, \quad \sum_{\nu \neq \mu_l, \nu_l} \frac{v_{\nu_l} - v_\nu}{x_{\nu_l} - x_\nu} = 0, \quad \text{for } l=1,\dots,k.
\end{split}
\]
Subtracting the $\mu_l$-equation from the $\nu_l$-equation gives $k$ independent constraints:
\[
    S_l \cdot (v_{\mu_l} - v_{\nu_l}) := \Big( \sum_{\nu \neq \mu_l, \nu_l} \frac{1}{x_{\mu_l} - x_\nu} \Big) (v_{\mu_l} - v_{\nu_l}) = 0.
\]
For the remaining $n-k$ equations, define new variables:
\[
    u_\mu := v_\mu \text{ for } \mu \notin I, \quad u_{\mu_l} := \frac{1}{2}(v_{\mu_l} + v_{\nu_l}) \text{ for } l=1,\dots,k.
\]
The system reduces to $\sideset{}{'}\sum_{\nu \notin I} \frac{u_\mu - u_\nu}{x_\mu - x_\nu} + \sideset{}{'}\sum_{l=1}^k 2 \frac{u_\mu - u_{\mu_l}}{x_\mu - x_{\mu_l}} = 0$ for all valid $\mu$. 

We obtain the conclusion that $u = c(1,\dots,1)$ as a special case of a more general fact. Let $z_\mu \in \mathbb{C}$ be $N$ distinct points and $w_\mu$ be nonzero weights satisfying $\sum w_\mu \neq 0$. If a vector $u$ satisfies the system
\[    \sum_{\nu \neq \mu} w_\nu \frac{u_\mu - u_\nu}{z_\mu - z_\nu} = 0, \quad \text{for } \mu=1,\dots,N,
\]
then $u = c(1,\dots,1)$.

Define the rational function
\[
    R(z) := \frac{ \sum_{\nu=1}^{N} \frac{w_\nu u_\nu}{z-z_\nu} }{ \sum_{\nu=1}^N \frac{w_\nu}{z-z_\nu} } = \frac{N(z)}{D(z)}.
\]
Since $\sum w_\mu \neq 0$, $D(z)$ has degree $N-1$, while $N(z)$ has degree $\leq N-1$. Its derivative $R'(z) = \frac{N'D - ND'}{D^2}$ satisfies three conditions:
\begin{enumerate}
    \item $R'(z_\mu) = 0$ for $\mu=1,\dots,N$, which is equivalent to the system above.
    \item If $D(z)$ has a zero of order $r_\mu$ at $d_\mu$, $R'(z)$ has a pole of order $r_\mu+1$. Thus, the numerator $N'D - ND'$ has a zero of order $r_\mu - 1$ at $d_\mu$.
    \item $\mathrm{Res}_{z=d_\mu} R'(z) = 0$, otherwise $R(z)$ would contain a logarithmic term $\ln(z-d_\mu)$.
\end{enumerate}
The numerator $N'D - ND'$ has degree $\leq 2N-4$ (the leading coefficient of $z^{2N-3}$ cancel). The three conditions listed above impose $2N-1$ independent root and multiplicity constraints. This shows $R'(z) \equiv 0$ and therefore $u_\mu = c$.

In summary, any vector in the left kernel of $\mathbf{B}$ takes the form
\[
    v = c(1, \dots, \overset{\mu_l}{1+c_l}, \dots, \overset{\nu_l}{1-c_l}, \dots, 1).
\]
In other words, $(1,\dots,1)$ is in the left kernel, and any extra degrees of freedom depend on whether $S_l$ is zero or not. This implies $\mathrm{rank}(\mathbf{B}_i) \geq n-1-k$. We prove the theorem.
\end{proof}

Note that $\mathbf{A}_i$, and $\mathbf{B}_i$ are matrices given by the system, not the Jacobian. It happens that $\overline{Y}_{\mathbf{n},\mathbf{p}}(\tau)$ is singular and/or has non-reduced structure. In those cases, the Jacobian will drop rank.
\begin{example} \label{eg:GLC_half_periods}
Let $\mathbf{n} = (\frac{1}{2}, \frac{1}{2}, \frac{1}{2}, \frac{1}{2})$ and $\mathbf{p} = (0, \frac{w_1}{2}, \frac{w_2}{2}, \frac{w_3}{2} )$. We have
\[
    \overline{Y}_{\mathbf{n}, \mathbf{p}}(\tau) = 2\Delta_0 \cup \Delta_1 \cup \Delta_2 \cup \Delta_3
\]
where for $i=1$, $2$, $3$,
\[
    \Delta_i = \Big\{ \{[\mathbf{a}] \} \in \Sym^2 E_{\tau} \Big| a_1 \equiv a_2 + \frac{w_i}{2} \Big\},
\]
and
\[
    \Delta_0 =  \Big\{ \{[\mathbf{a}] \}\in \Sym^2 E_{\tau} \Big| a_1 \equiv -a_2 \Big\}
\]
is the component with multiplicity 2. 
\end{example}

\subsection{Generalized Lam\'e curves}
Let $\mathbf{n} \in (\mathbb{C} \setminus \{0\})^r$ with total weight $n \in \mathbb Z_{\geq 0}$. We study the space with full ansatz data $\mathbf{a}$, and $h$. For $n>0$, we define 
\[
\begin{split}
&\mathcal{Y}_{\mathbf{n},\mathbf{p}}(\tau) := \Big\{  
[(\mathbf{a},h)] \in {\rm Sym}^n \mathbb C\times \mathbb C/\sim \,\Big| \, h=\frac{1}{n}\zeta_{\mathbf{n},\mathbf{p}} (\mathbf{a}),
\\
&\sum_{i = 1}^{r} \sum_{\substack{\nu = 1 \\ \nu \neq \mu}}^n 
n_i\Big(
\zeta(a_\mu - a_\nu) - \zeta(a_\mu - p_{i}) +  \zeta(a_\nu - p_{i})
\Big) = 0,  
\text{ for }\mu=1,\dots,n
\Big\},
\end{split}
\]
where $\sim$ is the equivalence generated by
\[
(a_1,\dots,a_i,\dots,a_n,h)\sim (a_1,\dots,a_i+\omega,\dots,a_n,h+\eta(\omega)), \text{ for }\omega \in \Lambda_\tau,
\]
where $\eta(\omega) := m_1 \eta_1 + m_2 \eta_2$ for $\omega = m_1 \omega_1 + m_2 \omega_2$.

For $n=0$, we simply define $\mathcal{Y}_{\mathbf{n},\mathbf{p}}(\tau) := \{  h \in \mathbb{C}  \}.$

\begin{remark}
Given $\mathbf{n},\mathbf{p}$ with total weight $n>0$, $h$ is uniquely determined by $\{[\mathbf{a}]\}$. The projection of $\mathcal{Y}_{\mathbf{n},\mathbf{p}}(\tau)$ to $\Sym^n E_\tau \cong \Sym^n(\mathbb{C} / \Lambda_\tau)$ is isomorphic to $Y_{\mathbf{n},\mathbf{p}}(\tau)$. 
When $n=0$, there is no constraint on $h$. 

A GLE could have different types of generalized Lam\'e curve on it. For each $i$, we have a freedom of choosing $n_i$ or $-n_i-1$. The only condition is that the total weight $n \geq 0$.
\end{remark}
\begin{remark}[Comparison between complex analytic and algebraic category]
We briefly recall the complex analytic and algebraic constructions of the universal vector extension (see \cite[Examples~2.1 and 2.2]{Fonseca_Mattes_2025} for more details). 
In the complex analytic category, the universal vector extension is conveniently describe as:
\[
    E^{\natural, {\rm an}} := \mathbb{C} \times \mathbb{C} / \sim, \text{ where } (z,h) \sim (z+\omega, h+\eta(\omega)) \text{ for }\omega \in \Lambda, 
\]
In algebraic category, it is constructed by 
\[
    0 \rightarrow H^1(E, \mathcal{O}_E)^{\vee} \xrightarrow{e} E^{\natural} \rightarrow E \rightarrow 0. 
\]
The class $e$ is chosen as the identity of the extension space
\[
    {\rm Ext}^1(E, H^1(E, \mathcal{O}_E)^{\vee}  ) \cong {\rm Hom}_{\mathbb{C}} ( H^1(E, \mathcal{O}_E), H^1(E, \mathcal{O}_E)  ). 
\]
To make the relation precise, consider two affine open sets $U_0 = E \setminus \{O\}$ and $U_1 = E \setminus \{P\}$ (where $P=(x_P,y_P)$ with $y_P \neq 0$) trivializing $E^{\natural}$. With fiber coordinates $t_0$ and $t_1$, the bundle is glued over $U_0 \cap U_1$ via the non-trivial \v{C}ech 1-cocycle:
\[
    t_0 - t_1 = \frac{1}{2} \frac{y + y_P}{x - x_P}.
\]
This rational function has simple poles exactly at $O$ and $P$, explicitly representing the canonical extension class in $H^1(E, \mathcal{O}_E)$. The isomorphism between $E^{\natural,{\rm an}}$ and $E^{\natural}$ is precisely:
\[
\begin{split}
    E^{\natural, {\rm an}}|_{U_0} &\rightarrow E^{\natural}|_{U_0}, \quad (z, h) \mapsto (x, y, t_0) := \Big( \wp(z), \wp'(z), h - \zeta(z) \Big)
    \\
    E^{\natural, {\rm an}}|_{U_1} &\rightarrow E^{\natural}|_{U_1}, \quad (z, h) \mapsto (x, y, t_1) := \Big( \wp(z), \wp'(z), h - \zeta(z-p) - \zeta(p) \Big),
\end{split}
\]
where $p$ is the complex coordinate of $P$. The transition between these analytic trivializations is governed by the Weierstrass addition formula for the $\zeta$-function, which is compatible with the algebraic \v{C}ech cocycle:
\[
    t_0 - t_1 = \zeta(z-p) + \zeta(p) - \zeta(z) = \frac{1}{2} \frac{\wp'(z) + \wp'(p)}{\wp(z) - \wp(p)} = \frac{1}{2} \frac{y + y_P}{x - x_P}.
\]
As an example, the analytic locus $\Gamma:=\{ (z, \zeta(z))| z \neq 0  \} \subset E^{\natural,{\rm an}}$ is simply the locus $\{t_0=0\} \subset E^{\natural} $ on the algebraic side.
\end{remark}
\vspace{8pt} 
While the definition of $\mathcal{Y}_{\mathbf{n},\mathbf{p}}(\tau)$ is naturally in the complex analytic setting, the ambient space and the locus $\mathcal{Y}_{\mathbf{n}}$ itself can be defined algebraically. It is equivalent to work in the algebraic category.

For $n=1$, it is clear from the construction that $\mathcal{Y}_{\mathbf{n},\mathbf{p}}(\tau) \subset E^{\natural}$. 
For general $n$, consider the addition map 
\[
\begin{split}
\sigma_{\mathbf{n},\mathbf{p}} : \Sym^n E_\tau &\rightarrow E_\tau
\\
\{[\mathbf{a}]\} &\rightarrow \Big[\sum_{\mu=1}^n a_{\mu} - \sum_{i=1}^r n_ip_i \Big].
\end{split}
\]
Define $\mathcal{E}_n^{\rm Sym} := \sigma_{\mathbf{n},\mathbf{p}}^*(E^{\natural})$. We similarly have
\[
    \mathcal{Y}_{\mathbf{n},\mathbf{p}}(\tau) \subset \mathcal{E}_n^{\rm Sym}.
\]

To compactify $\mathcal{Y}_{\mathbf{n}}(\mathbf{p};\tau)$, we first consider the projective compactification of $\mathcal{E}_n^{\Sym}$. 
For $n=1$, consider $\mathcal{P}_1$ := $\mathbb{P}(F_2)$, where $F_2$, first studied by Atiyah, is the unique non-trivial indecomposable vector bundle of rank 2 and degree 0 over $E$. It is constructed by the nontrivial extension: 
\[
    0 \rightarrow \mathcal O_E \rightarrow F_2 \rightarrow \mathcal O_E \rightarrow 0.  
\]
$E^{\natural}$ is canonically identified with the open subscheme $\mathcal{P}_1 \setminus \mathbb{P}(\mathcal{O}_E)$.

For general $n$, consider $\mathcal{P}_n := \mathbb{P} ( \sigma_{\mathbf{n},\mathbf{p}}^* F_2 ) $. We have
\[
\mathcal{E}_n^{\rm Sym} \cong \mathcal{P}_n \setminus \mathbb{P}(\mathcal{O}_{{\rm Sym}^n E}).
\]
Define the generalized Lam\'e curve $\overline{\mathcal{Y}}_{\mathbf{n},\mathbf{p}}(\tau)$ the scheme-theoretic closure of $\mathcal{Y}_{\mathbf{n},\mathbf{p}}(\tau)$ in $\mathcal{P}_n$.

\begin{remark}
For the reader's convenience, we briefly recall the construction of the projective compactification of affine bundle. 

Let $\mathcal{A}$ be an affine bundle over $X$ modeled on a vector bundle $\mathcal{V}$. Precisely, if the transition function of $\mathcal{A}$ is
\[
\begin{split}
    U_\alpha \times \mathbb{A}^n &\rightarrow U_\beta \times \mathbb{A}^n
    \\
    (x,v) & \rightarrow (x, f_{\alpha\beta}(x)v + g_{\alpha\beta}(x) ),
\end{split}
\]
the transition function for $\mathcal{V}$ is accordingly defined as
\[
\begin{split}
    U_\alpha \times \mathbb{A}^n &\rightarrow U_\beta \times \mathbb{A}^n
    \\
    (x,v) & \rightarrow (x, f_{\alpha\beta}(x)v ).
\end{split}
\]

From the isomorphism $H^1(X, \mathcal{V}) \cong {\rm Ext}^1(\mathcal{O}_X, \mathcal{V})$, the cocycle condition on $\mathcal{A}$ gives the extension:
\[
    0 \rightarrow \mathcal{V} \rightarrow \mathcal{E} \rightarrow \mathcal{O}_X \rightarrow 0.
\]
The bundle $\mathcal{A}$ is canonically isomorphic to $\mathbb{P}(\mathcal{E}) \setminus \mathbb{P}(\mathcal{V})$. Specifically, this complement parameterizes lines in $\mathcal{E}$ that project non-trivially onto $\mathcal{O}_X$. Normalizing this projection to $1 \in \mathcal{O}_X$ identifies this locus with an affine hyperplane in each fiber where the transition maps act precisely via the translational shifts of $\mathcal{A}$.
\end{remark}

In general, the natural morphism by the projection:
\[
\begin{split}
    \pi: \overline{\mathcal{Y}}_{\mathbf{n},\mathbf{p}}(\tau) &\rightarrow \overline{Y}_{\mathbf{n},\mathbf{p}}(\tau)
    \\
    [ (\mathbf{a},h) ] & \rightarrow \{ [\mathbf{a}] \}
\end{split}
\]
is an isomorphism on the interior but only a finite morphism on the boundary. This explains why introducing $h$ is essential.
\begin{example} \label{ex:primitive_two_h}
Let $\mathbf{n} = ( \frac{1}{2}, \frac{1}{2} ,\frac{1}{2}, \frac{1}{2})$ and $\mathbf{p} = (p_1,\dots,p_4)$ with $p_i \neq p_j$. 
Let
\begin{align*}
    a_1 &= p_1 + \epsilon \\
    a_2 &= p_1 + \sum_{i \ge 1} \alpha_i \epsilon^i
\end{align*}
with $|\epsilon| \ll 1$. Plug $(a_1, a_2)$ into the defining equation of generalized Lam\'e curve and expand the zeta function as follows:
\begin{align*}
    \zeta(\epsilon) &= \frac{1}{\epsilon} - \frac{g_2}{60}\epsilon^3 - \frac{g_3}{140}\epsilon^5 - \dots \\
    \zeta(p+\epsilon) &= \zeta(p) - \wp(p)\epsilon - \frac{\wp'(p)}{2}\epsilon^2 - \dots
\end{align*}

Compare the coefficient of $\frac{1}{\epsilon}$:
\begin{equation*}
    \frac{4}{1-\alpha_1} - 1 + \frac{1}{\alpha_1} = 0 \implies \alpha_1 = -1, \text{ with multiplicity 2}.
\end{equation*}

Compare the coefficient of $\epsilon^0$:
\begin{equation*}
    \frac{4\alpha_2}{(1-\alpha_1)^2} - \frac{\alpha_2}{\alpha_1^2} = 0,
\end{equation*}
which is automatically true.

Compare the coefficient of $\epsilon^1$:
\begin{equation*}
    \left( \frac{4\alpha_2^2}{(1-\alpha_1)^3} + \frac{4\alpha_3}{(1-\alpha_1)^2} \right) - (-\wp_1) + \left( \frac{\alpha_2^2}{\alpha_1^3} - \frac{\alpha_3}{\alpha_1^2} - \alpha_1\wp_1 \right) = 0,
\end{equation*}
which gives $\alpha_2^2 = 4\wp_1$ and hence $\alpha_2 = \pm 2\sqrt{\wp_1}$ (where $\wp_1 := \sum_{i=2}^4 \wp(p_1 - p_i)$).

Compute $h$ in the limit of two branches $(a_1, a_2) = (p_1 + \epsilon, p_1 - \epsilon \pm 2\sqrt{\wp_1}\epsilon^2 + \dots)$:
\[
    h = \lim_{\epsilon \rightarrow 0} \frac{1}{2} \sum_{i=1}^4\sum_{\mu=1}^2\zeta(a_\mu - p_i) = \frac{1}{2} \Big( \sum_{j=2}^4 \zeta_{ij} \pm \sqrt{\wp_1} \Big). 
\]
Same computation works for all $i$. We conclude that $\overline{\mathcal{Y}}_{\mathbf{n},\mathbf{p}}(\tau) |_{\pi^{-1}(p_i,p_i)} \rightarrow \overline{Y}_{\mathbf{n},\mathbf{p}}(\tau) |_ {(p_i,p_i)}$ is two-to-one. 
\end{example}

From equation~\eqref{A-a} and equation~\eqref{B-a}, $A_i$ and $B$ are expressed in terms of $\mathbf{a}$ and $h$. The morphism
\[
\begin{split}
     \mathcal{Y}_{\mathbf{n},\mathbf{p}}(\tau) & \rightarrow V_{\mathbf{n},\mathbf{p}}(\tau) 
\end{split}
\]
can be extended to the boundary:
\[
    \overline{\mathcal{Y}}_{\mathbf{n},\mathbf{p}}(\tau) \rightarrow \overline{V}_{\mathbf{n},\mathbf{p}}(\tau). 
\]
\subsection{The boundary $\overline{\mathcal{Y}} \setminus \mathcal{Y}$}
In this subsection, we study the boundary of generalized Lame curves. The boundary consists of points $[(\mathbf{a}, h)]$ when $a_\mu \rightarrow p_i$ for some $\mu$ and $i$. Our study is based on considering the parametrization:
\[
    a_\mu = p_i + \sum_{j \geq 1} \alpha_{\mu j} \epsilon^j,
\]
and solving $\{ \alpha_{\mu j} \}$ inductively. Solving the principal part $\epsilon^{-1}$ is already quite involved. We start with some preparations (readers primarily interested in the applications may skip directly to Proposition~\ref{prop_no_mixed}).
\begin{lemma} \label{Lemma_sol_x} The system
\[
    \sum_{\nu=1,\neq \mu}^k \frac{1}{\alpha_\mu - \alpha_\nu} - \frac{x}{\alpha_{\mu}}=0, \ \mu=1,\dots,k
\]
has no solution for $x \in \mathbb{C} \setminus \{ \frac{k-1}{2} \}$. For $x = \frac{k-1}{2}$, the solution is unique up to permutation and scaling and takes the form
\[
    [ \alpha_1,\dots,\alpha_k ] = [ M\zeta^0, M\zeta^1,\dots,M\zeta^{k-1} ],
\]
where $M \neq 0$, and $\zeta$ is a primitive $k$-th roots of unity.
\end{lemma}
\begin{proof}
Assume a solution $\{ \alpha_1,\dots,\alpha_k \}$ exists,
let $p(z) = \prod_{\mu=1}^k (z- \alpha_\mu)$ be a degree $k$ polynomial with roots given by $\{ \alpha_\mu \}$. Then
\[
    \frac{p''(\alpha_\mu)}{2p'(\alpha_\mu)} = \sum_{\nu=1,\neq \mu}^k \frac{1}{\alpha_\mu - \alpha_\nu}.
\]
The system can be rewritten as
\[
    \frac{p''(z)}{2p'(z)} - \frac{x}{z}  = 0, \text{ for }z=\alpha_1,\dots,\alpha_k.
\]
We claim that $p(z)$ satisfies the following differential equation
\[
    \Big(z \frac{d^2}{dz^2} - 2x \frac{d}{dz} \Big) p(z) = 0.
\]
The left hand side equals zero when $z = \alpha_1,\dots,\alpha_k$, and hence can be written as $Q(z)p(z)$. On the other hand, it has degree $\leq k-1$. We conclude that $Q(z)=0$. 

For the existence of degree $k$ polynomial solution, we have $x = \frac{k-1}{2}$ and the solution takes the following form 
\[
    p(z) = C_1 z^k - C_2. 
\]
This proves the lemma.
\end{proof}

\begin{lemma} \label{lemma_boundary}
The system of equations
\begin{equation} \label{equation_boundary}
    f_{\mu}(\alpha) := \sum_{\nu=1, \neq \mu}^k \frac{1}{\alpha_{\mu} - \alpha_{\nu}} - \frac{(k-1)x}{\alpha_{\mu}} + \sum_{\nu=1, \neq \mu}^k \frac{x}{\alpha_{\nu}} =0, \qquad \mbox{for $\mu =1 ,\dots,k$}
\end{equation}
has a unique solution (up to permutation and scaling) with $\alpha_\mu \neq \alpha_\nu$ for $x \in \mathbb{C}\setminus\{ 0, \frac{1}{2k}, \dots, \frac{k-2}{2k} \}$. 

Moreover, the Jacobian of $f$ evaluated at the corresponding solution, $J_f(\alpha)$,
\[
\begin{split}
    J_f(\alpha)_{\mu\mu} &= \sum_{\nu=1,\neq \mu}^k \frac{-1}{(\alpha_\mu-\alpha_\nu)^2} + \frac{(k-1)x}{\alpha_\mu^2}, 
    \\
    J_f(\alpha)_{\mu\nu} &= \frac{1}{(\alpha_\mu - \alpha_\nu)^2} - \frac{x}{\alpha_\nu^2} \text{ for } \mu \neq \nu,
\end{split}
\]
has rank $k-1$ for $x \in \mathbb{C}\setminus\{0, \frac{1}{2k}, \dots, \frac{k-1}{2k}\}$, and rank $k-2$ for $x = \frac{k-1}{2k}$. 
\end{lemma}
\begin{proof}
As in the proof of Lemma~\ref{Lemma_sol_x}, we reduce the problem to finding polynomial solutions of an ODE. 

Let $\lambda := \sum_{\nu=1}^k \frac{1}{\alpha_\nu}$. Rewrite the system as
\[
    \sum_{\nu=1, \neq \mu}^k \frac{1}{\alpha_{\mu} - \alpha_{\nu}} - \frac{kx}{\alpha_{\mu}} = -x \lambda, \qquad \mbox{for $\mu =1 ,\dots,k$}.
\]
Consider the following two cases:

\noindent 1. $\lambda =0$. 
It has been studied in Lemma~\ref{Lemma_sol_x}, which has a unique solution (up to permutation and scaling) if and only if $x = \frac{k-1}{2k}$.

\noindent 2. $\lambda \neq 0$. Up to scaling, $\lambda$ can be any nonzero constant.
Assume that a solution $\{ \alpha_1,\dots,\alpha_k \}$ exists with $\alpha_\mu\neq \alpha_\nu$.
Let $p(z) = \prod_{\mu=1}^k (z- \alpha_\mu)$. The system can be rewritten as
\[
    \frac{p''(z)}{2p'(z)} - \frac{kx}{z} = -x \lambda, \text{ for }z=\alpha_1,\dots,\alpha_k,
\]
and hence $p(z)$ satisfies the following equation 
\begin{equation} \label{eqn_sol_x}
    \mathcal{L}[p] := \Big( z \frac{d^2}{dz^2} + (2\lambda x z-2kx) \frac{d}{dz} - 2k\lambda x \Big) p(z) =0.
\end{equation}
The constant term $-2k\lambda x$ is chosen to have degree $k$ polynomial solutions. 

The solution takes the form of a Laguerre polynomial. Explicitly, up to scaling,
\[
    p(z) = L_k^{(-2kx-1)}(-2\lambda xz) :=  \sum_{m=0}^k \binom{k - 2kx - 1}{k - m} \frac{(2\lambda xz)^m}{m!}.
\]
The discriminant is proportional to
\[
(2\lambda x)^{k(k-1)} \prod_{j=1}^k (j - 2kx - 1)^{j-1}. 
\]
Now we classify the non-generic locus:
\[
\begin{split}
    p(0) =0 &\Leftrightarrow x \in \{ 0, \frac{1}{2k}, \dots, \frac{k-1}{2k} \}
    \\
    \Delta(p) =0 &\Leftrightarrow x \in \{ 0, \frac{1}{2k}, \dots, \frac{k-1}{2k} \}. 
\end{split}
\]
Note that $x = \frac{k-1}{2k}$ has been studied in Case 1. Combining Cases 1 and 2 yields the uniqueness proof. 

To compute the rank of the Jacobian $J_f(\alpha)$, given $v =(v_1,\dots,v_k) \in \mathbb{C}^k$, let $\delta_v p$ and $\delta_v \lambda$ be the infinitesimal deformation of $p$ and $\lambda$ corresponding to $\alpha \to \alpha + \epsilon v$. Explicitly,
\[
\begin{split}
    \delta_v p :=& \frac{d}{d\epsilon} \Big|_{\epsilon =0} \prod_{\mu=1}^k (z-\alpha_\mu - \epsilon v_\mu) = - p(z) \sum_{\mu=1}^k \frac{v_\mu}{z-\alpha_\mu} \in \mathbb{C}[z]_{\leq k-1},
    \\
    \delta_v \lambda :=& \frac{d}{d\epsilon} \Big|_{\epsilon=0} \left( \sum_{\nu=1}^k \frac{1}{\alpha_\nu + \epsilon v_\nu} \right) = - \sum_{\nu=1}^k \frac{v_\nu}{\alpha_\nu^2}.
\end{split}
\]
Let $D \mathcal{L}_p$ be the Fr\'echet derivative of the operator $\mathcal{L}$ from equation~\eqref{eqn_sol_x} at $p$, defined as
\[
\begin{split}
    D\mathcal{L}_p\,f :=&  \frac{d}{d\epsilon} \Big|_{\epsilon =0} \mathcal{L}[p + \epsilon f]
    \\
    =& z f'' + (2x\lambda z-2kx) f' -2kx \lambda f + 2 (\delta_v \lambda)(zp'-kp).
\end{split}
\]
Note that $v \in \ker (J_f(\alpha))$ if and only if $D\mathcal{L}_p\, \delta_v p  =0$. The condition $D\mathcal{L}_p \, \delta_v p  =0$ is equivalent to the following:
\begin{equation} \label{eqn_linearlized}
\begin{split}
     L[\delta_v p] :=& z (\delta_v p)'' + (2x \lambda z - 2kx)(\delta_v p)' - 2kx \lambda (\delta_v p) 
     \\
     =& -2(\delta_v \lambda)(z p' - k p).
\end{split}
\end{equation}
Here $L$ is the linear differential operator defined by the homogeneous part of $D\mathcal{L}_p$. For any $z^d$ with $0 \leq d \leq k-1$:
\[
L[z^d] = 2x\lambda(d-k) z^d + d(d - 1 - 2kx) z^{d-1}.
\]

When $x \in \mathbb{C} \setminus \{0, \frac{1}{2k}, \dots, \frac{k-1}{2k} \}$, $L$ induces an isomorphism on $\mathbb{C}[z]_{\leq k-1}$ since the leading coefficient $2x\lambda(d-k) \neq 0$. For any choice of $\delta_v \lambda \in \mathbb{C}$, there exists a unique solution $\delta_v p$ satisfying equation~\eqref{eqn_linearlized} and therefore a unique vector $v$ with $v_\mu = -\frac{\delta_v p(\alpha_\mu)}{p'(\alpha_\mu)}$. We conclude that $\ker J_f(\alpha)$ is 1-dimensional, spanned by $v$, so the rank is $k-1$. 

When $x = \frac{k-1}{2k}$, $L : \mathbb{C}[z]_{\leq k-1} \rightarrow \mathbb{C}[z]_{\leq k-1}$ has a 1-dimensional kernel since $\lambda=0$. We have an extra freedom of choice of $\delta_v p$ . This finishes the proof. 
\end{proof}

We study the deformed system:
\[
    f_\mu (\alpha) = \rho_\mu, \text{ for }\mu=1,\dots,k
\]
with $\sum_\mu \rho_\mu=0$. Intuitively, the solution in Lemma~\ref{lemma_boundary} deforms to a degree $k!$ quasi-projective rational curves. We make it precise below:
\begin{corollary} \label{cor_boundary}
Let $\alpha_1=1$ and $f$ be the system in Lemma~\ref{lemma_boundary}. There exist Zariski open sets 
\[
\begin{split}
    U &\subset \left\{ (x,\alpha)\in \mathbb{C}^{k+1} | \alpha_1=1; \ \alpha_i \neq \alpha_j   \right\},
    \\
    V &\subset \left\{ (x, \boldsymbol{\rho}) \in \mathbb{C}^{k+1} \Big| \sum_{\mu=1}^k \rho_\mu =0 \right\},
\end{split}
\]
such that $\Big( (\mathbb{C}\setminus\{ 0, \frac{1}{2k}, \dots,\frac{k-2}{2k} \}), \mathbf{0} \Big) \subset V$ and $f: U \rightarrow V$ is a morphism of degree $k!$.
\end{corollary}
\begin{proof}
The existence of $V$ containing $(\mathbb{C} \setminus \{ 0, \frac{1}{2k}, \dots, \frac{k-1}{2k} \}, \mathbf{0})$ follows from the rank of Jacobian and inverse function theorem. 

The technical part is to prove $p = ( \frac{k-1}{2k} , \mathbf{0} ) \in V$. Note that $f^{-1}(p)$ contains $(k-1)!$ points with multiplicity $k$. We claim that when deform in either $x$ or $r$ directions generically, each point splits into $k$ solutions. It gives the following Puiseux series expansion. 

Let $\alpha := (1, \alpha_2, \dots, \alpha_k)$ be a solution from Lemma~\ref{lemma_boundary}, and let $v_m := (1, \alpha_2^m,\dots,\alpha_k^m)$. Note that 
\[
\begin{cases}
    \ker J_f(\alpha) = \mathbb{C} v_1, \text{ for }x\in \mathbb{C} \setminus\{0, \frac{1}{2k}, \dots, \frac{k-2}{2k}\};
    \\
    \ker J_f(\alpha) = \mathbb{C} v_1 \oplus \mathbb{C} v_2,  \text{ for }x= \frac{k-1}{2k}.
\end{cases}
\]
For $x = \frac{k-1}{2k}$, we solve the deformed system
\[
    f_\mu \Big( \alpha + \sum_{m=1}^k\sum_{i \geq 1} c_{m i} v_m \epsilon^i\Big) = \sum_{j \geq k} \rho_{\mu j} \epsilon^j, \text{ for }\mu=1,\dots,k,
\]
where $\sum_m c_{m i} = 0$ and $\sum_{\mu} \rho_{\mu j}=0$. The right-hand side starts with $\epsilon^k$ so that $\epsilon \sim \rho^{1/k}$, ensuring the branching multiplicity is exactly $k$. Comparing the $\epsilon$-coefficients gives $c_{11} = -c_{21}$ and $c_{m 1} =0$ for $m=3,\dots,k$. 

For the coefficients of $\epsilon^{\leq k-1}$, note that the derivatives 
\[
    w_j := \frac{d^j}{d\epsilon^j} f(\alpha + \epsilon v_2)\Big|_{\epsilon=0} \in {\rm Im}(J_f), \text{ for }j=1,\dots,k-1;
\]
hence $\{ c_{m j} \}_{ j \leq k-1 }$ can be completely expressed in terms of $c_{21}$. In fact, $c_{m j} = {\rm const} \cdot c_{21}^j$. For the coefficient of $\epsilon^k$, we derive the form:
\[
    M_k c_{21}^k = \sum_{\mu =1}^k \alpha_\mu  \rho_{\mu k}.  
\]
For higher order $\epsilon^{>k}$, the inverse function theorem applies since
\[
    w_k \oplus {\rm Im}(J_f) = \Big\{ \mathbf{z} \in \mathbb{C}^k \ \Big| \ \sum_\mu z_\mu=0  \Big\}.
\]
It remains to show that $M_k = - \frac{k(k-1)^k}{2(k-1)!} \neq 0$. 

Computing $M_k$ directly by evaluating $c_{m j}$ is highly complicated. It requires iteratively differentiating the nonlinear rational system $k$ times and repeatedly inverting the restricted Jacobian. Instead, we bypass this by considering the deformed polynomial
\[
\begin{split}
    P_{k,\epsilon}(z) &:= \left[ \prod_{\mu=1}^k \Big(z - \alpha_\mu(\epsilon)\Big) \right]_{\leq t^{k-1}} 
    \\
    &= z^k-1 + p_0 \epsilon - \sum_{i=1}^{k-1} p_i \epsilon^i z^i \in \mathbb{C}[\epsilon, z].
\end{split}
\] 
The coefficients $p_i$ depend only on $c_{21}$. In particular, $p_0 = p_1 = kc_{21}$. It satisfies the following condition:
\[
\begin{split}
    \mathcal{L}[P_{k,\epsilon}](z) &:= \Bigg( z \frac{d^2}{dz^2} + \Big( \frac{ 2(k-1)\lambda z}{2k} + 1-k\Big) \frac{d}{dz} - (k-1)\lambda \Bigg) P_{k,\epsilon}(z) 
    \\
    &\equiv 0  \pmod{\epsilon^k},
\end{split}
\]
where $\lambda := {\rm Coeff}\Big(\epsilon,\sum_\mu \frac{1}{\alpha_{\mu}(\epsilon)}\Big) = -k c_{21} \epsilon$. 
Solving the equation gives:
\[
    p_i =  \frac{k(k-1)^i}{i! (k-i)} c_{21}^i, \text{ for }i=1,\dots,k-1.
\]
Hence, we can evaluate $M_k$:
\[
    M_k c_{21}^k = \sum_{\mu=1}^k \alpha_\mu r_{\mu k} = \sum_{\mu=1}^k \alpha_\mu \frac{ {\rm Coeff}(\epsilon^k,\mathcal{L}[P_{k,\epsilon}](\alpha_\mu) )}{ 2\alpha_\mu P_{k,\epsilon}'(\alpha_\mu) } = - \frac{k(k-1)^k}{2(k-1)!} c_{21}^k.
\]
This proves the corollary.
\end{proof}

\begin{remark} \label{remark_boundary}
In the case of classical Lam\'e equation, i.e. $r=1, p=0$, and $\mathbf{n} =n \in \mathbb{N}$, \cite[Theorem~7.3 (iv)]{Chai_Lin_Wang_2015} states that $\{[\mathbf{0}]\} \in \overline{Y}_n \setminus Y_n$, where they used the relation of $\{[\mathbf{a}]\}$ with parameter $B$ in the Lam\'e equation.

Here we give an alternate proof using Lemma~\ref{Lemma_sol_x}. The argument only uses the defining equation of Lam\'e curve. The technique will extend to the generalized Lam\'e curve in Theorem~\ref{special-pt}.

Let
\[
    a_\mu = \sum_{i \geq 1} \alpha_{\mu i} \epsilon^i, \text{ for }\mu = 1,\dots,n.
\]
The goal is to show that such a curve in $t$ exists in $Y_n$.
Expand the defining equation of the Lam\'e curve
\[
    \sum_{\nu=1,\neq \mu}^n \Big( \zeta(a_\mu-a_\nu) - \zeta(a_\mu) + \zeta(a_\nu) \Big) =0, \text{ for }\mu=1,\dots,n,
\]
using $\zeta(\epsilon) = \frac{1}{\epsilon} - \sum_{k=1}^{\infty} \mathcal{G}_{2k+2} \epsilon^{2k+1}$, where $\mathcal{G}_{2k+2}$ is the Eisenstein series of weight $2k+2$. We claim that $\{ \alpha_{\mu i} \}$ can be uniquely solved up to permutation and reparametrization of $\epsilon$, i.e. $\epsilon' = \sum_{i \geq 1} c_i \epsilon^i$ with $c_1 \neq 0$. 

The coefficient of $\epsilon^{-1}$ of the system reduces to the case in Lemma~\ref{Lemma_sol_x} with $k=n$, and $x = 1$. Hence $\{ \alpha_{\mu 1} \}$ can be uniquely solve up to permutation and scaling. The scaling freedom is absorbed in the choice of $c_1$. The same argument works for all higher order. To solve $\{ \alpha_{\mu i} \}$, compare the coefficient of $\epsilon^{i-1}$ gives the form
\[
    J_f( \alpha_{1 1}, \dots, \alpha_{\alpha_{n1}}) 
    \begin{bmatrix}
        \alpha_{1 i}
        \\
        \vdots
        \\
        \alpha_{ni}
    \end{bmatrix} = \{ \text{terms involving } \mathcal{G}_k, \alpha_{\mu,j}, \text{ with } j < i\}.
\]
The rank of $J_f$ is $n-1$ gives a one dimensional freedom of $\{\alpha_{\mu i}\}$. The freedom is absorbed in  the choice of $c_i$. This proves the claim. 
\end{remark}

\begin{proposition}[Admissible boundary points] \label{prop_no_mixed}
Given $\mathbf{n} \in (\mathbb{C} \setminus \{0\} )^r$ with total weight $n\in \mathbb{N}$, let $[(\mathbf{a},h)] \in  \partial\overline{\mathcal{Y}}_{\mathbf{n},\mathbf{p}}(\tau)$ with 
\[
    \{[\mathbf{a}]\} = \{[a_1], \dots, [a_l], [p_1]^{k_1}, \dots, [p_r]^{k_r}\}, 
\]
where $[a_\mu] \neq [p_j]$ for $\mu=1,\dots,l$ and $j=1,\dots,r$. The non-zero multiplicities $k_j$ either all take the form $2n_j+1$ (assuming $n_j \in \frac{1}{2}\mathbb{N}$) or none do. If $l>0$, then $k_j = 2n_j+1$ for all $k_j > 0$.  
\end{proposition}

\begin{proof}
Let $[(\mathbf{a}(\epsilon),h(\epsilon))] \in \mathcal{Y}_{\mathbf{n},\mathbf{p}}(\tau)$ be a curve converging to $[(\mathbf{a},h)] \in \partial \overline{\mathcal{Y}}_{\mathbf{n},\mathbf{p}}(\tau)$ at $\epsilon=0$.
Up to a choice of representative of $[(\mathbf{a},h)]$, we assume that $a_\mu = b$ whenever $[a_\mu] = [b] \in E_\tau$. Period shift can be absorbed by the shift of $h$ by the definition of $\overline{\mathcal{Y}}_{\mathbf{n},\mathbf{p}}(\tau)$. We omit all the bracket in the discussion below.

Rewrite the defining equation of $\mathcal{Y}_{\mathbf{n},\mathbf{p}}(\tau)$ using $h$:
\[
    h(\epsilon)+\sum_{\nu\ne\mu}\zeta(a_\mu(\epsilon)-a_\nu(\epsilon))
   -\sum_i n_i\zeta(a_\mu(\epsilon)-p_i)=0, \quad \text{for }\mu=1,\dots,n. 
\]
First we prove that two or more $\{a_\mu(\epsilon)\}_{\mu=1}^n$ cannot collide outside the set of $\{ p_i \}_{i=1}^r$ in the limit. Let $J$ be the set of $\{a_\mu(\epsilon)\}$ converging to a point $b \in E_\tau \setminus\{p_i\}_{i=1}^r$ with $|J| \geq 2$. Summing the defining equations over $J$ cancels the alternating terms and gives the boundedness of $h$.
\[
    h(0) |J| = \lim_{\epsilon\rightarrow 0}\sum_{\mu \in J}\Big( h(\epsilon)+\sum_{\nu\ne\mu}\zeta(a_\mu(\epsilon)-a_\nu(\epsilon))
   -\sum_i n_i\zeta(a_\mu(\epsilon)-p_i) \Big) = O(1). 
\]
On the other hand, choose $\mu_0,\nu_0\in J$ such that
\[
    d:=\max_{\substack{\mu,\nu\in J\\\mu\ne\nu}}
    \operatorname{ord}_{\epsilon=0}(a_\mu-a_\nu)
    =\operatorname{ord}_{\epsilon=0}(a_{\mu_0}-a_{\nu_0}),
\]
and set
\[
    K=\{\mu\in J:a_\mu-a_{\mu_0}=O(\epsilon^d)\}.
\]
Note that $\{ \mu_0, \nu_0\} \subset K$ and hence $|K| \geq 2$. We write 
\[
    a_\mu-a_{\mu_0}
    =\beta_\mu\epsilon^d+O(\epsilon^{d+1}),
    \qquad \mu\in K.
\]
The coefficients $\beta_\mu$ are distinct by maximality of $d$. Moreover extract the $\epsilon^{-d}$ coefficient of the defining equations gives
\[
    \sum_{\substack{\nu\in K\\\nu\ne\mu}}
    \frac1{\beta_\mu-\beta_\nu}=0,
    \qquad \mu\in K.
\]
Multiplying by $\beta_\mu$ and summing over $K$ gives
\[
    0=
    \sum_{\substack{\mu,\nu\in K\\\mu<\nu}}
    \left(
    \frac{\beta_\mu}{\beta_\mu-\beta_\nu}
    +\frac{\beta_\nu}{\beta_\nu-\beta_\mu}
    \right)
    =\binom{|K|}{2},
\]
a contradiction.

Next, we study all admissible collision to $\{ p_i \}_{i=1}^r$. Let $I_j$ be the index set of $\mu$ with $a_\mu(\epsilon)$ converging to $p_j$ and with $|I_j| = k_j$. Write
\[
a_\mu(\epsilon) = p_j + \alpha_\mu \epsilon^{d_{\mu j}} + O(\epsilon^{d_{\mu j}+1}), \quad \text{for } \mu \in I_j,
\]
where $d_{\mu j} \geq 1$ and $\alpha_\mu \neq 0$. 

To balance the local expansions, we start with the highest order poles. Define
\[
    d=\max_{j,\mu\in I_j}d_{\mu j}, \qquad S=\bigsqcup_j\{\mu\in I_j:d_{\mu j}=d\}.
\] 
Then
\[
    h(\epsilon)=H\epsilon^{-d}+O(\epsilon^{-d+1}), \qquad H=\frac1n\sum_j n_j \sum_{\mu\in I_j\cap S}\frac1{\alpha_\mu}.
\]

For each $\mu \in I_j \cap S$, extracting the $\epsilon^{-d}$ coefficient of the defining
equations gives
\begin{equation}\label{eqn:local_collide}
 \sum_{\substack{\nu\in I_j\cap S\\\nu\ne\mu}}
 \frac1{\alpha_\mu-\alpha_\nu}
 -\frac{n_j}{\alpha_\mu}+H=0,
 \qquad \mu\in I_j\cap S.
\end{equation}
Summing over index set $S$ yields
\[
 H(n-|S|)=0.
\]
We consider two cases.

\noindent \textbf{Case 1:} $H \neq 0$. Then $|S|=n$, so $l=0$ and $I_j\cap S=I_j$ for every $j$. The local equation on each $I_j$ takes the form
\[
\sum_{\nu \in I_j, \neq \mu} \frac{1}{\alpha_\mu - \alpha_\nu} - \frac{n_j}{\alpha_\mu} + H = 0, \quad \text{for } \mu \in I_j.
\]
By the proof of Case 2 of Lemma~\ref{lemma_boundary}, this system admits solutions only if $\frac{n_j}{k_j} \notin \{ 0, \frac{1}{2k_j}, \dots, \frac{k_j-1}{2k_j} \}$. Equivalently, $k_j \leq 2n_j$ when $n_j \in \frac{1}{2} \mathbb{N}$.

\noindent \textbf{Case 2:} $H = 0$. We prove that $h(0)$ is finite. First we observe that equation~\eqref{eqn:local_collide} reduces to Lemma~\ref{Lemma_sol_x} and therefore $|I_j\cap S|=2n_j+1 \leq k_j = |I_j|$ whenever $I_j\cap S\ne\varnothing$. Fix such $j$ and suppose that $h(\epsilon)$ is unbounded when $\epsilon \rightarrow 0$, define
\[
    P(z,\epsilon)=\prod_{\mu\in I_j}(z-a_\mu(\epsilon)+p_j)
    =z^{k_j}+\sum_{s=0}^{k_j-1}c_s(\epsilon) z^s.
\]
Using the same ODE method as in Lemma~\ref{Lemma_sol_x} and Lemma~\ref{lemma_boundary}, the defining system can be rewritten as:
\[
    \Bigg( \delta(\epsilon) z \frac{d^2}{dz^2} + 2\Big(z-\delta(\epsilon)n_j + \delta(\epsilon) z g(z,\epsilon)\Big)   \frac{d}{dz} \Bigg) P(z,\epsilon) = Q(z,\epsilon) P(z,\epsilon),
\]
where $Q$ and $g$ are holomorphic function in $\epsilon$ and $a_{\mu}(\epsilon)$, and $\delta(\epsilon) := 1/h(\epsilon)$. 

By assumption, we have $\delta(0)=0$, and we solve $P(t,0) = t^{k_j}$ and $Q(t,0) = 2k_j$. Furthermore, direct computation shows that for $|\epsilon|<<1$, the vanishing of coefficients extend:
\[
    c_0(\epsilon) = \dots = c_{2n_j}(\epsilon) =0. 
\]
In particular $P(0,\epsilon) =0$ for $|\epsilon|<<1$. This is a contradiction, since 
\[
    P(0,\epsilon) = \prod_{\mu\in I_j}(-a_\mu(\epsilon)+p_j) \neq 0, \quad \text{for }\epsilon \neq 0. 
\]
We conclude that $h(\epsilon)$ remains bounded in the limit.

Finally, for each $\mu \in I_j$, multiply the defining equation by $a_\mu(\epsilon) - p_j$ and sum over $I_j$ gives:
\[
0=
\sum_{\substack{\mu,\nu\in I_j\\\mu<\nu}}
(a_\mu(\epsilon)-a_\nu(\epsilon))\zeta(a_\mu(\epsilon)-a_\nu(\epsilon))-n_j\sum_{\mu\in I_j}
(a_\mu(\epsilon)-p_j)\zeta(a_\mu(\epsilon)-p_j)+o(1).
\]
Passing to the limit $\epsilon \rightarrow 0$ gives:
\[
    0=\frac{k_j(k_j-1)}2-n_jk_j=\frac{k_j}{2}(k_j-2n_j-1).
\]
We conclude that $k_j = 2n_j+1$. For the last statement, since $l>0$ excludes Case 1, we complete the proof of the proposition.
\end{proof}

We will call the boundary with $h=\infty$ Type-I and with $h$ finite Type-II. 

\begin{theorem}[Classification of Type-I boundary] \label{special-pt}
Given $\mathbf{n} \in (\mathbb{C} \setminus \{0\} )^r$ with total weight $n\in \mathbb{N}$,
the Type-I boundary (where $h=\infty$) of $\overline{\mathcal{Y}}_{\mathbf{n},\mathbf{p}}(\tau)$ consists of the points $[(\mathbf{a}, \infty)]$ satisfying
\[
    \{[\mathbf{a}]\} = \{[p_1]^{k_1}, \dots, [p_r]^{k_r}\},
\]
where the multiplicities $k_j$ satisfy $k_j \leq 2n_j$ whenever $n_j \in \frac{1}{2}\mathbb{N}$.
\end{theorem}

\begin{proof}
Up to a choice of representative, we assume that $a_\mu = p_j$ whenever $[a_\mu] = [p_j]$. For $j=1, \dots, r$, let $I_j$ be the index set of $\mu$ such that $a_\mu = p_j$. Consider the parametrization
\[
    a_\mu = p_j + \sum_{i \geq 1} \alpha_{\mu i} \epsilon^i, \text{ for } \mu \in I_j, \quad j=1,\dots,r.
\]
The multiplicity constraints have been studied in Case 1 of Proposition~\ref{prop_no_mixed}. 
Here we show that the coefficients $\{ \alpha_{\mu i} \}$ in the parametrization can be solved inductively as in the discussion of Remark~\ref{remark_boundary}.

The coefficient of $\epsilon^{-1}$ in the defining equation of the generalized Lam\'e curve gives
\[
    \sum_{\nu\in I_j,\neq \mu}\frac{n}{\alpha_{\mu 1} - \alpha_{\nu 1}} - \frac{n \cdot n_j}{\alpha_{\mu 1}} + \sum_{i=1}^r n_i \lambda_i = 0, \text{ for } \mu\in I_j, \ j=1,\dots,r,
\]
where 
\[
    \lambda_j := \sum_{\nu \in I_j} \frac{1}{\alpha_{\nu 1}}, \text{ for } j=1,\dots,r.
\]

Summing over the equations for $\mu \in I_j$ gives the global constraints:
\[
    - n n_j \lambda_j + k_j \sum_{i=1}^r n_i \lambda_i = 0, 
\]
which is equivalent to
\[
    [ n_1 \lambda_1 : n_2 \lambda_2 : \dots : n_r \lambda_r ] = [ k_1 : k_2 : \dots : k_r ]. 
\]
Replacing $\sum_{i=1}^r n_i \lambda_i$ with $\frac{n \cdot n_j}{k_j} \lambda_j$ for the equations with $\mu \in I_j$, we separate the variables:
\[
    \sum_{\nu\in I_j,\neq \mu}\frac{1}{\alpha_{\mu 1} - \alpha_{\nu 1}} - \frac{(k_j-1) \frac{n_j}{k_j} }{\alpha_{\mu 1}} + \sum_{\nu\in I_j,\neq \mu} \frac{ \frac{n_j}{k_j} }{\alpha_{\nu 1}} =0, \text{ for }\mu\in I_j, \ j=1,\dots,r.
\]
For each $j$, the system is of the form in Lemma~\ref{lemma_boundary}. Each system gives a unique solution up to scaling. The global constraint ensures that all the scalings are compatible. Overall, we have a $1$-dimensional freedom. 

The same argument works for all higher orders. By separating the system into subsystems for each $j$ and connecting them via global constraints, each order gives only a one-dimensional freedom. These are all absorbed by the possible reparametrizations of $\epsilon$:
\[
    \epsilon' = \sum_{i\geq 1}c_i \epsilon^i, \quad c_1 \neq 0.
\]
Each freedom that arises in solving for $\{ \alpha_{\mu i} \}_{\mu=1}^n$ corresponds to a choice of $c_i$. 

For the second statement, note that 
\[
    h := \lim_{\epsilon\rightarrow 0} \frac{1}{n} \zeta_{\mathbf{n},\mathbf{p}}(\mathbf{a}) = \lim_{\epsilon\rightarrow 0} \frac{1}{n \epsilon} \sum_{j=1}^r n_j \lambda_j =  \infty.
\]
The proof is complete.
\end{proof}

\begin{definition}
We denote the number of ansatz with $h=\infty$ as
\[
    F(\mathbf{n}) := \{ (k_1,\dots,k_r) \in \mathbb{Z}_{\geq 0}\ | \ \sum k_i=n, k_i \leq 2n_i \text{ when } n_i \in \tfrac{1}{2} \mathbb{N}   \}.
\]
Let $I :=\{ i \ |\ n_i \in \tfrac{1}{2} \mathbb{N}  \}$. The number takes the following form:
\[
    F(\mathbf{n}) = \sum_{ J \subset I } (-1)^{|J|} \binom{n_J+r-1}{r-1}_{\rm comb}, 
\]
where $n_J:= n - \sum_{j \in J}(2n_j+1)$, and $\binom{A}{B}_{\rm comb}:= 0$ if $A < B$. 
\end{definition}

By the proof of Case 2 in Proposition~\ref{prop_no_mixed}, any Type-II boundary point (where $h$ is finite) must consist of root collisions with multiplicities of the form $2n_j+1$. Up to reordering, we can write it as
\[
    \{[\mathbf{a}]\} = \{[a_1], \dots, [a_{n'}], [p_1]^{2n_1+1}, \dots, [p_k]^{2n_k+1}\},
\]
where $[a_\mu] \neq [p_j]$ for $\mu=1,\dots,n'$ and $j=1,\dots,r$. Consequently, this boundary configuration only appears when the corresponding fixed weights are half-integers: $n_1,\dots,n_k \in \frac{1}{2}\mathbb{N}$.

To describe the constraints on the remaining non-colliding roots $\mathbf{a}' = \{a_\mu\}_{\mu=1}^{n'}$ and the parameter $h$, we define $\mathbf{n}'$ with:
\[
n_i'=
\begin{cases}
    -n_i-1, \text{ for }i=1,\dots,k;
    \\
    n_i, \text{ otherwise}.
\end{cases}
\]
Furthermore, let $F_{\mathbf{n},\mathbf{p};i}$ denote the log-free condition for the generalized Lam\'e equation of type $\mathbf{n}$ at $p_i$, and let $A_{\mathbf{n},\mathbf{p};i}$ and $B_{\mathbf{n},\mathbf{p}}$ be the local coefficients defined by equations~\eqref{A-a} and~\eqref{B-a}. We have

\begin{theorem}[Algebraic constraints of Type-II boundary] \label{special-pt-II}
With the setup and notation above, let $[(\mathbf{a},h)] \in \partial\overline{\mathcal{Y}}_{\mathbf{n},\mathbf{p}}(\tau)$ be a Type-II boundary point. The remaining non-colliding roots $\mathbf{a}'$ and the parameter $h$ satisfy:
\begin{enumerate}
    \item If $n'=0$, 
    \[
        F_{\mathbf{n},\mathbf{p};i}\Big(A_{\mathbf{n}',\mathbf{p};i}(\mathbf{a}',h),B_{\mathbf{n}',\mathbf{p}}(\mathbf{a}',h)\Big)=0, \text{ for }i=1,\dots,k;
    \]
    \item If $n'>0$, 
    \[
    \begin{cases}
        [(\mathbf{a}',h)] \in \mathcal{Y}_{\mathbf{n}', \mathbf{p}}(\tau);
        \\
        F_{\mathbf{n},\mathbf{p};i}\Big(A_{\mathbf{n}',\mathbf{p};i}(\mathbf{a}',h),B_{\mathbf{n}',\mathbf{p}}(\mathbf{a}',h)\Big)=0, \text{ for }i=1,\dots,k.
    \end{cases}
    \]
\end{enumerate}
\end{theorem}

\begin{proof}
Up to a choice of representative, we assume that $a_\mu = p_j$ whenever $[a_\mu] = [p_j]$. Let $I_j$ be the index set of $\mu$ such that $a_\mu = p_j$.
Consider the following parametrization
\[
    a_\mu(\epsilon) = p_j + \sum_{i \geq 1} \alpha_{\mu i} \epsilon^i, \text{ for }\mu \in I_j, \quad j=1,\dots,k,
\]
with $\{\alpha_{\mu 1}\}_{\mu \in I_j} = \{ \xi_{2n_j+1}^m \}_{m=0}^{2n_j}$, where $\xi_{2n_j+1}$ is a primitive $(2n_j+1)$-th root of unity. This choice follows from Lemma~\ref{Lemma_sol_x}. Assuming the existence of boundary points ensures that we can solve $\alpha_{\mu i}$ inductively without obstruction. 

A direct computation shows that  
\[
\begin{split}
    &\lim_{\epsilon \rightarrow 0} A_{\mathbf{n},\mathbf{p};i} \Big(\mathbf{a}(\epsilon), \frac{1}{n} \zeta_{\mathbf{n},\mathbf{p}}( \mathbf{a}(\epsilon) )  \Big) = A_{\mathbf{n}',\mathbf{p};i}(\mathbf{a}',h);
    \\
    &\lim_{\epsilon \rightarrow 0} B_{\mathbf{n},\mathbf{p}} (\mathbf{a}(\epsilon), h(\epsilon))  = B_{\mathbf{n}',\mathbf{p}}(\mathbf{a}',h).
\end{split}
\]
Note that the ansatz solution corresponding to $\mathbf{a}(\epsilon)$ is log-free. It satisfies the log-free condition at $p_i$. The values of $A_{\mathbf{n},\mathbf{p};i}$ and $B_{\mathbf{n},\mathbf{p}}$ vary continuously in the limit $\epsilon \rightarrow 0$. The limit also satisfies the same log-free condition $F_{\mathbf{n},\mathbf{p};i}$.  
\end{proof}
\begin{remark}
We explain more about Theorem~\ref{special-pt-II}.

First, it classifies all possible Type-II boundary points of $\mathcal{Y}_{\mathbf{n},\mathbf{p}}(\tau)$. More precisely, these points take the form $[(\mathbf{a},h)] \in \overline{\mathcal{Y}}_{\mathbf{n},\mathbf{p}}(\tau)$ where the root configuration contains the collisions $\{[p_1]^{2n_1+1}, \dots, [p_k]^{2n_k+1}\}$ and the remaining roots form a point $[(\mathbf{a}',h)] \in \mathcal{Y}_{\mathbf{n}',\mathbf{p}}(\tau)$.  Unlike the Type-I boundary, Type-II boundary might not exist. The reason comes from the algebraic constraints in Theorem~\ref{special-pt-II}. 

Note that $[(\mathbf{a}',h)] \in \mathcal{Y}_{\mathbf{n}',\mathbf{p}}(\tau)$ automatically implies that it satisfies the log-free condition of type $\mathbf{n}'$ at the fixed poles $p_i$:
\[
    F_{\mathbf{n}',\mathbf{p};i}\Big(A_{\mathbf{n}',\mathbf{p};i}(\mathbf{a}',h),B_{\mathbf{n}',\mathbf{p}}(\mathbf{a}',h)\Big)=0, \text{ for }i=1,\dots,k.
\]
The constraints in Theorem~\ref{special-pt-II} additionally impose the log-free condition of type $\mathbf{n}$ at $p_i$: 
\[
    F_{\mathbf{n},\mathbf{p};i}\Big(A_{\mathbf{n}',\mathbf{p};i}(\mathbf{a}',h),B_{\mathbf{n}',\mathbf{p}}(\mathbf{a}',h)\Big)=0, \text{ for }i=1,\dots,k.
\]
We emphasize the appearance of both $\mathbf{n}$ and $\mathbf{n}'$. The type $\mathbf{n}$ condition gives $k$ extra constraints to the point $[(\mathbf{a}',h)] \in \mathcal{Y}_{\mathbf{n}',\mathbf{p}}(\tau)$. Fixing $\mathbf{p}$, the expected dimension of this kind of boundary is $1-k$. The locus of all Type-II boundaries needs to be further explored. 

Second, Theorem~\ref{special-pt-II} only gives the necessary conditions for the existence of Type-II boundary points. It is expected that these conditions are also sufficient. In other words, 
\[
    F_{\mathbf{n},\mathbf{p};i}\Big(A_{\mathbf{n}',\mathbf{p};i}(\mathbf{a}',h),B_{\mathbf{n}',\mathbf{p}}(\mathbf{a}',h)\Big)=0, \text{ for }i=1,\dots,k,
\]
is exactly the obstruction to existence. For small $n_i$, this can be proved by direct power series computation, but the computational complexity grows rapidly for large $n_i$.

Finally, note that when $n'>0$, $h$ is uniquely determined. When $n'=0$, $h$ has finitely many choices (see Example~\ref{eg:GLC_half_periods}). We give an alternative computation using the constraints in Theorem~\ref{special-pt-II}. 

Let $\mathbf{n} = (\frac{1}{2},\frac{1}{2},\frac{1}{2}, \frac{1}{2} )$. 
Consider a Type-II boundary point $[(\mathbf{a},h)] \in \overline{\mathcal{Y}}_{\mathbf{n},\mathbf{p}}(\tau)$ with $\{\mathbf{a}\} = \{p_1^2\}$. 
The corresponding ansatz takes the weight 0 form
\[
    w_{\mathbf{n}',\mathbf{p};h}(z) = \frac{\exp(hz) \sigma(z-p_1)^{3/2}  }{ \prod_{i=2}^4 \sigma(z-p_i)^{1/2} }, 
\]
where $\mathbf{n}' = (-\frac{3}{2},\frac{1}{2},\frac{1}{2},\frac{1}{2})$.

From Theorem~\ref{Ansatz}, we have
\[
\begin{split}
    A_1 &= 3h - \frac{3}{2} \sum_{j=2}^4 \zeta_{1j},
    \\
    A_i &= -h - \frac{3}{2} \zeta_{i1} + \frac{1}{2}\sum_{j=2,\neq i}^4 \zeta_{ij}, \quad \text{for } i =2,3,4,
\end{split}
\]
and
\[
\begin{split}
    B &= A_i^2 - \sum_{j \neq i} \zeta_{ij} A_j - \frac{3}{4} \wp_i, \quad \text{for } i=2,3,4,
    \\
    B &= \Big( \frac{A_1}{3} \Big)^2  - \sum_{j \neq 1} \zeta_{1j} A_j + \frac{5}{4}\wp_1.
\end{split}
\]
From Theorem~\ref{special-pt-II}, we have $1$ extra equation at $p_1$:
\[
    B = A_1^2 - \sum_{j \neq 1} \zeta_{1j} A_j - \frac{3}{4} \wp_1. 
\]
Comparing the two equations involving $B$ yields $\frac{8}{9}A_1^2 = 2\wp_1$, which gives
\[
    h = \frac{1}{2} \Big( \sum_{j=2}^4 \zeta_{1j} \pm \sqrt{\wp_1} \Big).
\]
\end{remark}

\section{Degeneration formula}
\subsection{Addition map}
Let $\mathbf{n} \in ( \mathbb{C} \setminus \{ 0\} )^r$ with total weight $n\in \mathbb{N}$.
 Consider the  addition map:
\[
\sigma_{\mathbf{n},\mathbf{p}} : E_{\tau}^n \rightarrow E_{\tau}, \quad \
[\mathbf{a}] \mapsto \Big[\sum_{\mu=1}^n a_{\mu} - \sum_{i=1}^r n_ip_i \Big].
\]
By abuse of notation, we use the same symbol $\sigma_{\mathbf{n},\mathbf{p}}$ for the induced map on the quotient:
\begin{equation}\label{e:addmap}
\sigma_{\mathbf{n},\mathbf{p}} : \Sym^n E_{\tau} \rightarrow E_{\tau}, \quad \{[\mathbf{a}]\} \mapsto \Big[\sum_{\mu=1}^n a_{\mu} - \sum_{i=1}^r n_ip_i \Big].
\end{equation}
Since $n_i$ is in general not integer, the assumption $p_i \in \mathbb{C}$ is essential here.

The following degree formula is crucial for studying the geometry of (underlying) generalized Lam\'e curve. 
\begin{theorem} \label{thm_deg}
Let $I := \{ i\ |\ n_i \in \tfrac{1}{2}\mathbb{N} \}$. The addition map
\[
\sigma_{\mathbf{n},\mathbf{p}}: \overline{Y}_{\mathbf{n},\mathbf{p}}(\tau) \to E_\tau
\] 
is generically of degree 
\[
    \sum_{J \subset I} (-1)^{|J|} \binom{n_J +r}{r+1}_{\rm comb},
\]
where $n_J := n - \sum_{j \in J}(2n_j +1)$, and $\binom{A}{B}_{\rm comb} := 0$ if $A<B$.
\end{theorem}
\begin{proof}
It is equivalent to study the generic degree of $\sigma_{\mathbf{n},\mathbf{p}}$ on the open set $Y_{\mathbf{n}, \mathbf{p}}(\tau)$, since the boundary has strictly lower dimension.
The proof proceeds in two steps. First, we compute the intersection number of the divisors $D_{\mu,\rho_\mu}$ in $E_\tau^n$ for general $\boldsymbol{\rho}$. By deforming to $|\rho_\mu| \to \infty$, we combinatorially count the points and obtain $n! \binom{n+r}{r+1}$. Second, we analyze the limit $\rho_\mu \to 0$ and apply an inclusion-exclusion argument to systematically subtract degenerate solutions (where $a_\mu \to p_i$), ensuring all remaining intersection points are valid. Because $Y_{\mathbf{n},\mathbf{p}}(\tau)$ is defined in the symmetric product ${\rm Sym}^n E_\tau$, the generic degree of $\sigma_{\mathbf{n},\mathbf{p}}$ is $1/n!$ times the number of valid solutions.

Given $\boldsymbol{\rho}=(\rho_1,\dots,\rho_n) \in \mathbb{C}^n$ with $\sum_{\mu=1}^n \rho_\mu=0$, we define 
\[
D_{\mu,\rho_\mu} := \left\{ \sum_{i = 1}^{r} \sum_{\nu = 1, \neq \mu}^n 
n_s\Big(
\zeta(a_\mu - a_\nu) - \zeta(a_\mu - p_i) + \zeta(a_\nu - p_i)
\Big) =  \rho_\mu \right\} \subset E_{\tau}^n,
\]
for $\mu=1,\dots,n$. Furthermore, given $c \in E_\tau$, we define
\[
Z^c(\boldsymbol{\rho}) :=  \left( \bigcap_{\mu=1}^n D_{\mu,\rho_\mu}\right) \cap \sigma_{\mathbf{n},\mathbf{p}}^{-1}(c) \subset E_\tau^n.
\]

We claim that, for generic $c$ and $\boldsymbol{\rho}$, by counting with multiplicities, 
\[
\#(Z^c(\boldsymbol{\rho})) = n! \binom{n+r}{r+1}.
\]
Assume $|\rho_\mu| >M>0$ for all $\mu$. For $M$ large enough, the points of $Z^c(\boldsymbol{\rho})$ cluster into components corresponding to partitions of the indices:
\[
Z^c(\boldsymbol{\rho}) = \bigsqcup_{\substack{I_0 \sqcup \dots \sqcup I_r = \{1,\dots,n\} \\ I_0 \neq \emptyset }} P_{I_0,\dots,I_r}^c (\boldsymbol{\rho}).
\]
As $M \to \infty$, each component $P^c_{I_0,\dots,I_r}$ is characterized by its limit:
\[
P_{I_0,\dots,I_r}^c(\infty) = 
\left\{
\mathbf{a}\in E_\tau^n \;\middle|\; 
\begin{array}{l}
a_\mu = a_\nu \text{ for all } \mu, \nu \in I_0, \\
a_\mu = p_j \text{ for all } \mu \in I_j, \forall j \in \{1,\dots,r\}
\end{array}
\right\}
\]
with multiplicity
\[
(|I_0|-1)! \prod_{i=1}^r |I_i|!.
\] 
The multiplicity corresponds to the number of ways to obtain $P_{I_0,\dots,I_r}^c(\infty)$ from the intersection of the irreducible components of $D_{\mu,\rho_\mu}$ for $\mu=1,\dots,n$. We also note that $P_{I_0,\dots,I_r}^c(\infty)$ is the disjoint union of $|I_0|^2$ points, corresponding to the number of $|I_0|$-torsion points in $E_\tau$.

Combining all the information above, we have 
\[
\begin{split}
\#(Z^c(\boldsymbol{\rho})) &=  \sum_{I_0 \sqcup \dots \sqcup I_r = \{1,\dots,n\}} | P_{I_0,\dots,I_r}^c (\boldsymbol{\rho}) | \cdot (|I_0|-1)! \prod_{i=1}^r |I_i|!
\\
&= \sum_{I_0 \sqcup \dots \sqcup I_r = \{1,\dots,n\}}   |I_0|\prod_{i=0}^r |I_i|!
\\
&= \sum_{|I_0|+\dots+|I_r|=n} \frac{n!}{\prod_{i=0}^r |I_i|!} |I_0|\prod_{i=0}^r |I_i|!
\\
&= n! \sum_{|I_0|+\dots+|I_r|=n} |I_0| = n! \binom{n+r}{r+1}.
\end{split}
\]
This finishes the first step.

Next, let $\boldsymbol{\rho} \to 0$. We study the intersection points that degenerate, i.e., where $a_\mu \to p_i$. We claim that the set of degenerate solutions can be stratified by subsets $J \subset I := \{ i \mid n_i \in  \tfrac{1}{2} \mathbb{N} \}$ satisfying 
\[
n_J := n - \sum_{j \in J} (2n_j+1) \ge 0.
\]
More precisely, for $J=\{ 1,\dots,l \}$, up to permutation of coordinates, the degenerate solutions in $E_\tau^n$ take the form:
\[
    \mathbf{a} = (\underbrace{p_1, \dots, p_1}_{2n_1+1}, \dots, \underbrace{p_l, \dots, p_l}_{2n_l+1}, a'_1, \dots, a'_{n_J}),
\]
where the remaining coordinates form a point $\{[a'_1], \dots, [a'_{n_J}]\} \in Y_{\mathbf{n}',\mathbf{p}}^{c'}(\tau)$,
with $\mathbf{n}' = ( -n_1-1,\dots,-n_l-1,n_{l+1}, \dots,n_r )$ and 
\[
    Y_{\mathbf{n}',\mathbf{p}}^{c'}(\tau) := Y_{\mathbf{n}',\mathbf{p}}(\tau) \cap \sigma_{\mathbf{n}',\mathbf{p}}^{-1}(c') \text{ with }  c' = c- \sum_{j=1}^l (2n_j+1) p_j.
\]
Counting with multiplicity, there are exactly $n! \deg (\sigma_{\mathbf{n}',\mathbf{p}}: Y_{\mathbf{n}',\mathbf{p}}(\tau) \to E_\tau)$ such solutions in $E_\tau^n$.

To verify this, consider the special case $J=\{1\}$. As in the proof of Theorem~\ref{special-pt}, let $m$ be the number of coordinates $\mu$ for which $a_\mu \to p_1$. Without loss of generality, let 
\[
    a_\mu = p_1 + \sum_{k \ge 1} \alpha_{\mu k} \epsilon^k, \quad \text{for } \mu=1,\dots, m.
\]
Comparing the coefficients of $\epsilon^{-1}$ gives:
\[
\begin{cases}
    \displaystyle\sum_{\nu=1}^{m} \frac{1}{\alpha_{\nu 1}} = 0, \\[15pt]
    \displaystyle\sum_{\substack{\nu=1 \\ \nu \neq \mu}}^{m} \frac{1}{\alpha_{\mu 1} - \alpha_{\nu 1}} - \frac{n_1-1}{\alpha_{\mu 1}} + \sum_{\substack{\nu=1 \\ \nu \neq \mu}}^m \frac{1}{\alpha_{\nu 1}}  = 0.
\end{cases}
\]
By Lemma~\ref{lemma_boundary}, the solution set $\{ \alpha_{\mu 1} \}$ exists if and only if $m = 2n_1+1$. The constant terms for the remaining equations (for $\mu = 2n_1+2,\dots,n$) precisely impose the condition
\[
    \{[a_{2n_1+2}], \dots, [a_n]\} \in Y_{\mathbf{n}',\mathbf{p}}(\tau). 
\]
The higher-order terms $\{ \alpha_{\mu k} \}$ can be solved inductively by Corollary~\ref{cor_boundary}. Finally, the number of such degenerate solutions in $E_\tau^n$ is given by
\[
\begin{split}
    \binom{n}{2n_1+1} (2n_1+1)!  (n-2n_1-1)! & \deg ( \sigma_{\mathbf{n}',\mathbf{p}}: Y_{\mathbf{n}',\mathbf{p}}(\tau) \to E_\tau ) 
    \\
     = n! & \deg ( \sigma_{\mathbf{n}',\mathbf{p}}: Y_{\mathbf{n}',\mathbf{p}}(\tau) \to E_\tau ).
\end{split}
\]
Here, the binomial coefficient corresponds to the choice of the $2n_1+1$ coordinates converging to $p_1$, the $(2n_1+1)!$ factor corresponds to the number of valid leading-order solutions $\{ \alpha_{\mu 1} \}$, and the final term corresponds to the ordered degree of $\sigma_{\mathbf{n}',\mathbf{p}}$ on $E^{n-2n_1-1}_\tau$. 

The proof generalizes to an arbitrary subset $J = \{ 1,\dots,l \}$. To correctly isolate the non-degenerate solutions, we derive the degree formula in the inclusion-exclusion form. We complete the proof.
\end{proof}

\begin{example} 
In the case of classical Lam\'e equation, i.e. $r=1$, $p=0$, $\mathbf{n} = n \in \mathbb{N}$, the degree of the addition map is $\binom{n+1}{2}$. It is consistent with the result in \cite[Theorem~2.4]{Lin_Wang_2017} which was obtained there by applying Theorem of the cube.
\end{example}

The natural map $\pi : \overline{\mathcal{Y}}_{\mathbf{n},\mathbf{p}}(\tau) \rightarrow \overline{Y}_{\mathbf{n},\mathbf{p}}(\tau)  $ is an isomorphism on the interior and finite on the boundary. Theorem~\ref{thm_deg} can be equivalently formulated on $\overline{\mathcal{Y}}_{\mathbf{n},\mathbf{p}}(\tau)$. 

\begin{definition} Fix $c \in E_\tau$. We define 
\[
\begin{split}
    \overline{Y}_{\mathbf{n},\mathbf{p}}^c(\tau) &:= \overline{Y}_{\mathbf{n},\mathbf{p}}(\tau) \ \cap \ \sigma^{-1}_{\mathbf{n},\mathbf{p}}( c );
    \\
    \overline{\mathcal{Y}}_{\mathbf{n},\mathbf{p}}^c(\tau) &:= \overline{\mathcal{Y}}_{\mathbf{n},\mathbf{p}}(\tau) \ \cap \ \pi^{-1}\sigma^{-1}_{\mathbf{n},\mathbf{p}}( c ).
\end{split}
\]
\end{definition}

\subsection{Flat family of GLC and BGG category}
In this subsection, we state the main theorem of this paper and its direct corollaries. The proof is involved and deferred to the next subsection.

\begin{theorem} \label{thm:universal_flatness}
Fix a total weight $n \in \mathbb{N}$. There exists a flat family $\overline{\mathcal{Y}}(\tau)$ over the space of weights and singularities:
\[
    \overline{\mathcal{Y}}(\tau) \rightarrow \Big\{ \mathbf{n} \in \mathbb{C}^r \ \Big| \ \sum_i n_i = n \Big\} \times \mathbb{C}^r
\]
whose fibers satisfy the following:
\begin{enumerate}
    \item Over the generic locus, the fiber is the generalized Lam\'e curve $\overline{\mathcal{Y}}_{\mathbf{n},\mathbf{p}}(\tau)$.
    \item Over the special locus, the fiber is uniquely determined by the tensor algebra of $\mathfrak{sl}_2(\mathbb{C})$-modules within the BGG category $\mathcal{O}$.
\end{enumerate}
\end{theorem}

We explain the construction of the special fiber using the language of BGG category $\mathcal{O}$ below.

\begin{definition}[The Category $\mathcal{O}$ Dictionary] \label{def:BGG}
To each singularity $p_i$ with weight $n_i$, we assign the highest weight module $V_{2n_i}$, where
\[
    V_{\lambda} :=
    \begin{cases}
        M_{\lambda} & \text{if } \lambda \notin \mathbb{Z}_{\geq 0} \quad \text{(Verma module)}, \\
        L_{\lambda} := M_{\lambda}/M_{-\lambda-2} & \text{if } \lambda \in \mathbb{Z}_{\geq 0} \quad \text{(finite dimensional quotient)}.
    \end{cases}
\]
We associate the generalized Lam\'e curve $\overline{\mathcal{Y}}_{\mathbf{n},\mathbf{p}}(\tau)$ with the tensor product $V_{2n_1} \otimes \dots \otimes V_{2n_r}$. 
\end{definition}

The special fiber can be iteratively constructed using two algebraic rules derived from category $\mathcal{O}$. To explicitly track the components during degenerations, we consider the slice with fixed generic $c$. In all formulas below, the total degree of the addition map is conserved.

\vspace{0.5em}
\noindent \textbf{1. Decomposition rule:} For $n_1 \in \frac{1}{2}\mathbb{Z}_{\geq 0}$, the boundary splitting corresponds to the BGG resolution:
\[
    [M_{2n_1}] = [L_{2n_1}] + q^{2n_1+1} [M_{-2n_1-2}].
\]
The formal grading variable $q$ tracks the collision multiplicity, where $q^k$ corresponds geometrically to $\{p^k\}$.
\begin{theorem}[Decomposition theorem] \label{thm_decom} 
Assume $n_1 \in \tfrac{1}{2}\mathbb{Z}_{\geq 0}$, and $n_i$ are generic for $i \ge 2$. We have:
\[  
\begin{split}
    \lim_{\delta \rightarrow 0}\overline{\mathcal{Y}}_{(n_1+\delta,n_2-\delta,n_3,\dots,n_r),\mathbf{p} }^c(\tau) &= \overline{\mathcal{Y}}_{\mathbf{n},\mathbf{p} }^c(\tau) \\
    &\quad \cup \ \{p_1^{2n_1+1}\} \times \overline{\mathcal{Y}}^{c'}_{(-n_1-1,n_2,\dots,n_r),\mathbf{p}} (\tau),
\end{split}
\]
where $c' := c - (2n_1+1) [p_1]$. We adopt the following two conventions:
\[
    \overline{\mathcal{Y}}_{(0,n_2,\dots,n_r), \mathbf{p}}^c(\tau) 
    := \overline{\mathcal{Y}}_{(n_2,\dots,n_r), (p_2,\dots,p_r)}^c(\tau),
\]
and
\[
    \overline{\mathcal{Y}}_{\mathbf{m},\mathbf{p}}(\tau) := \emptyset, \quad \text{if } \sum_i m_i < 0.
\]
\end{theorem}

\begin{remark}
We briefly recall our convention for the symmetric product to clarify the ambient spaces. The operation $\times$ denotes the natural multiplication of unordered symmetric classes. 

Specifically, any element in the boundary stratum on the right-hand side,
\[
    [(\mathbf{a},h)] \in \{p_1^{2n_1+1}\} \times \overline{\mathcal{Y}}^{c'}_{(-n_1-1,n_2,\dots,n_r),\mathbf{p}} (\tau),
\]
corresponds to a root configuration $\{ \mathbf{a} \} = \{p_1^{2n_1+1}\} \times \{\mathbf{a}'\}$. Here, $2n_1+1$ roots have collided at the pole $p_1$, while the remaining moving component $[(\mathbf{a}', h)]$ lives in the moduli space $\overline{\mathcal{Y}}^{c'}_{(-n_1-1,n_2,\dots,n_r),\mathbf{p}} (\tau)$. Because the total degree is conserved, $(2n_1+1) + (n - 2n_1 - 1) = n$, both components on the right-hand side naturally embed into the same ambient space $\mathcal{P}_n$ as the left-hand side.
\end{remark}

\vspace{0.5em}
\noindent \textbf{2. Generic tensor rule:} For generic weights $n_1,n_2,n_1+n_2 \in \mathbb{C} \setminus \{ \tfrac{1}{2} \mathbb{Z}_{\geq 0} \}$, the collision of two singularities corresponds to the Verma tensor decomposition:
\[
    [M_{2n_1}] \otimes [M_{2n_2}] = \sum_{k=0}^\infty q^k [M_{2n_1+2n_2-2k}].
\]

For $\omega \in \Lambda_\tau$, let $\mathbf{p}_{12, \epsilon}^\omega = (p_{12} + \omega + \epsilon, p_{12} - \epsilon, p_3, \dots, p_r)$, and define the shift operator $S_v$:
\[
\begin{split}
     S_v : \mathcal{P}_n &\rightarrow \mathcal{P}_n \\
     [(\mathbf{a},h)] &\mapsto [(\mathbf{a},h+v)].
\end{split}
\]

\begin{theorem}[Degeneration theorem - Generic Case] \label{thm:dege_generic}
Assume $n_1, n_2, n_1+n_2 \in \mathbb{C} \setminus \{ \tfrac{1}{2} \mathbb{Z}_{\geq 0} \}$. We have:
\[
    \lim_{\epsilon \to 0} \overline{\mathcal{Y}}_{\mathbf{n}, \mathbf{p}_{12, \epsilon}^\omega}^c(\tau) =
    \bigcup_{k=0}^{n} \{p_{12}^k\} \times S_{-n_1 \eta(\omega)} \Big( \overline{\mathcal{Y}}^{c_k}_{\mathbf{n}_{12}^{(k)},\mathbf{p}_{12}}(\tau) \Big),
\]
where $c_k = c - k [p_{12}]$, $\mathbf{n}_{12}^{(k)} = (n_1+n_2-k,n_3,\dots,n_r)$, and $\mathbf{p}_{12} = (p_{12},p_3,\dots,p_r)$.
\end{theorem}

All special degeneration limits are obtained by iteratively applying these two algebraic rules. By evaluating the generic tensor product and replacing any $m \in \tfrac{1}{2}\mathbb{Z}_{\geq 0}$ with its irreducible quotient $[L_{2m}]$, we systematically recover the following non-generic geometric limits.

\begin{corollary}[Degeneration theorem - Special cases] \label{cor:dege_special}
With the same convention as in Theorem~\ref{thm_decom}, we have:
\begin{enumerate}[label=(\arabic*)]
\item If $n_1 \in \tfrac{1}{2} \mathbb{Z}_{\geq 0}$ and $n_2 \in \mathbb{C} \setminus \{ \tfrac{1}{2} \mathbb{Z}_{\geq 0} \}$,
\[
    \lim_{\epsilon \to 0} \overline{\mathcal{Y}}_{\mathbf{n}, \mathbf{p}_{12, \epsilon}^\omega}^c(\tau) = \bigcup_{k=0}^{2n_1} \{p_{12}^k\} \times S_{-n_1 \eta(\omega)} \Big( \overline{\mathcal{Y}}^{c_k}_{\mathbf{n}_{12}^{(k)},\mathbf{p}_{12}}(\tau) \Big).
\]
\item If $n_1, n_2 \in \mathbb{C} \setminus \{ \tfrac{1}{2} \mathbb{Z}_{\geq 0} \}$ and $n_1+n_2 \in \tfrac{1}{2} \mathbb{Z}_{\geq 0}$, 
\[
\begin{split}
    \lim_{\epsilon \to 0} \overline{\mathcal{Y}}_{\mathbf{n}, \mathbf{p}_{12, \epsilon}^\omega}^c(\tau) &= \bigcup_{k=0}^{n} \{p_{12}^k\} \times S_{-n_1 \eta(\omega)} \Big( \overline{\mathcal{Y}}^{c_k}_{\mathbf{n}_{12}^{(k)},\mathbf{p}_{12}}(\tau) \Big) 
    \\
    &\quad \cup \bigcup_{k = \lfloor n_1+n_2+2 \rfloor}^{2n_1+2n_2+1} \{p_{12}^k\} \times S_{-n_1 \eta(\omega)} \Big( \overline{\mathcal{Y}}^{c_k}_{\mathbf{n}_{12}^{(k)},\mathbf{p}_{12}}(\tau) \Big).
\end{split}
\]
\item If $n_1 \in \tfrac{1}{2} \mathbb{Z}_{\geq 0}$ and $n_2 \in -\tfrac{1}{2} \mathbb{Z}_{\geq 0}$ with $n_1+ n_2 \in \tfrac{1}{2} \mathbb{Z}_{\geq 0}$, 
\[
\begin{split}
    \lim_{\epsilon \to 0} \overline{\mathcal{Y}}_{\mathbf{n}, \mathbf{p}_{12, \epsilon}^\omega}^c(\tau) &= \bigcup_{k=0}^{2n_1} \{p_{12}^k\} \times S_{-n_1 \eta(\omega)} \Big( \overline{\mathcal{Y}}^{c_k}_{\mathbf{n}_{12}^{(k)},\mathbf{p}_{12}}(\tau) \Big) 
    \\
    &\quad \cup \bigcup_{k = \lfloor n_1+n_2+2 \rfloor}^{2n_1+2n_2+1} \{p_{12}^k\} \times S_{-n_1 \eta(\omega)} \Big( \overline{\mathcal{Y}}^{c_k}_{\mathbf{n}_{12}^{(k)},\mathbf{p}_{12}}(\tau) \Big).
\end{split}
\]
\item If $n_1, n_2 \in \tfrac{1}{2} \mathbb{Z}_{\geq 0}$,
\[
    \lim_{\epsilon \to 0} \overline{\mathcal{Y}}_{\mathbf{n}, \mathbf{p}_{12, \epsilon}^\omega}^c(\tau) = \bigcup_{k=0}^{2\min(n_1,n_2)} \{p_{12}^k\} \times S_{-n_1 \eta(\omega)} \Big( \overline{\mathcal{Y}}^{c_k}_{\mathbf{n}_{12}^{(k)},\mathbf{p}_{12}}(\tau) \Big).
\]
\end{enumerate}
\end{corollary}
\begin{proof}
By the Category $\mathcal{O}$ dictionary, evaluating the geometric limits reduces strictly to computing the tensor product of $\mathfrak{sl}_2(\mathbb{C})$-modules within the Grothendieck ring $R[q]$. 

We demonstrate case (4) as an example. Because $n_1, n_2 \in \frac{1}{2} \mathbb{Z}_{\geq 0}$, the associated modules are finite-dimensional irreducible representations $L_{2n_1}$ and $L_{2n_2}$. Their tensor product is given by the Clebsch-Gordan decomposition:
\[
    [L_{2n_1}] \otimes [L_{2n_2}] = \sum_{j=0}^{2\min(n_1,n_2)} q^j [L_{2n_1+2n_2-2j}].
\]
Translating the right-hand side back to geometry via the dictionary yields exactly the union in (4). Cases (1), (2), and (3) follow by identical algebraic substitutions.
\end{proof}

Iterating degeneration formula can be complicated. The BGG dictionary bypasses this: for example, on the most degenerate fiber, the component multiplicities follow directly from the character formula.
\begin{corollary} \label{Cor:special_fiber_p=0}
Let $\mathbf{0} = (0,\dots,0)$. We have:
\[
    \lim_{\mathbf{p} \to \mathbf{0}} \overline{\mathcal{Y}}_{\mathbf{n}, \mathbf{p}}(\tau) = \bigcup_{k=0}^{n}  m_k(\mathbf{n}) \Big( \{0^k\} \times  \overline{\mathcal{Y}}_{n-k, 0}(\tau) \Big),
\]  
where 
\[
    m_k(\mathbf{n}) := {\rm Coeff}\Bigg( x^k; \frac{ \prod_{i \in I} (1-x^{2n_i+1}) }{(1-x)^{r-1}} \Bigg), \quad I := \Big\{ i\ \Big| \ n_i \in \tfrac{1}{2}\mathbb{Z}_{\geq 0} \Big\}. 
\]
\end{corollary}

\begin{proof}
By the dictionary, this limit is equivalent to the following decomposition:
\[
    [V_{2n_1}]\otimes \dots \otimes [V_{2n_r}] = \sum_{k=0}^{\infty} m_k(\mathbf{n}) q^k [M_{2n-2k}],
\] 
Note that using $[M_{2n-2k}]$ or $[L_{2n-2k}]$ on the right hand side does not change the conclusion, since the decomposition on the geometry side is truncated at weight 0.

The coefficients are computed using standard $\mathfrak{sl}_2(\mathbb{C})$ character formulas. Let $x = e^{-2}$ be the formal variable tracking the weight shifts. Recall that:
\[
\begin{split}
    {\rm ch}(M_{2n_i}) &= \frac{e^{2n_i}}{1 - e^{-2}} = \frac{e^{2n_i}}{1 - x};
    \\
    {\rm ch}(L_{2n_j}) &= {\rm ch}(M_{2n_j}) - {\rm ch}(M_{-2n_j-2}) = \frac{e^{2n_j}(1 - x^{2n_j+1})}{1 - x}.
\end{split}
\]
Taking the tensor product of all $r$ modules, we have:
\[
    {\rm ch}(V_{2n_1} \otimes \dots \otimes V_{2n_r}) = e^{2n} \frac{\prod_{j \in I} (1 - x^{2n_j+1})}{(1 - x)^r}.
\]
Assuming the tensor product decomposes as $\sum m_k(\mathbf{n}) [M_{2n - 2k}]$, the character of this sum is:
\[
    \sum_{k=0}^{\infty} m_k(\mathbf{n}) {\rm ch}(M_{2n - 2k}) = \sum_{k=0}^{\infty} m_k(\mathbf{n}) \frac{e^{2n - 2k}}{1 - x} = \frac{e^{2n}}{1 - x} \sum_{k=0}^{\infty} m_k(\mathbf{n}) x^k.
\]
Equating the two character expressions yields the generating function for $m_k(\mathbf{n})$. 
\end{proof}

\begin{corollary}
The generalized Lam\'e curve $\overline{\mathcal{Y}}_{\mathbf{n},\mathbf{p}}(\tau)$ is connected.
\end{corollary}

\begin{proof}
By the universal flatness theorem, there exists a flat family over the parameter space whose special fiber over $\mathbf{p} = \mathbf{0}$ is:
\[
    \bigcup_{k=0}^{n}  m_k(\mathbf{n}) \Big( \{0^k\} \times  \overline{\mathcal{Y}}_{n-k, 0}(\tau) \Big).
\]
This special fiber is connected because all of its components contain the shared point $[(\mathbf{0},\infty)]$. Since flatness preserves connectedness, the generic fiber $\overline{\mathcal{Y}}_{\mathbf{n},\mathbf{p}}(\tau)$ is also connected. 
\end{proof}

\subsection{Proof of flatness for $\overline{\mathcal{Y}}(\tau)$}
This subsection consists of two parts. In the first part, we prove Theorem~\ref{thm:dege_generic}. In the second part, we show that the constructive proof in the first part can be enhanced to construct the flat family. 

Theorem~\ref{thm:dege_generic} naturally descends to a degeneration on the underlying generalized Lam\'e curve $\overline{Y}_{\mathbf{n},\mathbf{p}}(\tau)$. Because points in $\overline{Y}_{\mathbf{n},\mathbf{p}}(\tau)$ forget $h$, the shift operator $S_v$ becomes trivial. 

\begin{theorem}[Degeneration theorem - generic case] \label{thm:underlying_dege_generic}
For generic $c\in E_\tau$, assume $n_1, n_2, n_1+n_2 \in \mathbb{C} \setminus \{ \tfrac{1}{2} \mathbb{Z}_{\geq 0} \}$. We have:
\[
    \lim_{\epsilon \to 0} \overline{Y}_{\mathbf{n}, \mathbf{p}_{12, \epsilon}}^c(\tau) =
    \bigcup_{k=0}^{n} \{[p_{12}]^k\} \times \overline{Y}^{c_k}_{\mathbf{n}_{12}^{(k)},\mathbf{p}_{12}}(\tau),
\]
where $c_k = c - k [p_{12}]$, and $\mathbf{p}_{12,\epsilon} = \mathbf{p}_{12,\epsilon}^0 = (p_{12}+\epsilon, p_{12}-\epsilon, p_3, \dots, p_r)$.
\end{theorem}
To prove Theorem~\ref{thm:underlying_dege_generic}, consider the parametrization
\[
\begin{split}
    a_\mu(\epsilon) &= p_{12} + \sum_{j \geq 1} \alpha_{\mu j} \epsilon^j, \mbox{ for $\mu =1,\dots,k$}.
\end{split}
\]
Let $p_i(\epsilon)$ be the $i$-th component of $\mathbf{p}_{12,\epsilon}$.
The solvability of the following system:
\begin{equation} \label{eqn:degeneration}
\begin{cases}
    &\displaystyle \sum_{\mu=1}^n a_{\mu}(\epsilon) = c 
    \\
    &\displaystyle \sum_{i = 1}^{r}  \sum_{\substack{\nu = 1\\ \nu\neq \mu}}^n  
    n_i\Big( \zeta(a_\mu(\epsilon) - a_\nu(\epsilon)) - \zeta(a_\mu(\epsilon) - p_i(\epsilon)) +  \zeta(a_\nu(\epsilon) - p_i(\epsilon)) \Big) = 0
\end{cases}
\end{equation}
will imply the existence of the component
\[
     \{[p_{12}]^k\} \times \Big( \overline{Y}^{c_k}_{\mathbf{n}_{12}^{(k)},\mathbf{p}_{12}}(\tau) \Big),
\]
on the special fiber.

The principal part (coefficients of $\epsilon^{-1}$) of the system~\eqref{eqn:degeneration} gives
\[
\begin{cases}
    &\displaystyle\sum_{\mu=1}^k \Big( \frac{n_1}{\alpha_{\mu 1} -1} + \frac{n_2}{\alpha_{\mu 1} + 1} \Big) = 0 ;
    \\
    &\displaystyle\sum_{\nu=1,\neq \mu}^k \frac{n}{\alpha_{\mu 1}-\alpha_{\nu 1}} - (n-1) \Big( \frac{n_1}{\alpha_{\mu 1} -1} + \frac{n_2}{\alpha_{\mu 1} +1} \Big) 
    \\
    &\displaystyle\hspace{50mm}+ \sum_{\nu=1,\neq \mu}^k \Big( \frac{n_1}{\alpha_{\nu 1} -1} + \frac{n_2}{\alpha_{\nu 1}+1} \Big) =0.
\end{cases}
\]
Equivalently, we rewrite it in a more symmetric form
\begin{equation} \label{eqn:principal}
\begin{split}
    g_\mu(\alpha_{11},\dots,\alpha_{k1} ):=\sum_{\nu=1,\neq \mu}^k\frac{1}{\alpha_{\mu 1} - \alpha_{\nu 1}} + \frac{-n_1}{\alpha_{\mu 1} - 1} + \frac{-n_2}{\alpha_{\mu 1} + 1}=0,
    \\
    \mbox{ for $\mu =1,\dots,k.$}
\end{split}
\end{equation}
For notation simplicity, we write $\alpha_{\mu}$ for $\alpha_{\mu 1}$, $x$ (resp. $y$) for $n_1$ (resp. $n_2$). The existence of the first order deformation (i.e. the existence of $\alpha_{\mu}$) is equivalent to the solvability of the system~\eqref{eqn:principal}.

To formulate the result, let
\[
\begin{split}
    E_k &:= \{ x=0,\frac{1}{2},\dots, \frac{k-1}{2} \} \cup \{y=0,\frac{1}{2},\dots,\frac{k-1}{2}\} 
    \\
    & \hspace{30mm}\cup \{ x + y= \frac{k-1}{2}, \frac{k}{2},\dots, k-1 \} \subset \mathbb{C}^2;
    \\
    F_k &:= \{ x, y = 0,\frac{1}{2},\dots, \frac{k-1}{2} \}  \cap \{ x + y \geq \frac{k-1}{2} \} \subset \mathbb{C}^2.
\end{split}
\]
\begin{center}
\begin{figure}[ht] 
\begin{tikzpicture}[scale=3]
    \foreach \x in {0.5,1,1.5,2,2.5} {
        \draw[gray!50] (\x,0) -- (\x,2);
    }
    \foreach \y in {0,0.5,1,1.5} {
        \draw[gray!50] (0,\y) -- (3,\y);
    }
    
    \draw[->] (0,0) -- (3,0);
    \draw[->] (0,0) -- (0,2);
    
    \node[below] at (0.5,-0.05) {$x=0.5$};
    \node[below] at (1,-0.05) {1};
    \node[below] at (1.5,-0.05) {1.5};
    
    \node[left] at (-0.05,0.5) {$y=0.5$};
    \node[left] at (-0.05,1) {1};
    \node[left] at (-0.05,1.5) {1.5};
    
    \fill[blue] (0, 1.5) circle (0.03);
    \fill[blue] (0.5,1.5) circle (0.03);
    \fill[blue] (1,1.5) circle (0.03);
    \fill[blue] (1.5,1.5) circle (0.03);
    \fill[blue] (0.5, 1) circle (0.03);
    \fill[blue] (1, 1) circle (0.03);
    \fill[blue] (1.5, 1) circle (0.03);
    \fill[blue] (1, 0.5) circle (0.03);
    \fill[blue] (1.5, 0.5) circle (0.03);
    \fill[blue] (1.5, 0) circle (0.03);
    
    \draw[blue,thick] (0,-0.05) -- (0,1.9);
    \draw[blue,thick] (0.5,-0.05) -- (0.5,1.9);
    \draw[blue,thick] (1,-0.05) -- (1,1.9);
    \draw[blue,thick] (1.5,-0.05) -- (1.5,1.9);
    
    \draw[blue,thick] (-0.05,0) -- (2.9,0);
    \draw[blue,thick] (-0.05,0.5) -- (2.9,0.5);
    \draw[blue,thick] (-0.05,1) -- (2.9,1);
    \draw[blue,thick] (-0.05,1.5) -- (2.9,1.5);
        
    \draw[blue,thick] (-0.25,1.75) -- (1.75,-0.25);
    \draw[blue,thick] (0.25,1.75) -- (1.75,0.25);
    \draw[blue,thick] (0.75,1.75) -- (1.75,0.75);
    \draw[blue,thick] (1.25,1.75) -- (1.75,1.25);

    \node[right] at (3,0) {$\{1\}^1$};
    \node[right] at (3,0.5) {$\{1\}^2$};
    \node[right] at (3,1) {$\{1\}^3$};
    \node[right] at (3,1.5) {$\{1\}^4$};

    \node[above] at (0,2) {$\{(-1)\}^1$};
    \node[above] at (0.5,2) {$\{(-1)\}^2$};
    \node[above] at (1,2) {$\{(-1)\}^3$};
    \node[above] at (1.5,2) {$\{(-1)\}^4$};

    \node[right] at (1.75,-0.25) {$\{\infty\}^4$};
    \node[right] at (1.75,0.25) {$\{\infty\}^3$};
    \node[right] at (1.75,0.75) {$\{\infty\}^2$};
    \node[right] at (1.75,1.25) {$\{\infty\}^1$};
\end{tikzpicture}
\caption{the blue lines represent $E_4$, the blue buttons represent $F_4$. \\ The label $\{a\}^m$ on each blue line corresponds to the solution set which contains $m$'s many of $a$. See Remark~\ref{Rmk_dege_loci}.}
\label{figure_E_F}
\end{figure}
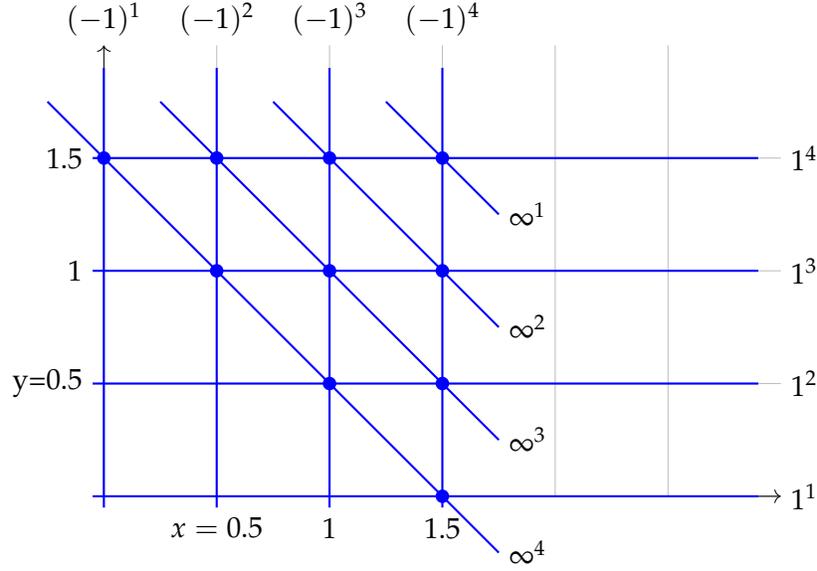
\end{center}

We have
\begin{lemma} \label{lemma_principal} For $x,y \in \mathbb{C}^2 \setminus E_k$, up to permutation, the system 
\[
    g_{\mu}(\alpha) = \sum_{\nu=1,\neq \mu}^k\frac{1}{\alpha_{\mu } - \alpha_{\nu }} + \frac{-x}{\alpha_{\mu} - 1} + \frac{-y}{\alpha_{\mu} + 1}=0,
    \mbox{ for $\mu =1,\dots,k.$}
\]
has an unique solution $\{ \alpha_\mu \}_{\mu=1}^k$ with $\alpha_\mu \neq \alpha_\nu$.

Moreover, the Jacobian of $g$ evaluated at the corresponding solution, $J_g(\alpha)$, has rank $k$. 
\end{lemma}
\begin{proof}
The proof strategy is the same as in Lemma~\ref{lemma_boundary}. We repeat here for completeness and mention the differences.

Assume a solution $\{ \alpha_1,\dots,\alpha_k \}$ exists with $\alpha_\mu \neq \alpha_\nu$,
let $p(z) = \prod_{\mu=1}^k (z- \alpha_\mu)$ be a degree $k$ polynomial with roots given by $\{ \alpha_\mu \}$. Then
\[
    \frac{p''(\alpha_\mu)}{2p'(\alpha_\mu)} = \sum_{\nu=1,\neq \mu}^k \frac{1}{\alpha_\mu - \alpha_\nu}.
\]
The system can be rewritten as
\[
    \frac{p''(z)}{2p'(z)} + \frac{-x}{z-1} + \frac{-y}{z +1} = 0, \text{ for }z=\alpha_1,\dots,\alpha_k.
\]
$p(z)$ satisfies the following equation
\[
    \mathcal{L}[q] := \Big( (z^2-1)\frac{d^2}{dz^2} + 2(-x(z+1)-y(z-1)) \frac{d}{dz} + \lambda \Big) p(z) =0,
\]
To have a degree $k$ polynomial solution, the choice of $\lambda = k(2x+2y+1-k)$ is unique by comparing the coefficient of $z^k$.

The solution takes the form of the Jacobi polynomial $P_k^{(-2x-1, -2y-1)}(z)$. Explicitly, up to scaling,
\[
\begin{split}
    p(z)
    =& P_k^{(-2x-1, -2y-1)}(z) 
    \\
    :=& \frac{1}{2^k k!} \sum_{m=0}^k \binom{k}{m} (k-2y-1)_m (k-2x-1)_{k-m} (z-1)^m (z+1)^{k-m}.
\end{split}
\]
Here, $(a)_n = a(a-1)\dots(a-n+1)$ is the falling factorial.

Now we show that the non-generic locus is $E_k$. Note that 
\[
\begin{split}
P_k^{(-2x-1, -2y-1)}(1) = 0  &\Leftrightarrow (k-2x-1)_k = 0 \Leftrightarrow  x \in \left\{ 0, \frac{1}{2}, \dots, \frac{k-1}{2} \right\};
\\
P_k^{(-2x-1, -2y-1)}(-1) = 0  &\Leftrightarrow (k-2y-1)_k = 0  \Leftrightarrow  y \in \left\{ 0, \frac{1}{2}, \dots, \frac{k-1}{2} \right\};
\\
\text{Coeff}(z^k; P_k^{(-2x-1, -2y-1)})=0 &\Leftrightarrow (2k-2x-2y-2)_k = 0 \;  \Leftrightarrow 
\\
&\hspace{85pt}  x+y \in \left\{ \frac{k-1}{2}, \frac{k}{2}, \dots, k-1 \right\}.
\end{split}
\]
The three conditions correspond to the case that at least one $\alpha_\mu$ equals to $1$, $-1$, and $\infty$ respectively. The union of these loci gives $E_k$. 

For the condition $\alpha_\mu$ mutually distinct, note that the discriminant of $P_k^{(-2x-1, -2y-1)}(z)$ is proportional to
\[
    \prod_{j=1}^k (j-2x-1)^{j-1} (j-2y-1)^{j-1} (k+j-2x-2y-2)^{k-j},
\]
whose zero locus is strictly contained in $E_k$. This finish the proof of the first statement.

For the rank statement, given $v = (v_1,\dots,v_k) \in \mathbb{C}^k$, let $\delta_v p$ be the infinitesimal deformation of $p$. Explicitly,
\[
    \delta_v p := \frac{d}{d\epsilon} \Big|_{\epsilon =0} \prod_{\mu=1}^k (z-\alpha_\mu - \epsilon v_\mu) = - p(z) \sum_{\mu=1}^k \frac{v_\mu}{z-\alpha_\mu} \in \mathbb{C}[z]_{\leq k-1}.
\]
Unlike Lemma~\ref{lemma_boundary}, the differential operator $\mathcal{L}$ does not depend on the roots $\alpha_\mu$, as the constant $\lambda = k(2x+2y+1-k)$ is entirely fixed by $x,y$, and $k$. Thus, the Fr\'echet derivative of $\mathcal{L}$ at $p$ is simply $\mathcal{L}$ itself. 

Note that $v \in \ker (J_g(\alpha))$ if and only if $\mathcal{L}[\delta_vp]=0$. Evaluate $\mathcal{L}$ on the monomial:
\[
    \mathcal{L}[z^d] = (d-k)(d+k-1-2x-2y) z^d + \text{ lower degree terms}.
\]
This implies that $\mathcal{L}$ induces an isomorphism on $\mathbb{C}[z]_{\leq k-1}$. The condition $\mathcal{L}[\delta_v p] = 0$ implies that $\delta_v p = 0$. Since $\delta_v p(\alpha_\mu) = -v_\mu p'(\alpha_\mu)$ and the roots are distinct ($p'(\alpha_\mu) \neq 0$), we conclude that $v = 0$. Thus, $\ker J_f(\alpha) = \{0\}$, meaning the Jacobian has full rank $k$.
\end{proof}

\begin{remark} \label{Rmk:Heine_Stieltjes}
The system~\ref{eqn:principal} has been studied in \cite{Varchenko} to find critical points of the product of powers of linear functions. They show that for generic case, the counting is the same as the Euler characteristic of the complement of the configuration of the union of hyperplanes defined by those linear functions. Since the Jacobian of the solution and the degenerate locus is also important for us, we give an alternate proof.
\end{remark}

\begin{remark} \label{Rmk_dege_loci}
The degenerate locus can be further studied. 

When $(x,y) \in E_k \setminus F_k$, the solution only exists in the limiting sense, since some of $\alpha_\mu$ takes values in $\{1,-1,\infty\}$, see Figure~\ref{figure_E_F}. The value and multiplicity follows from identities of Jacobi polynomial. For example, for $x = \frac{l-1}{2}$, we have
\[
    P_k^{-l, -2y-1} (z) = C (z-1)^l P_{k-l}^{l, -2y-1}(z), \quad C = \frac{(k-2y-1)_l}{2^l (k)_l}.
\]
yielding $l$ values of $\mu$ such that $\alpha_\mu = 1$. Similar discussion works for $y=c_1$ and $x+y=c_2$ by M\"obius transform.

When $(x,y) \in F_k$, there are infinitely many solutions. For $(x,y) = (x_0+u\epsilon, y_0+v\epsilon)$ with $(x_0,y_0) \in F_k$, 
\[
    P_k^{-2x-1,-2y-1}(z) = \epsilon F(u,v,z) + O(\epsilon^2),
\]
where $F(u,v,z)$ is a polynomial in $z$ of degree $k$ with coefficients linear functions of $u,v$. When taking the limit $\epsilon\rightarrow 0$, it suffices to look at the leading term $F(u,v,z)$. As a result, the solution forms a $\mathbb{P}^1$ family with $\mathbb{P}^1 \cong E \subset {\rm Bl}_{(x_0,y_0)} \mathbb{C}^2 $, the exceptional locus of the blowup. Lemma~\ref{Lemma_sol_x} is a special case for $(x,y) = ( \frac{n-1}{2},0 )$. 
\end{remark}
\vspace{3mm}

We proceed to solve the higher order deformation.
\begin{proof}[Proof of Theorem~\ref{thm:underlying_dege_generic}] 
We claim that given a point in the special fiber
\[
    \{[\mathbf{a}]\} = \{[p_{12}]^k, [a_{k+1}], \dots, [a_n]\} \in \{[p_{12}]^k\} \times \overline{Y}_{\mathbf{n}_{12}^{(k)}, \mathbf{p}_{12}}^{c_k}(\tau),
\]
there exists an explicit deformation in $\overline{Y}_{\mathbf{n},\mathbf{p}_{12,\epsilon}}^c(\tau)$. We construct this deformation as a power series:
\[
    a_\mu(\epsilon) = a_\mu + \sum_{m \geq 1} \alpha_{\mu m}\epsilon^m, \quad \mu=1,\dots,n.
\]
To show that the coefficients $\{ \alpha_{\mu m} \}$ can be uniquely solved inductively, we separate the variables into the deforming part and the fixed part:
\[
\begin{split}
    \mathbf{a}_{\rm deform} &= (p_{12},\dots,p_{12}), \quad \boldsymbol{\alpha}_{{\rm deform},m} = (\alpha_{1m}, \dots,\alpha_{km});
    \\
    \mathbf{a}_{\rm fix} &= (a_{k+1},\dots,a_{n}), \quad \boldsymbol{\alpha}_{{\rm fix},m} = (\alpha_{k+1, m}, \dots, \alpha_{nm}). 
\end{split}
\]

\vspace{0.5em}
\noindent\textbf{Base cases:}
The coefficient of $\epsilon^{-1}$ yields the system in Lemma~\ref{lemma_principal}, which uniquely determines $\boldsymbol{\alpha}_{{\rm deform }, 1}$. 
The coefficient of $\epsilon^0$ reduces to the linear system: 
\[
    J_g( \boldsymbol{\alpha}_{{\rm deform},1} ) \, \boldsymbol{\alpha}_{{\rm deform},2}^T = \mathbf{0}.
\]
Because the Jacobian $J_g$ has maximal rank, this forces $\boldsymbol{\alpha}_{{\rm deform},2} = \mathbf{0}$. 

\vspace{0.5em}
\noindent\textbf{The inductive step:}
For the fixed coordinates, let $M = M_{\mathbf{n}_{12}^{(k)}, \mathbf{p}_{12}}( \mathbf{a}_{\rm fix} )$ be the $(n-k)\times (n-k)$ matrix whose entries for $\mu, \nu \in \{k+1, \dots, n\}$ are defined by:
\[
\begin{split}
     M_{\mu \mu} &= - \sum_{\nu=k+1,\neq \mu}^n \wp(a_\mu - a_\nu) + \frac{n-k-1}{n-k} \sum_{i=2}^r n_i' \wp(a_{\mu} - p_i), 
    \\
    M_{\mu \nu} &=  \wp(a_\mu - a_\nu) - \frac{1}{n-k} \sum_{i=2}^r n_i' \wp(a_{\mu} - p_i), \quad (\mu \neq \nu).
\end{split}
\]
For $m \geq 1$, analyzing the coefficient of $\epsilon^m$ yields a block-triangular system for the higher-order terms:
\begin{equation} \label{eqn:inductive_block}
    \begin{pmatrix}
        J_g( \boldsymbol{\alpha}_{{\rm deform},1} ) & * \\
        \mathbf{0} & M
    \end{pmatrix}
    \begin{pmatrix}
        \boldsymbol{\alpha}_{{\rm deform},m+2}^T \\
        \boldsymbol{\alpha}_{{\rm fix},m}^T
    \end{pmatrix} 
    =  
    \begin{pmatrix}
        \mathbf{R}_{{\rm deform}, m}^T \\
        \mathbf{R}_{{\rm fix}, m}^T
    \end{pmatrix},
\end{equation}
where the right-hand side vectors consist of terms involving $\wp, \wp', g_2, g_3$ and the lower-order coefficients $\boldsymbol{\alpha}_{{\rm deform}, i < m+2}$ and $\boldsymbol{\alpha}_{{\rm fix}, j < m}$.

\vspace{0.5em}
\noindent\textbf{Rank argument and conclusion:}
To guarantee a unique solution, we must establish the rank of $M$. We claim that for a generic $\{[a_\mu]\}_{\mu=k+1}^n \in \overline{Y}_{\mathbf{n}_{12}^{(k)},\mathbf{p}_{12}}(\tau)$, the matrix $M$ has rank $n-k-1$. 

We verify this by degenerating $\{[a_{k+1}],\dots,[a_n]\} \to \{[p_1]^{k_1},\dots,[p_r]^{k_r}\} \in \partial\overline{Y}_{\mathbf{n}_{12}^{(k)},\mathbf{p}_{12}}(\tau)$ with a boundary parameter $s$. As established in the proof of Theorem~\ref{special-pt}, the principal part (the $s^{-1}$-coefficient) separates into $r$ blocks governed by Lemma~\ref{lemma_boundary}, along with global constraints that impose $r-1$ relations among the blocks. In total, the principal part has rank:
\[
    \sum_{i=1}^r (k_i -1) + (r-1) = n-k-1.
\]
This implies that $M$ has rank $n-k-1$.

The $1$-dimensional kernel of $M$ corresponds to uniform translations. On the boundary point, the kernel of the principal part of the $r$-block matrix is spanned by the vector $(\alpha_{\mu 1})$, where for each block $I_j$, the components $\alpha_{\mu 1}$ are the roots of the Laguerre polynomial $L_{k_j}^{(-2n_j-1)}\big(-2\lambda_j \frac{n_j}{k_j} z\big)$ from Lemma~\ref{lemma_boundary}. By Vieta's formulas, the sum of the roots for the $j$-th block is $\frac{k_j^2(2n_j - k_j + 1)}{2n_j\lambda_j}$. Substituting the global constraint $n_j \lambda_j = \frac{k_j}{n} \sum_{i=1}^r n_i \lambda_i$, the sum of all entries evaluates to
\[
    \sum_{\mu} \alpha_{\mu 1} = \frac{n}{2 \sum_{i=1}^r n_i \lambda_i} \sum_{j=1}^r k_j(2n_j - k_j + 1).
\]
Evaluating at the boundary point $\{[p_{12}]^{n-k}\}$, the sum simplifies to
\[
    \sum_{\mu} \alpha_{\mu 1} = \frac{n(n-k)(2n_1 - n + k + 1)}{2n_1\lambda_1} \neq 0.
\]
This implies that the sum of the components of the kernel vector of $M$ is generically non-zero, and that the condition $\sum_\mu a_\mu(\epsilon) = c$ strictly eliminates the uniform translation ambiguity. This guarantees that the deformation coefficients $\alpha_{\mu m}$ can be uniquely solved inductively.

Through this deformation construction, we obtain
\[
    \sum_{k=0}^{n} \binom{n+r-1-k}{r}_{\rm comb} = \binom{n+r}{r+1}
\]
many points in $\overline{Y}_{\mathbf{n},\mathbf{p}}^c(\tau)$ stratified by the index $k$. For generic $c$, these account for all points by the degree formula. This completes the proof.
\end{proof}

\begin{remark}
We note that the Decomposition Theorem~\ref{thm_decom} can be proved in the same way by explicit formal power series deformation construction. 

The special cases of the Degeneration Theorem~\ref{cor:dege_special} can, in principle, also be proved directly via deformation constructions. In such an explicit approach, the principal $\epsilon^{-1}$ system governing the leading-order coefficients corresponds to the degenerate loci of Lemma~\ref{lemma_principal} (see Remark~\ref{Rmk_dege_loci}). 

Rather than repeating these explicit local constructions for each case, we can rely only on the generic degeneration and the decomposition rules. They are precisely encoded by the BGG dictionary, iteratively applying these rules ensures that all special limits are systematically recovered within the exact same category.
\end{remark}

\vspace{5mm}
We prove that the form in the degeneration theorem with fix $c \in E$ can be enhanced to the flat family of the generalized Lam\'e curve.

We start with a general strategy to prove the flatness. The strategy is by checking the following geometric conditions. First we show that how these conditions is sufficient to prove the flatness.

Assume that $f: \mathcal{X} \rightarrow S$ be a family of curves, where $\mathcal{X}$ is a reduced scheme. Furthermore the following properties hold:
\begin{enumerate}
    \item At any point of $x \in \mathcal{X}_0 $, there exists a multi-variable power series deformation of the point $x$. The number of variables equal to $\dim \mathcal{X}$. 
    \item For any $x \in \mathcal{X}_0$, fix a generic quantity (e.g., a generic linear combination $L_x = \sum_{i=1}^r c_i x_i$) such that in an open subset $x\in U \subset \mathcal{X}$, the fibers of this quantity have a constant number of points. Moreover, the number is exactly equal to the algebraic multiplicity (the sum of the multiplicities of the irreducible components in $\mathcal{X}_0$) containing $x$.
\end{enumerate}
\begin{proposition} With the assumptions above, the morphism $f: \mathcal{X} \rightarrow S$ is flat, and any fiber $\mathcal{X}_t$ has no embedded point.
\end{proposition}
\begin{proof}
To show $f$ is flat, it is equivalently to show it is flat at every point $x \in \mathcal{X}$. The problem reduces to prove the following composition is flat:
\[
    U \xrightarrow{\pi} \mathbb{A}^1_S \xrightarrow{\rm pr} S,  
\]
where $\pi$ is defined via the choice $L_x$. Property (1) ensures that $\pi$ exists.

From the property (2), $\pi$ is a finite morphism with constant degree. This implies that it is locally-free and therefore flat. The projection is flat and the composition of flat morphism is flat. 

Finally, since the multiplicity at $x$ is already the same as the degree of the morphism, there is no zoom for the existence of embedded point supported on $x$. The argument works for every point. The proof is finished.
\end{proof}

\begin{remark}
We briefly highlight the geometric subtleties of extending a flat family over a higher-dimensional base $S$.

Recall that when the base $S$ is a smooth curve, any flat family over a punctured curve $S \setminus \{0\}$ can be uniquely extended to a flat family over all of $S$ simply by taking the scheme-theoretic closure of the generic fiber within the total space.

Over a higher-dimensional base, this is no longer automatic. Consider the family given by the graph of the rational function $(u,v) \mapsto \frac{uv}{u^2+v^2}$ over $\mathbb{C}^2 \setminus \{(0,0)\}$. While this restricts to an isomorphism away from the origin, the family cannot be extended to $(0,0)$ because the limit depends entirely on the direction of approach. Resolving such indeterminacies typically requires blowing up the base. Property (1) guarantees that this directional ambiguity does not occur in our setting.

Furthermore, reducedness is not generally preserved under base change. A priori, the special fiber $\mathcal{X}_0$ could acquire embedded points even if the total space is reduced. Property (2) is specifically required to rule out this possibility.
\end{remark}

Let $B \subset E^r$ be the curve consisting of the points $\mathbf{p}_{12,\epsilon} := (p_{12} + \epsilon, p_{12}-\epsilon, p_3, \dots, p_r)$ parametrized by $\epsilon$. By the generic flatness theorem, there exists a Zariski open subset $U \subset B$, such that $\overline{\mathcal{Y}}_{\mathbf{n},U}(\tau)$ is a flat family over $U$ with fiber at any point $\mathbf{p}_{12,\epsilon} \in U$ being isomorphic to the generalized Lam\'e curve $\overline{\mathcal{Y}}_{\mathbf{n}, \mathbf{p}_{12,\epsilon}}(\tau)$. Let $\overline{\mathcal{Y}}_{\mathbf{n},B}(\tau)$ be the scheme-theoretic closure of $\overline{\mathcal{Y}}_{\mathbf{n},U}(\tau)$ in $\mathcal{P}_n \times B$. Since $B$ is a regular curve, the projection $\overline{\mathcal{Y}}_{\mathbf{n},B}(\tau) \rightarrow B$ is flat.

By explicitly checking Properties (1) and (2) established above, we prove that, for $\epsilon \neq 0$, this abstract flat closure is fiberwise isomorphic to the generalized Lam\'e curve itself. For $\epsilon=0$, let $\overline{\mathcal{Y}}_{\mathbf{n}, \mathbf{p}_{12,0}}(\tau)$ denote the union of the components on the right-hand side of the degeneration theorem. We have:

\begin{theorem}
$\overline{\mathcal{Y}}_{\mathbf{n},B}(\tau)|_\epsilon \cong \overline{\mathcal{Y}}_{\mathbf{n}, \mathbf{p}_{12,\epsilon}}(\tau)$ for all $\epsilon\in B$.
\end{theorem}

\begin{proof}
First, we note that $\overline{\mathcal{Y}}_{\mathbf{n},U}$ (before taking the closure) is a complete intersection (which implies $S_k$ for all $k$) and generically smooth (which implies $R_0$). By Serre's criterion, $R_0 + S_1$ implies that the space is reduced. Consequently, the scheme-theoretic closure of the total space coincides with its topological closure equipped with the reduced induced scheme structure.

We first show that the special fiber $\overline{\mathcal{Y}}_{\mathbf{n}, \mathbf{p}_{12,0}}(\tau)$ is contained in the topological closure of the generic fiber. From the degeneration theorem, we have shown that by fixing generic $c = \sum_{\mu=1}^n a_\mu$, the locus $\overline{\mathcal{Y}}_{\mathbf{n}, \mathbf{p}_{12,0}}^c(\tau)$ deforms to the generic fiber via an explicit power series construction. We claim that by relaxing the fixed $c$ condition, every point in the special fiber deforms without obstruction.

For the generic $c$ argument, the deformation system decomposes into a block-triangular form:
\begin{equation} \label{eqn:deform_dege_sepa}
\begin{split}
 J_g (\boldsymbol{\alpha}_{\rm deform}) + * (\boldsymbol{\alpha}_{\rm fix}) &= \text{lower-order terms},
\\
 M_{\mathbf{n}_{12}^{(k)}, \mathbf{p}_{12}}(\mathbf{a}_{\rm fix}) \boldsymbol{\alpha}_{\rm fix}  &= \text{lower-order terms}.
\end{split}
\end{equation}
Because the Jacobian $J_g$ has full rank, the upper system is always solvable. The technical difficulty arises when the matrix $M_{\mathbf{n}_{12}^{(k)}, \mathbf{p}_{12}}$ drops rank at a non-generic configuration (i.e., when $[(\mathbf{a}_{\rm fix}, h)]$ is a singular point of the lower weight generalized Lam\'e curve $\overline{\mathcal{Y}}_{\mathbf{n}_{12}^{(k)}, \mathbf{p}_{12}}(\tau)$). In this case, a standard Puiseux series argument resolves the singularity by providing a fractional parametrization:
\[
\begin{split}
 \mathbf{a}_{\rm deform}(\epsilon) &= \mathbf{a}_{\rm deform} + \sum_{i \geq 1} \boldsymbol{\alpha}_{{\rm deform},i} \epsilon^{i/m},
 \\
 \mathbf{a}_{\rm fix}(\epsilon) &= \mathbf{a}_{\rm fix} + \sum_{i \geq 1} \boldsymbol{\alpha}_{{\rm fix},i} \epsilon^{i/m},
\end{split}
\]
where $m \in \mathbb{N}$ is a ramification index intrinsically determined by the singularity of the boundary stratum $\{p_{12}^k\} \times [(\mathbf{a}_{\rm fix}, h)]$. By fixing a generic linear combination $c' = \sum_\mu s_\mu a_\mu$ for generic $s_\mu \in \mathbb{C}$, the uniform translation ambiguity is transversally resolved, and the deformation's fractional coefficients can be solved inductively without obstruction. (For a generic point $[(\mathbf{a}_{\rm fix}, h)]$, the standard condition $\sum_\mu a_\mu = c$ with $m=1$ suffices.) This proves there is no obstruction to deforming any point into the generic fiber.

Finally, the algebraic multiplicity of each irreducible component of the special fiber is the product of two distinct geometric sources. The first comes from the coefficient of the degeneration formula (encoded by the isolated solutions of the $J_g$ block), which dictates the number of distinct deformation branches. The second comes from the intrinsic multiplicity of the boundary stratum $\{p_{12}^k\} \times [(\mathbf{a}_{\rm fix}, h)]$ itself (encoded by the $M$ block), which dictates the ramification index $m$ of each branch. Because the deformation construction explicitly captures both the number of branches and their Puiseux ramification, it perfectly reproduces this combined multiplicity. This guarantees there is no room for embedded points in the special fiber, completing the proof.
\end{proof}

Finally, iterating the argument in the proof for the curve base case, we prove Theorem~\ref{thm:universal_flatness}.
\begin{proof}
For a fixed total weight $n\in \mathbb{N}$, let
\[
    U \subset \Big\{ \mathbf{n} \in \mathbb{C}^r \ \Big| \ \sum_i n_i = n \Big\} \times \mathbb{C}^r
\]
be an open set such that $\overline{\mathcal{Y}}_U \rightarrow U$ is a flat family with fiber $\overline{\mathcal{Y}}_{\mathbf{n},\mathbf{p}}$ over $(\mathbf{n},\mathbf{p}) \in U$. We claim that the scheme-theoretic closure of $\overline{\mathcal{Y}}_U$, denoted $\overline{\mathcal{Y}}$, over the special locus, the fiber is uniquely determined by the tensor algebra of $\mathfrak{sl}_2$-modules within the BGG category $\mathcal{O}$ and the family 
\[
\overline{\mathcal{Y}} \rightarrow \Big\{ \mathbf{n} \in \mathbb{C}^r \ \Big| \ \sum_i n_i = n \Big\} \times \mathbb{C}^r
\]
is flat. 

For any point $[(\mathbf{a},h)]$ lying in a degenerate boundary stratum, we construct a multivariable deformation inductively on the codimension of the stratum. At each step, choosing a 1-parameter slice that deforms exactly one parameter (either separating colliding points or shifting a weight off the half-integer lattice) reduces the system to the 1-parameter construction of the previous theorem. That construction explicitly guarantees that the deformation is unobstructed and that the number of formal branches matches the local algebraic multiplicity. 

Iterating these 1-parameter deformations constructs the full multivariable power series, satisfying Property (1). Because algebraic multiplicity is multiplicative under successive flat degenerations, Property (2) is systematically preserved at every inductive step. This proves the family is globally flat.
\end{proof}

\section{Geometry of log-free curves}
In this section, we focus on the most interesting case $\mathbf{n} \in (\tfrac{1}{2} \mathbb{N})^r$. 
Recall from Remark~\ref{r:CI} that the log-free conditions form a complete intersection whenever $\mathbf{n} \notin (\tfrac{1}{2}\mathbb{N})^r$. The case $\mathbf{n} \in (\tfrac{1}{2} \mathbb{N})^r$ is precisely the exception where the complete intersection property fails.

\begin{conjecture}[{ \cite[Conjecture~5.7]{Wang_2020} }] \label{conj_Wang}
The log-free variety $V_{\mathbf{n},\mathbf{p}}(\tau)$ contains curve components. 
\end{conjecture}

The main content of this section is to prove a stronger form of this conjecture, showing that $V_{\mathbf{n},\mathbf{p}}(\tau)$ is a finite union of curves with no isolated points.

First, we prove Conjecture~\ref{conj_Wang} by considering the natural projection:
\[
    \pi: \mathcal{Y}_{\mathbf{n},\mathbf{p}}(\tau) \rightarrow V_{\mathbf{n},\mathbf{p}}(\tau).
\]
We establish a precise correspondence between the branches of $V_{\mathbf{n},\mathbf{p}}(\tau)$ near $B=\infty$ and the Type-I boundary points of the generalized Lam\'e curve (GLC). This correspondence implies the non-constancy of the morphism $\pi$, which guarantees that $V_{\mathbf{n},\mathbf{p}}(\tau)$ contains curve components. To state this correspondence formally in Theorem~\ref{branch}, we parameterize both sets up to a natural involution.

On the log-free side, near the infinity divisor $B=\infty$, the space $V_{\mathbf{n},\mathbf{p}}(\tau)$ is dominated by the leading terms of the log-free conditions:
\[
    q_{2n_i}(A_i, B) = \frac{(-1)^{2n_i}}{({2n_i} !)^2}\prod_{j = 0}^{2n_i} (A_i - ({2n_i} - 2j)B^{1/2})=0, \quad \text{for } i=1,\dots,r.
\]
Imposing the global constraint $\sum_{i=1}^r A_i = 0$, the possible asymptotic branches $A_i \sim (2n_i-2k_i)B^{1/2}$ are characterized by the index set
\[
    \Big\{ (k_1, \dots, k_r) \mid 0 \leq k_i \leq 2n_i, \sum_{i=1}^r k_i = n \Big\} / \sim,
\]
where the equivalence relation $(k_1,\dots,k_r) \sim (2n_1-k_1,\dots,2n_r-k_r)$ identifies branches differing only by the local reparametrization $\epsilon \mapsto -\epsilon$ at infinity. Note that when $\mathbf{n} \in \mathbb{N}^r$, there is a unique self-equivalent (symmetric) branch given by $k_i = n_i$ for all $i$.

On the GLC side, we define an analogous involution on the Type-I boundary points. Two points are equivalent, $\mathbf{a} \sim \mathbf{a}'$, if their divisor sum is $\mathbf{a} + \mathbf{a}' = \sum_{i=1}^r 2n_i p_i$. As on the log-free side, when $\mathbf{n} \in \mathbb{N}^r$, there is a unique self-equivalent (symmetric) boundary point given by $\mathbf{a} = \sum_{i=1}^r n_i p_i$.

\begin{theorem}\label{branch}
The natural morphism $\pi$ maps the open neighborhoods of the Type-I boundary points onto the branches of $V_{\mathbf{n},\mathbf{p}}(\tau)$ near the infinity divisor $B=\infty$, inducing a 1-1 correspondence up to equivalence. Specifically:
\begin{enumerate}
    \item If the boundary point is non-symmetric, i.e., $\mathbf{a}_{\mathbf{k}} \nsim \mathbf{a}_{2\mathbf{n}-\mathbf{k}}$, $\pi$ maps its neighborhood onto the branch intersecting the infinity divisor at 
    \[
        [2(n_1 - k_1) : \dots : 2(n_r - k_r) : 1].
    \]
    \item If the boundary point is symmetric, $\pi$ maps the neighborhood of the self-equivalent point $\mathbf{a}$ onto the unique branch intersecting the infinity divisor at the orbifold point
    \[
        [0 : \dots : 0 : 1].
    \]
\end{enumerate}
\end{theorem}

\begin{proof}
Recall the computation for the existence of the Type-I boundary in Theorem~\ref{special-pt}. Up to a choice of representative, we assume that $a_\mu = p_j$ whenever $[a_\mu] = [p_j]$. For $j=1,\dots,r$, let $I_j$ be the index set of $\mu$ such that $a_{\mu}=p_j$. The parametrization
\[
    a_\mu = p_j + \sum_{m \geq 1} \alpha_{\mu m} \epsilon^m, \text{ for }\mu \in I_j, \ j=1,\dots,r,
\]
always exists. The linear terms $\{ \alpha_{\mu 1} \}$ satisfy:
\[
    \Big( \sum_{\nu \in I_1} \frac{n_1}{\alpha_{\nu 1}}: \dots: \sum_{\nu \in I_r} \frac{n_r}{\alpha_{\nu 1}}   \Big) = \lambda ( k_1: \dots: k_r ), \quad \lambda \in \mathbb{C}^\times.
\]
Furthermore, for each $i$, by Lemma~\ref{lemma_boundary}, $\{ \alpha_{\nu 1} \}_{\nu \in I_i}$ are the roots of the corresponding Laguerre polynomial. We can compute:
\[  
    \sum_{\nu \in I_i} \frac{1}{\alpha_{\nu 1}^2} = \frac{\lambda^2 k_i (k_i-2n_i)}{n_i^2(1-2n_i)}. 
\]

Now the theorem follows from a direct computation:
\[
\begin{split}
\frac{A_i(\epsilon)}{2n_i} &= - \zeta_{i \mathbf{a}} + \sum_{j = 1, \ne i}^r n_j \zeta_{ij} - h 
\\
&= \frac{\lambda k_i}{n_i \epsilon} + O(1) - \frac{\lambda}{\epsilon} = \frac{\lambda(k_i-n_i)}{n_i\epsilon}
\\
B(\epsilon) &= \frac{A_i(\epsilon)^2}{(2n_i)^2} - \sum_{j = 1, \ne i}^r n_j(n_j + 2n_i) \wp_{ij} - \sum_{j = 1, \neq i}^{r} A_j(\epsilon) \zeta_{ij} + (2n_i - 1) \wp_{i \mathbf{a}} 
\\
&= \frac{1}{\epsilon^2} \Bigg( \frac{\lambda^2 (n_i-k_i)^2}{n_i^2} + (2n_i-1) \sum_{\nu \in I_i} \frac{1}{\alpha_{\nu 1}^2} \Bigg) + O(1)
\\
&= \frac{\lambda^2}{\epsilon^2} + O(1). 
\end{split}
\]
This finishes the proof.
\end{proof}
We further prove that the log-free variety has no isolated points. Recall that $F(\mathbf{n})$ is the number of Type-I boundary points. 
\begin{theorem} \label{thm:log_free_curveness}
$\overline{V}_{\mathbf{n},\mathbf{p}}(\tau)$ is a curve of degree $\frac{1}{2} F(\mathbf{n})$ in variables $A_i$ and $B$.
\end{theorem}

\begin{proof}
\textbf{Step 1: Degree of the general fiber.}

Note that every curve component of $\overline{V}_{\mathbf{n},\mathbf{p}}(\tau)$ intersects the infinity divisor $B =\infty$. It follows from the top degree terms $q_{2n_i}(A_i,B)$, where $B \rightarrow \infty$ whenever $A_i \rightarrow \infty$. By Theorem~\ref{branch}, there are $\frac{1}{2} F(\mathbf{n})$ such branches, accounting for the orbifold point. Thus, the union of the 1-dimensional components of $\overline{V}_{\mathbf{n},\mathbf{p}}(\tau)$ has degree $\frac{1}{2} F(\mathbf{n})$.

\vspace{0.5em}
\noindent \textbf{Step 2: Flat degeneration and the central fiber.} 

To rule out isolated and embedded points, we construct a flat family via two steps. The second step uses the Gr\"obner degeneration. 

First, the log-free condition at $p_i$ is 
\[
    F_{2n_i}(\mathbf{A}, B) = q_{2n_i}(A_i, B) + H_i (\mathbf{A},B,\zeta,\wp) = 0.
\]
This equation is quasi-homogeneous with respect to the weights $1$ for $A_i, \zeta$ and $2$ for $B, \wp$. Thus, the simultaneous substitution $A_i \mapsto tA_i, B \mapsto t^2B, \zeta \mapsto t\zeta, \wp \mapsto t^2\wp$ factors out a power of $t$, perfectly preserving the ideal for any $t \neq 0$. 
Second, we apply a standard Gr\"obner weight vector strictly to the variables: $w(A_i) = -1$ and $w(B) = -2$, which induces the substitution $A_i \mapsto t^{-1}A_i$ and $B \mapsto t^{-2}B$. Composing these two steps cancels the scaling on the variables, yielding the transformation $\zeta \mapsto t\zeta$ and $\wp \mapsto t^2\wp$. 

This composition constructs a 1-parameter family $V_{\mathbf{n},\mathbf{p}}(\tau)_t$ that is flat over $\mathbb{C}[t]$ (\cite[Theorem 15.17]{Eisenbud_commutative_algebra_1995}). At $t=1$, we recover the general fiber $V_{\mathbf{n},\mathbf{p}}(\tau)$. At $t=0$, the central fiber, denoted $V_{\mathbf{n},0}$, is the flat limit, defined by the initial ideal $\text{in}_w(I)$.

By definition, the initial ideal $\text{in}_w(I)$ is generated by the leading (highest $w$-weight) terms of all polynomials in $I$. In general, this differs from the naive ideal generated by only the leading terms of the explicit equations, because hidden algebraic relations (S-polynomials) can emerge during a Gr\"obner basis computation. 

To emphasize the difference, let $W_0$ be the naive limit scheme defined purely by taking the limit $t \rightarrow 0$ of our given generators:
\[
    q_{2n_i}(A_i, B) = \frac{(-1)^{2n_i}}{({2n_i} !)^2}\prod_{j = 0}^{2n_i} (A_i - ({2n_i} - 2j)B^{1/2})=0, \quad 1 \le i \le r,
\]
together with the global condition $\sum_{i=1}^r A_i = 0$. We have $V_{\mathbf{n},0} \subset W_0$. 

\vspace{0.5em}
\noindent \textbf{Step 3: Reducedness of $W_0$ and $V_{\mathbf{n},0} = W_0$.} 

Consider the base change $B=C^2$, yielding the space $\widetilde{W}_0$ with coordinate ring $\mathbb{C}[\mathbf{A},C]/I_{\widetilde{W}_0}$ where 
\[
    I_{\widetilde{W}_0} = \Bigg( \prod_{j=0}^{2n_1} \big(A_1-(2n_1-2j)C\big), \dots, \prod_{j=0}^{2n_r} \big(A_r-(2n_r-2j)C\big), \sum_{i=1}^r A_i \Bigg).
\]
We analyze the affine cone $\widetilde{W}_0$ into three parts: the affine chart $C=1$, the infinity hyperplane $C=0$, and the origin $C = A_i=0$.

On the chart $C=1$, the linear factors $A_i - (2n_i - 2j_i)$ are pairwise coprime. By the Chinese Remainder Theorem, the ideal generated by their products decomposes into an intersection of maximal ideals:
\[
    \bigcap_{0 \le j_i \le 2n_i} \big( A_1-(2n_1-2j_1), \dots, A_r-(2n_r-2j_r) \big).
\]
These $\prod_{i=1}^r (2n_i+1)$ distinct points intersect the hyperplane $\sum_{i=1}^r A_i = 0$ yields $F(\mathbf{n})$ points. Because the dehomogenized ideal is given by this radical intersection of points, the affine cone away from $C=0$ consists of union of $F(\mathbf{n})$ reduced lines.

Second, at the hyperplane $C=0$, the only root of the ideal
\[
(A_1^{2n_1+1}, \dots, A_r^{2n_r+1}, \sum A_i).
\] 
is the origin $A_1 = \dots = A_r = 0$. In the projective scheme, this corresponds to the empty set, proving no additional lines are supported at infinity.

Third, to prove the affine cone has no embedded point at the origin (meaning the irrelevant maximal ideal is not an associated prime), we show that $C$ is a non-zero-divisor on the coordinate ring. We analyze this via the graded Artinian complete intersection $S = \mathbb{C}[A_1,\dots,A_r]/(A_1^{2n_1+1},\dots,A_r^{2n_r+1})$. Recall the classical work by Stanley:
\begin{theorem}[\cite{Stanley_1980}]
The ring $S$ satisfies the Hard Lefschetz property with respect to the linear operator $L= A_1 + \dots + A_r$. 
\end{theorem}
Our argument only requires the Weak Lefschetz Property (the maximal rank of the multiplication maps $L: S_m \to S_{m+1}$). Combined with the unimodal symmetry of $S$, the map is injective for $m \le n-1$ and surjective for $m \ge n$. 

Let $\mathcal{J}(\mathbf{n})$ be the index set such that $\{ A_J \mid J \in \mathcal{J}(\mathbf{n}) \}$ forms a vector space basis for $S/L(S) = \bigoplus_{m=0}^{n-1} S_{m+1}/L(S_m)$. Note that 
\[
| \mathcal{J}(\mathbf{n}) | = \dim S_n = F(\mathbf{n}).
\]
Furthermore, we claim that $\{ A_J C^k \}_{J \in \mathcal{J(\mathbf{n})}, k \geq 0}$ forms a basis of the ring $\mathbb{C}[\mathbf{A},C]/ I_{\widetilde{W}_0}$. First of all, this set spans the ring since any term containing $A^{2n_i+1}$ reduces to higher powers of $C$. Second, they have no relation since when $C=1$ the cardinality of the spanning set and the degree of the affine cone are both equal to $F(\mathbf{n})$. This proves the basis claim. $C$ is a non-zero-divisor follows directly from the claim. 

The naive central fiber $W_0$ corresponds to the quotient of $\mathbb{Z}/2\mathbb{Z}$ action $C \mapsto -C$ on $\widetilde{W}_0$. Taking the quotient preserves the reducedness. We conclude that $W_0$ is a reduced curve of degree $\frac{1}{2} F(\mathbf{n})$. Note that the degree is in weighted projective space, the symmetric $C$-axis (if exists) will be counted as $\frac{1}{2}$ after quotient.

Finally, from Step 2, $V_{\mathbf{n},0} \subseteq W_0$. From Step 1, by flatness, $V_{\mathbf{n},0}$ must have the same degree as the general fiber, which is $\tfrac{1}{2} F(\mathbf{n})$. Since $W_0$ is reduced and has degree $\tfrac{1}{2} F(\mathbf{n})$, $V_{\mathbf{n},0} = W_0$. Therefore, no additional algebraic relations exist in the initial ideal. Because $V_{\mathbf{n},0}$ is a reduced curve, flatness guarantees the general fiber $\overline{V}_{\mathbf{n},\mathbf{p}}(\tau)$ is a reduced curve of degree $\frac{1}{2} F(\mathbf{n})$. This finishes the proof.
\end{proof}

\begin{corollary}
The Hilbert quasi-polynomial of $\overline{V}_{\mathbf{n},\mathbf{p}}(\tau)$ is given by
\[
\chi_{\overline{V}_{\mathbf{n},\mathbf{p}}(\tau)}(m)=\frac{m-n+1}{2}F(\mathbf{n})+\frac{1}{4}\left( \prod_{i=1}^r (2n_i+1)+(-1)^m\delta_{\mathbf{n},\mathbb{N}^r}\right)
\]
where
\[
\delta_{\mathbf{n},\mathbb{N}^r} = 
\begin{cases}
    1,\qquad\textnormal{if } \mathbf{n} \in \mathbb{N}^r
    \\
    0,\qquad\textnormal{otherwise.} 
\end{cases}
\]
\end{corollary}
\begin{proof}
For $m$ sufficiently large, the value of the Hilbert quasi-polynomial $\chi_{\overline{V}_{\mathbf{n},\mathbf{p}}(\tau)}(m)$ equals the dimension of the $m$-th graded piece of the homogeneous coordinate ring. In the weighted projective space $\mathbb{P}^{r+1}(1,\dots,1,2)$, this dimension is exactly the number of basis elements $A_J B^k$ from the affine coordinate ring whose total weighted degree satisfies $|J| + 2k \le m$.

Let $F(\mathbf{n}, j) = \dim S_j$ denote the dimension of the $j$-th graded piece of the complete intersection $S$. Because the index set $\mathcal{J}(\mathbf{n})$ forms a basis for $S/L(S)$, the number of basis elements $J \in \mathcal{J}(\mathbf{n})$ with degree $|J| \le k$ evaluates to a telescoping sum for any $k \le n$:
\[
    |\{ J \in \mathcal{J}(\mathbf{n}) \mid |J| \le k \}| = \sum_{j=0}^k (\dim S_j - \dim S_{j-1}) = F(\mathbf{n}, k).
\]
By summing these dimensions over all valid powers $k \ge 0$ such that $|J| + 2k \le m$, we can express the Hilbert quasi-polynomial as a piecewise function depending on the parity of $m$: 
\[
\chi_{\overline{V}_{\mathbf{n},\mathbf{p}}(\tau)}(m)=
\begin{cases}
    P_1(m), \qquad\textnormal{ if } m\equiv 1 \pmod 2\\
    P_2(m), \qquad\textnormal{ if } m\equiv 0 \pmod 2
\end{cases}
\]
where
\begin{align*}
    P_1(m)&=F(\mathbf{n})\lceil \frac{m-n}{2}\rceil+\sum_{j=0}^{\lfloor\frac{n}{2}\rfloor}F(\mathbf{n},2j+1)\\
    P_2(m)&=F(\mathbf{n})\lceil \frac{m-n}{2}\rceil+\sum_{j=0}^{\lfloor\frac{n}{2}\rfloor}F(\mathbf{n},2j).
\end{align*}

We can further simplify these summations by noticing that:
\[
    2\sum_{j=0}^n F(\mathbf{n},j)= \prod_{i=1}^r (2n_i+1)+F(\mathbf{n})
\]
\[
    2\sum_{j=0}^{n}(-1)^j F(\mathbf{n},j)= \delta_{\mathbf{n},\mathbb{N}^r} +(-1)^nF(\mathbf{n}).
\]
Therefore, combining the parity cases, we compute:
\[
    \chi_{\overline{V}_{\mathbf{n},\mathbf{p}}(\tau)}(m)=\frac{m-n+1}{2}F(\mathbf{n})+\frac{1}{4}\left( \prod_{i=1}^r (2n_i+1)+(-1)^m\delta_{\mathbf{n},\mathbb{N}^r}\right).
\]
Notice that the constant term alternates with the difference $\frac{1}{4}\delta_{\mathbf{n},\mathbb{N}^r}$, which is precisely the contribution of the orbifold point.
\end{proof}

From the calculation of the Hilbert polynomial $\chi_{\overline{V}_{\mathbf{n},\mathbf{p}}(\tau)}$, we also have

\begin{corollary} \label{Vn-genus}
The arithmetic genus of $\overline{V}_{\mathbf{n},\mathbf{p}}(\tau)$ (as a coarse moduli space) is given by 
\[
p_a(\overline{V}_{\mathbf{n},\mathbf{p}}(\tau))=\frac{n-1}{2}F(\mathbf{n})-\frac{1}{4}\left( \prod_{i=1}^r(2n_i+1)+\delta_{\mathbf{n},\mathbb{N}^r} \right)+1 \in \mathbb{Z}_{\geq 0}.
\]
The orbifold genus of $\overline{V}_{\mathbf{n},\mathbf{p}}(\tau)$ is given by 
\[
p_{a}^{orb}(\overline{V}_{\mathbf{n},\mathbf{p}}(\tau))=\frac{n-1}{2}F(\mathbf{n})-\frac{1}{4}\prod_{i=1}^r(2n_i+1)+1 \in \mathbb{Q}_{\geq 0}. 
\]
\end{corollary}
\begin{proof}
It follows directly from 
\[
    p_a(\overline{V}_{\mathbf{n},\mathbf{p}}(\tau))  = 1 - \chi_{\overline{V}_{\mathbf{n},\mathbf{p}}(\tau)}(0), \quad p_a^{orb}(\overline{V}_{\mathbf{n},\mathbf{p}}(\tau)) = 1- \chi^{orb}_{\overline{V}_{\mathbf{n},\mathbf{p}}(\tau)}(0).
\]
Note that the difference between the holomorphic Euler characteristic and orbifold Euler characteristic is the contribution of the twisted sector, which is exactly $\frac{1}{4}\delta_{\mathbf{n},\mathbb{N}^r}$. 
\end{proof}

\section{Monodromy and isomonodromic deformations}

\subsection{Monodromy}
We start with a general argument. Let $n_i \in \mathbb{C}$, the factors $\sigma(z-p_i)^{n_i}$ are multi-valued. To define their analytic continuation, we fix a base point $z_0 \in \mathbb{C} \setminus \{p_i + \Lambda_\tau\}$ and consider a generic path $\gamma_1$ from $z_0$ to $z_0+\omega_1$. 

Let $\gamma_{1,i} = \gamma_1 - p_i$, and consider a closed curve $\Gamma_{1,i} = \gamma_{1,i} + \delta + \gamma_{1,i}' + \delta'$, where $\gamma_{1,i}'$ is the point reflection of $\gamma_{1,i}$ with respect to $p_i$, $\delta$ and $\delta'$ are two curves connecting $\gamma_{1,i}$ and $\gamma_{1,i}'$ such that $\delta$ and $-\delta'$ are differed by a translation of $\omega_1$.

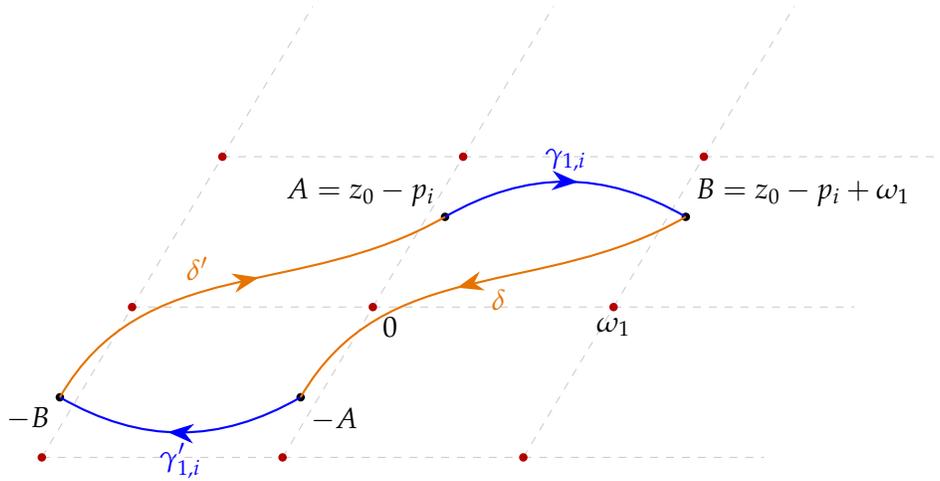
\begin{center}
\begin{figure}[ht!]
    \begin{tikzpicture}[
    scale=0.8,
    midarrow/.style={postaction={decorate,decoration={markings, mark=at position 0.55 with {\arrow{Stealth[scale=1.5]}}}}}
]

\coordinate (W1) at (4, 0);     
\coordinate (W2) at (1.5, 2.5); 

\foreach \i in {-1, 0, 1} {
    \foreach \j in {-1, 0, 1} {
        \coordinate (L) at ($\i*(W1) + \j*(W2)$);
        
        \draw[gray!40, thin, dashed] (L) -- ($(L)+(W1)$);
        \draw[gray!40, thin, dashed] (L) -- ($(L)+(W2)$);
        
        \fill[red!70!black] (L) circle (2pt);
    }
}

\node[below right] at (0,0) {$0$};
\node[below] at (W1) {$\omega_1$};
\node[below] at ($-1*(W1)$) {};
\node[above left] at (W2) {};
\node[above right] at ($(W1)+(W2)$) {};

\coordinate (A) at (1.2, 1.5);
\coordinate (B) at ($(A) + (W1)$);

\coordinate (nA) at ($-1*(A)$);
\coordinate (nB) at ($-1*(B)$);

\fill[black] (A) circle (2pt) node[above left] {$A = z_0-p_i$};
\fill[black] (B) circle (2pt) node[above right] {$B = z_0-p_i+\omega_1$};
\fill[black] (nA) circle (2pt) node[below right] {$-A$};
\fill[black] (nB) circle (2pt) node[below left] {$-B$};


\draw[blue, thick, midarrow] (A) to[out=30, in=150] node[above, midway] {$\gamma_{1,i}$} (B);

\draw[blue, thick, midarrow] (nA) to[out=210, in=330] node[below, midway] {$\gamma_{1,i}'$} (nB);

\draw[orange!90!black, thick, midarrow] (nB) to[out=60, in=210] node[left, midway, xshift=-4pt, yshift=6pt] {$\delta'$} (A);

\draw[orange!90!black, thick, midarrow] (B) to[out=210, in=60] node[right, midway, xshift=4pt, yshift=-6pt] {$\delta$} (nA);

\end{tikzpicture}
\end{figure}
\end{center}

Integrating $\zeta$ along the translated path $\gamma_{1,i} = \gamma_1 - p_i$ yields
\[
\begin{split}
\log \sigma(z_0-p_i+\omega_1) & - \log \sigma(z_0-p_i) = \int_{\gamma_{1,i}} \zeta(w)\ dw 
\\
&= \frac{1}{2}\left(\int_{\Gamma_{1,i}} \zeta(w)\ dw - \int_{\delta+\delta'} \zeta(w)\ dw\right) \\
&= {\rm Ind}_{\Gamma_{1,i}}(\Lambda_\tau)\pi\sqrt{-1} + \eta_1\big(z_0-p_i + \tfrac{1}{2}\omega_1\big),
\end{split}
\]
where 
${\rm Ind}_{\Gamma_{1,i}}(\Lambda_\tau)$
is the winding number of $\Lambda_\tau$ along $\Gamma_{1,i}$. The number does not depend on the choice of $\delta$ and $\delta'$, and hence only depends on $\gamma_1$. 

Apply this to the ansatz solution, denote $w_1^{\gamma_1}(z_0)$ be the analytic continuation of $w_1(z_0)$ (with a fixed branch) along $\gamma_1$, we have
\begin{equation} \label{e:analy_conti}
w_1^{\gamma_1}(z_0) = w_1(z_0) \exp\left[ \omega_1 h - \eta_1 \sigma_{\mathbf{n},\mathbf{p}}(\mathbf{a}) + \Theta(\gamma_1) \right]
\end{equation}
where the constant
\[
\Theta(\gamma_1) = \pi \sqrt{-1} \sum_{i=1}^r n_i \text{Ind}_{\Gamma_{1,i}}(\Lambda_\tau),
\] 
depends only on the homotopy class of $\gamma_1$ on $\mathbb{C} \setminus\Big( \Lambda_\tau + \{ p_1,\dots,p_r \}\Big)$. 

Let $\widetilde{\sigma}_{\mathbf{n},\mathbf{p}}$ be the addition map on the universal cover:
\[
\begin{split}
    \widetilde{\sigma}_{\mathbf{n},\mathbf{p}}: \Sym^n \mathbb{C} &\rightarrow \mathbb{C}
    \\
    \{ \mathbf{a} \} & \rightarrow \sum_{\mu=1}^n a_\mu - \sum_{i=1}^r n_ip_i. 
\end{split}
\]
By convention, when $n=0$, we define $\Sym^0 \mathbb{C} := \{\emptyset\}$, and the empty sum $\sum_{\mu=1}^n a_\mu$ naturally evaluates to $0$.

\begin{lemma} \label{l:twisted_mono_data}
For any representative $(\mathbf{a}, h) \in \Sym^n \mathbb{C} \times \mathbb{C}$ of an equivalence class in $\overline{\mathcal{Y}}_{\mathbf{n},\mathbf{p}}(\tau)$, the linear system
\begin{equation}
\begin{split}
    \tilde{t} \omega_1 + \tilde{s} \omega_2 &= \widetilde{\sigma}_{\mathbf{n},\mathbf{p}}(\mathbf{a}), \\
    \tilde{t} \eta_1 + \tilde{s} \eta_2 &= h,
\end{split}
\end{equation}
determines a unique pair of complex values $(\tilde{t},\tilde{s})\in \mathbb{C}^2$.

Furthermore, the equivalence class $(t,s) \in (\mathbb{C}/\mathbb{Z})^2$ of the solution $(\tilde{t},\tilde{s})$ is independent of the choice of representative. 
\end{lemma}
\begin{proof}
The Legendre relation $\eta_1 \omega_2 - \eta_2 \omega_1 = 2\pi \sqrt{-1}$ gives the uniqueness of the solution $(\tilde{t}, \tilde{s})$. 

Under the equivalence relation defining $\overline{\mathcal{Y}}_{\mathbf{n},\mathbf{p}}(\tau)$, a lattice shift by $m\omega_1 + n\omega_2 \in \Lambda_\tau$ on the coordinates $\mathbf{a}$ simultaneously shifts the lifted addition map $\widetilde{\sigma}_{\mathbf{n},\mathbf{p}}(\mathbf{a})$ by $m\omega_1 + n\omega_2$ and the coordinate $h$ by $m\eta_1 + n\eta_2$. This induces the exact integer translations $t \mapsto t+m$ and $s \mapsto s+n$. Thus, the image in $(\mathbb{C}/\mathbb{Z})^2$ is well-defined.
\end{proof}

\begin{definition}[Twisted monodromy data]
For $[(\mathbf{a},h)] \in \overline{\mathcal{Y}}_{\mathbf{n},\mathbf{p}}(\tau)$, the well-defined pairs $(t,s)\in (\mathbb{C}/\mathbb{Z})^2$ defined by Lemma~\ref{l:twisted_mono_data} is called the \emph{twisted monodromy data} associated to $[(\mathbf{a},h)]$.
\end{definition}

This proposition demonstrates that the actual monodromy multipliers $\lambda_{\gamma_i}$ and the twisted monodromy data $(t,s)$ differ exactly by the explicit factor $e^{\Theta(\gamma_i)}$.

\begin{proposition} 
For general $\mathbf{n}$, let $w(z)$ be the GHH ansatz solution corresponding to $[(\mathbf{a},h)]$ with associated twisted monodromy data $(t,s) \in (\mathbb{C}/\mathbb{Z})^2$.

Define the eigenvalues $\lambda_{\gamma_1} := e^{\Theta(\gamma_1)} e^{-2\pi \sqrt{-1}s}$ and $\lambda_{\gamma_2} := e^{\Theta(\gamma_2)} e^{2\pi \sqrt{-1}t}$. Let $v(z)$ (possibly not of ansatz form) be a second linearly independent solution. We have
\[
\begin{pmatrix}
    v(z_0) \\
    w(z_0)
\end{pmatrix}^{\gamma_i} = 
\begin{pmatrix}
    \lambda_{\gamma_i}^{-1} & c_{\gamma_i} \\
    0 & \lambda_{\gamma_i} 
\end{pmatrix}
\begin{pmatrix}
    v(z_0) \\
    w(z_0)
\end{pmatrix},
\]
for some constant $c_{\gamma_i} \in \mathbb{C}$. 
\end{proposition}

\begin{proof}
From equation~\eqref{e:analy_conti} and the Legendre relation $\eta_1 \omega_2 - \eta_2 \omega_1 = 2\pi \sqrt{-1}$, we have $w^{\gamma_i}(z_0) = \lambda_{\gamma_i} w(z_0)$.

Since the generalized Lam\'e equation has no first-derivative term, its Wronskian is a global non-zero constant: 
\[
    W(z) = w(z) v'(z) - w'(z) v(z) = c.
\]
Analytic continuation of the Wronskian along $\gamma_i$ implies that the determinant of the monodromy matrix must be $1$. Since $w(z)$ is an eigenvector with eigenvalue $\lambda_{\gamma_i}$, the other diagonal entry must be $\lambda_{\gamma_i}^{-1}$, forcing the upper-triangular form.
\end{proof}

While an ansatz solution forces the monodromy representation to be upper-triangular, the existence of a pair of linearly independent ansatz solutions simultaneously diagonalizes the monodromy matrices. It also imposes constraints on both the local exponents of GLE and the monodromy factors $e^{\Theta(\gamma_i)}$.
\begin{proposition} \label{prop:monodromy}
Let $w_1(z)$ and $w_2(z)$ be two linearly independent GHH ansatz solutions corresponding to the points $[(\mathbf{a}_1,h_1)] \in \mathcal{Y}_{\mathbf{n},\mathbf{p}}(\tau)$ and $[(\mathbf{a}_2,h_2)] \in \mathcal{Y}_{\mathbf{m},\mathbf{p}}(\tau)$, respectively. Then for each $i$, the local parameters $n_i, m_i \in \tfrac{1}{2}\mathbb{Z} \setminus \{0, -\tfrac{1}{2}\}$, and satisfy either
\[
    n_i = m_i \in \tfrac{1}{2} \mathbb{N} \qquad \text{or} \qquad n_i + m_i = -1.
\]
Furthermore, let $(t,s)$ be the twisted monodromy data associated to $[(\mathbf{a}_1,h_1)]$. The monodromy matrices along the period cycles $\gamma_1, \gamma_2$ are strictly diagonal and given by:
\[
\begin{split}
    \begin{pmatrix} w_1(z_0) \\ w_2(z_0) \end{pmatrix}^{\gamma_1} &= \begin{pmatrix} e^{\Theta(\gamma_1)} e^{-2\pi \sqrt{-1}s} & 0 \\ 0 & e^{-\Theta(\gamma_1)} e^{2\pi \sqrt{-1} s} \end{pmatrix} \begin{pmatrix} w_1(z_0) \\ w_2(z_0) \end{pmatrix},
    \\
    \begin{pmatrix} w_1(z_0) \\ w_2(z_0) \end{pmatrix}^{\gamma_2} &= \begin{pmatrix} e^{\Theta(\gamma_2)} e^{2\pi \sqrt{-1}t} & 0 \\ 0 & e^{-\Theta(\gamma_2)} e^{-2\pi \sqrt{-1} t} \end{pmatrix} \begin{pmatrix} w_1(z_0) \\ w_2(z_0) \end{pmatrix},
\end{split}
\]
where $e^{\Theta(\gamma_i)} \in \{ \pm 1 \}$. Moreover, if $\mathbf{n} \in \mathbb{Z}^r$, we have $e^{\Theta(\gamma_i)} = 1$.
\end{proposition}

\begin{proof}
As in the proof of the previous proposition, we perform analytic continuation on the Wronskian (which is a global constant):
\[
\begin{split}
    W^\gamma(z) &= w_1^\gamma(z) (w_2^\gamma)'(z) - (w_1^\gamma)'(z) w_2^\gamma(z) \\
    &= (M_1(\gamma) w_1(z))(M_2(\gamma) w_2'(z)) - (M_1(\gamma) w_1'(z))(M_2(\gamma) w_2(z)) \\
    &= M_1(\gamma) M_2(\gamma) W(z).
\end{split}
\]
We conclude that $M_1(\gamma)M_2(\gamma) =1$. Note that we can always choose a loop $\gamma'$ that winds once around $p_i$ for any $i$. We have
\[
    M_1(\gamma') M_2(\gamma') = M_1(\gamma) M_2(\gamma) e^{-2\pi (n_i+m_i)\sqrt{-1} } =1.
\]
When $n_i =m_i$, we have $n_i,m_i \in \tfrac{1}{2}\mathbb{Z}$. Furthermore, $n_i,m_i \in \tfrac{1}{2}\mathbb{N}$ by the Frobenius method. For $n_i\neq m_i$, from the local exponents of the generalized Lam\'e equation, we automatically have $n_i+m_i = -1$.

To establish the diagonal monodromy matrix for $w_2$, note that because $M_1(\gamma_i) M_2(\gamma_i) = 1$, the monodromy multipliers for $w_1$ and $w_2$ are inverses of each other. This implies that their associated parameters on the universal cover satisfy $h_1 + h_2 = 0$ and $\widetilde{\sigma}_{\mathbf{n},\mathbf{p}}(\mathbf{a}_1) + \widetilde{\sigma}_{\mathbf{m},\mathbf{p}}(\mathbf{a}_2) = 0$. We therefore have the linear system: 
\[
\begin{split}
\omega_1( h_1 + h_2 ) - \eta_1 \big(\widetilde{\sigma}_{\mathbf{n},\mathbf{p}}(\mathbf{a}_1) + \widetilde{\sigma}_{\mathbf{m},\mathbf{p}}(\mathbf{a}_2) \big) &= 0,
\\
\omega_2(h_1 + h_2 ) - \eta_2 \big( \widetilde{\sigma}_{\mathbf{n},\mathbf{p}}(\mathbf{a}_1) + \widetilde{\sigma}_{\mathbf{m},\mathbf{p}}(\mathbf{a}_2) \big) &= 0.
\end{split}
\]
By the Legendre relation, this forces the twisted monodromy data for $w_2$ to be exactly $(-t, -s)$, yielding the inverted exponents in the diagonal matrices. The proposition then follows.
\end{proof}

\begin{remark}
Geometrically, suppose $\gamma_i$ forms a simple loop when projected to $E_\tau$. Consider the translations of $\gamma_i$ by $\Lambda_\tau$ on $\mathbb{C}$, which cut $\mathbb{C}$ into a countable number of strips. The sign $e^{\Theta(\gamma_i)}$ is positive if the total weight of the poles $p_j$ lying in the union of alternating strips is an integer.
\end{remark}

\vspace{5mm}
Let $\mathbf{n} \in (\tfrac{1}{2} \mathbb{N})^r$. 
There are three scenarios for the fibers of the natural morphism $\overline{\mathcal{Y}}_{\mathbf{n},\mathbf{p}}(\tau) \rightarrow \overline{V}_{\mathbf{n},\mathbf{p}}(\tau)$. We classify these scenarios according to their monodromy representations: 

\begin{proposition} \label{generic-2-to-1} 
The fiber of the morphism consists of:
\begin{itemize}
    \item \textbf{Two distinct points:} if either (i) the generalized Lam\'e equation admits two linearly independent ansatz solutions and $(t, s) \notin (\tfrac{1}{2} \mathbb{Z} / \mathbb{Z})^2$, or (ii) it corresponds to a non-symmetric Type-I boundary point. 
    \item \textbf{A rational curve (isomorphic to $\mathbb{P}^1$):} if the generalized Lam\'e equation admits two linearly independent ansatz solutions and $(t, s) \in (\tfrac{1}{2} \mathbb{Z} / \mathbb{Z})^2$. 
    \item \textbf{A single ramified point of multiplicity two:} if either (i) the generalized Lam\'e equation admits only one ansatz solution and $(t, s) \in (\tfrac{1}{2} \mathbb{Z} / \mathbb{Z})^2$, or (ii) it corresponds to a symmetric Type-I boundary point. 
\end{itemize}
\end{proposition}

\begin{proof}
The relation between Type-I boundary points and the boundary $\overline{V}_{\mathbf{n},\mathbf{p}}(\tau) \setminus V_{\mathbf{n},\mathbf{p}}(\tau)$ has been discussed in Theorem~\ref{branch}. It suffices to consider the affine part $V_{\mathbf{n},\mathbf{p}}(\tau)$. 

Let $\alpha_i \in \pi_1(E_\tau \setminus \{p_1,\dots,p_r\},z_0)$ be small simple loops encircling $p_i$ for $i=1,\dots,r$. Given any two linearly independent solutions $f_1$ and $f_2$ (not necessarily of the ansatz form), and any loop $\gamma$, let $M_{\gamma}$ denote the corresponding monodromy matrix. 

The topological relation of the fundamental group
\[
    [\gamma_1,\gamma_2] \prod_{i=1}^r \alpha_i = \mathrm{Id}
\]
implies
\[
    M_{\gamma_1} M_{\gamma_2} M_{\gamma_1}^{-1} M_{\gamma_2}^{-1} = \prod_{i=1}^r M_{\alpha_i} =  \prod_{i=1}^r \Big( (-1)^{2n_i} \mathrm{Id} \Big) = \mathrm{Id}. 
\]
By the Wronskian argument from the previous proposition, $M_{\gamma_i}$ has determinant $1$, so its eigenvalues are $\lambda_i$ and $\lambda_i^{-1}$. Furthermore, if $\lambda_i \neq \pm 1$ for some $i$, the eigenvalues are distinct, meaning the matrices can be diagonalized. Since $M_{\gamma_1}$ and $M_{\gamma_2}$ commute, they can be diagonalized simultaneously, with eigenvectors given by solutions of the ansatz form.

By Proposition~\ref{prop:monodromy}, the degenerate case happens only when $2s \equiv 2t \equiv 0 \pmod{1}$, which implies $(t,s) \in (\tfrac{1}{2}\mathbb{Z} / \mathbb{Z})^2$. To further distinguish the two scenarios, we look at the monodromy matrix $M_{\gamma_i}$ (where $\gamma_i$ is the path from $z$ to $z+ \omega_i$):
\begin{enumerate}
    \item If $M_{\gamma_i}$ is a non-diagonalizable Jordan block for any $i$, the generalized Lam\'e equation admits only one ansatz solution $w_1(z)$. By standard ODE theory, the second linearly independent solution can be obtained by 
    \[
    w_2(z) = w_1(z) \int_{z_0}^z \frac{d\xi}{w_1(\xi)^2}.
    \]
    \item If $M_{\gamma_i} \in \{ \pm \mathrm{Id} \}$ for all $i$, then the generalized Lam\'e equation admits two linearly independent ansatz solutions. Furthermore, any linear combination of ansatz solutions is still quasi-periodic and hence of the ansatz form (see the proof of Proposition~\ref{exist_ansatz_sol}). This continuous family of solutions naturally forms the $\mathbb{P}^1$ fiber.
\end{enumerate}
The proof is finished.
\end{proof}

The $\mathbb{P}^1$-fiber case is classified as follows:

\begin{proposition} \label{prop:P1_codim_3}
The morphism $\overline{\mathcal{Y}}_{\mathbf{n},\mathbf{p}}(\tau) \rightarrow \overline{V}_{\mathbf{n},\mathbf{p}}(\tau)$ has $\mathbb{P}^1$-fiber(s) over a locus of codimension $3$ in the parameter space of $(\mathbf{p}; \tau)$.
\end{proposition}

\begin{proof}
Let $w_1$ and $w_2$ be any two linearly independent ansatz solutions on the $\mathbb{P}^1$-fiber. Then the ratio $f := w_1/w_2$ defines a degree $n$ meromorphic function $E_\tau \rightarrow \mathbb{P}^1$, since the projectivized monodromy of $f$ is trivial. Furthermore, $f$ is ramified at $p_i$ with ramification index $2n_i+1$, which exactly corresponds to the difference of the local exponents of the equation. A change of basis for the pair of ansatz solutions corresponds to post-composing the map $f$ with a M\"obius transformation in $\mathrm{PGL}_2(\mathbb{C})$. 

In the generic case where the images $f(p_i)$ are distinct, this yields a well-defined element
\[
    [f] \in \mathcal{H}_{1,n}(2n_1+1,\dots,2n_r+1),
\]
where $\mathcal{H}_{g,d}(\text{profile})$ denotes the Hurwitz space, the moduli space parameterizing degree $d$ morphisms from a genus $g$ curve to $\mathbb{P}^1$ with a prescribed ramification profile. 

Conversely, any $f \in \mathcal{H}_{1,n}(2n_1+1,\dots,2n_r+1)$ yields a generalized Lam\'e equation via the Schwarzian derivative. Precisely, the differential operator is given by
\[
    \frac{d^2}{dz^2} + \frac{1}{2}\{ f,z \} := \frac{d^2}{dz^2} + \frac{1}{2} \Bigg( \frac{f'''}{f'} - \frac{3}{2}\bigg( \frac{f''}{f'} \bigg)^2 \Bigg). 
\]

By the Riemann existence theorem, there is an \'etale (local bi-holomorphic) morphism from this Hurwitz space to the configuration space of the branch points on $\mathbb{P}^1$:
\[
    \mathcal{H}_{1,n}(2n_1+1,\dots,2n_r+1) \rightarrow \mathrm{Conf}_r(\mathbb{P}^1) / \mathrm{PGL}_2(\mathbb{C}).
\]

For the strata where branch points collide, the corresponding Hurwitz space admits an \'etale morphism to $\mathrm{Conf}_s(\mathbb{P}^1) / \mathrm{PGL}_2(\mathbb{C})$, where $s < r$ is the number of distinct branch points, and the correspondence to the generalized Lam\'e equations follows by a parallel argument. Since all degenerate strata have strictly lower dimensions, we conclude that the Hurwitz space, and hence the space of corresponding generalized Lam\'e equations, has maximal complex dimension $r-3$.

On the other hand, the base space of punctured tori, parameterized by $\tau$ and the $r$ marked points $\mathbf{p}$ up to global translation, has complex dimension $1 + (r-1) = r$. This establishes that the generic locus where the equation admits a $\mathbb{P}^1$-fiber has codimension $3$.
\end{proof}

\subsection{Counting generalized Lam\'e function}
In this subsection, we assume $\mathbf{n} \in \tfrac{1}{2}\mathbb{N}^r$.
\begin{definition}
The ramified fiber over the affine part $V$ in the classical case is called the Lam\'e function. Here we call it generalized Lam\'e function. Denote $L_1(\mathbf{n},\mathbf{p})$ be the number of generalized Lam\'e function and $L_2(\mathbf{n},\mathbf{p})$ be the number of $\mathbb{P}^1$-fiber counting with multiplicity. 
\end{definition}

Given generic $\mathbf{p}$, by Proposition~\ref{prop:P1_codim_3}, we have $L_2(\mathbf{n},\mathbf{p}) =0$, and the morphism $\overline{\mathcal{Y}}_{\mathbf{n},\mathbf{p}}(\tau) \rightarrow \overline{V}_{\mathbf{n},\mathbf{p}}(\tau)$ is a degree 2 morphism. 

\begin{theorem}[Number of generalized Lam\'e function] \label{t:GLF} 
For generic $\mathbf{p}$, we have 
\[
    L_1(\mathbf{n},\mathbf{p}) = \prod_{i=1}^r (2n_i +1).
\]
\end{theorem}
\begin{proof}
The proof relies on the degeneration formula. Let $\mathbf{n}_k = (n_1+n_2-k,n_3,\dots,n_r)$ and $\mathbf{p}_{12} = (p_{12},p_3,\dots,p_r)$. 
We claim the counting follows the recursive degeneration formula:
\[
    L_1(\mathbf{n},\mathbf{p}) = \sum_{k=0}^{2\min (n_1,n_2)} L_1(\mathbf{n}_k,\mathbf{p}_{12})
\]
with the initial condition that $L_1(n) = 2n+1$ for $n \geq 0$. This initial condition corresponds to the classical case studied in \cite{Chai_Lin_Wang_2015}. We emphasize that $L_1(\mathbf{m},\mathbf{p}) = \prod_{i=1}^r(2m_i +1)$ with some $m_i =0$ is conventionally included for the recursion.      

While the recursive formula is consistent with the degeneration formula of $\overline{\mathcal{Y}}_{\mathbf{n},\mathbf{p}}(\tau)$ in Corollary~\ref{cor:dege_special}, it is a nontrivial fact that this discrete counting of ramified points obeys the same recursion.

It suffices to consider the case $(t,s) = (0,0) \in ( \frac{1}{2}\mathbb{Z}/\mathbb{Z} )^2 $. We claim the following:

Given a point on the special fiber:
\[
    [(\mathbf{a}, h)] \in \{p_{12}^k\} \times \overline{\mathcal{Y}}_{\mathbf{n}_k,\mathbf{p}_{12}}^{c_k}(\tau),
\]
corresponding to the root configuration $\{ \mathbf{a} \} = \{p_{12}^k\} \times \{\mathbf{a}'\}$, where the fixed component $[(\mathbf{a}', h)]$ has $(t,s) = (0,0)$, the deformation 
\[
    a_\mu(\epsilon) = a_\mu + \sum_{m \geq 1} \alpha_{\mu m} \epsilon^m,
\]
induced by deforming the poles $p_i(\epsilon)$ subject to the vanishing sum condition 
\[
    \sum_{\mu=1}^n a_\mu(\epsilon) = \sum_{i=1}^r n_i p_i(\epsilon) = 0,
\]
must automatically satisfy $h(\epsilon,\tau) = h(0,\tau) = 0$. Note that this guarantees $(t,s) = 0$ in the deformation over $\epsilon$.

Expand $h$ as a power series $h(\epsilon, \tau) = \sum_{m=0}^\infty H_m(\tau) \epsilon^m$. By induction on the number of singularities, with the classical Lam\'e equation as the base case, a solution with $(t,s) = (0,0)$ exists on the special fiber ($\epsilon=0$) for all $\tau \in \mathbb{H}$. This existence implies the lowest-order term vanishes: $H_0(\tau) \equiv 0$ for all $\tau \in \mathbb{H}$.

We now proceed by induction on the power order $m$ to show $H_m(\tau) \equiv 0$. Assume $H_k(\tau) \equiv 0$ on $\mathbb{H}$ for all $k < m$. The series simplifies to:
\[
    h(\epsilon, \tau) = H_m(\tau)\epsilon^m + O(\epsilon^{m+1}).
\]
For any arbitrary initial modulus $\tau_0 \in \mathbb{H}$, there exists a (possible non-unique) path $\tau(\epsilon) = \tau_0 + O(\epsilon)$ along which the condition $(t,s) = (0,0)$ is maintained. Along this path, we have $h(\epsilon, \tau(\epsilon)) \equiv 0$ as a series in $\epsilon$.

Evaluate the series along this path yields:
\[
    h(\epsilon, \tau(\epsilon)) = H_m(\tau_0)\epsilon^m + O(\epsilon^{m+1}) \equiv 0.
\]
This forces the leading coefficient to vanish at the starting point: $H_m(\tau_0) = 0$.

Since $\tau_0$ was chosen arbitrarily from $\mathbb{H}$, we conclude that $H_m(\tau) \equiv 0$ across the entire upper half-plane. By this induction on $m$, we have $h(\epsilon, \tau) \equiv 0$ for all $\epsilon$, completing the proof of the claim.

For general $(t,s) \in (\tfrac{1}{2} \mathbb{Z} / \mathbb{Z})^2$, the argument is parallel for the shifted function $\tilde{h} := h - \eta \big( \sum_{\mu=1}^n a_\mu \big)$.
\end{proof}

\begin{remark} \label{rmk:representation_cg}
The degeneration formula $ L_1(\mathbf{n},\mathbf{p}) = \sum_{k=0}^n L_1(\mathbf{n}_k,\mathbf{p}')$ established in the proof of Theorem \ref{t:GLF} has a profound algebraic interpretation. It serves as the exact analytic realization of the Clebsch-Gordan decomposition for $\mathfrak{sl}_2(\mathbb{C})$ representations. 

In the representation-theoretic framework of the BGG category $\mathcal{O}$, the local algebraically completely integrable potentials correspond to the highest-weight modules $V_{2n_i}$ of dimension $2n_i + 1$. In the geometric limit where the singular sources $p_i$ collide, the global algebraic state space—defined by the tensor product $\bigotimes_{i=1}^r V_{2n_i}$—fractures into a direct sum of lower-weight irreducible representations. 

The linear recursion over $k$ captures this fracturing of the generalized Lam\'e curve into the lower-weight Lam\'e boundary strata. Because the total dimension is conserved during this flat deformation, the sum precisely yields the product $\prod_{i=1}^r (2n_i + 1)$. Thus, the function-theoretic scaling limit flawlessly mirrors the algebraic Clebsch-Gordan splitting.
\end{remark}

The above theorem gives interesting examples of non-complete intersection of elliptic equations. 
\begin{example}
For $\mathbf{n} = (\tfrac{1}{2})^4$ and $\mathbf{p}$ generic with $\sum_i p_i=0$, the underlying generalized Lam\'e curve is defined as
\[
    4 \zeta(a_1-a_2) - \sum_{i=1}^4\zeta(a_1-p_i) + \sum_{i=1}^4 \zeta(a_2-p_i)  =0,
\]
and the ramified points of the form $\{ a,-a \}$ is characterized by:
\[
\begin{split}
    4\zeta(2a) - \sum_{i=1}^4 \Big( \zeta(a-p_i) + \zeta(a+p_i)  \Big) =0;
    \\
    h = \frac{1}{4} \sum_{i=1}^4 \Big( \zeta(a-p_i) - \zeta(a + p_i) \Big)=0;
\end{split}
\]
The equations are intersect non-transversely for all $\mathbf{p}$, any solution of $a$ for the equation $h=0$ automatically solve the equation for GLC. This nontrivial statement follows directly from the existence of generalized Lam\'e function. We give an alternate proof.

Consider
\[
\begin{split}
 &\text{Elliptic curve } E: \quad y^2 = 4x^3 - g_2x - g_3.
 \\
 &\text{Parabola } Y: \quad y = Y(x) = Ax^2 + Bx + D.
    \\
 &\text{Generic points: } \quad \{ p_1,\dots,p_4 \} = Y \cap E.
\end{split}
\]
More precisely, $p_i = ( x_i, Y(x_i) )$ where the $x_i$ are the roots of the polynomial:
$$    P(x) = \prod_{i=1}^4 (x - x_i) := Y(x)^2 - (4x^3 - g_2x - g_3) = 0 .$$
Let $z = \wp(a)$. We rewrite the equation for GLC as:
\[
\begin{split}
 & 4\zeta(2a) - \sum_{i=1}^4 \Big( \zeta(a-p_i) + \zeta(a+p_i) \Big)
 \\
 & =\left( 8\zeta(a) + 2\frac{\wp''(a)}{\wp'(a)} \right) - \left( 8\zeta(a) + \wp'(a) \sum_{i=1}^4 \frac{1}{\wp(a)-\wp(p_i)} \right) =0
 \\
 & \Longleftrightarrow \frac{2\wp''(a)}{(\wp'(a))^2} = \sum_{i=1}^4 \frac{1}{\wp(a)-\wp(p_i)} \\
 &\Longleftrightarrow \frac{12z^2 - g_2}{4z^3 - g_2z - g_3} = \sum_{i=1}^4 \frac{1}{z - x_i}
 \\
 & \Longleftrightarrow P_{1}(z) := (12z^2 - g_2)P(z) - (4z^3 - g_2z - g_3)P'(z) = 0.
\end{split}
\]
Similarly, rewrite the $h=0$ equation as:
\[
\begin{split}
 & h = \frac{1}{4} \sum_{i=1}^4 \Big( \zeta(a-p_i) - \zeta(a + p_i) \Big) = 0
    \\
    & \Longleftrightarrow \frac{1}{4} \sum_{i=1}^4 \left( -2\zeta(p_i) + \frac{\wp'(p_i)}{\wp(a) - \wp(p_i)} \right) = 0
    \\
    & \Longleftrightarrow \sum_{i=1}^4 \frac{y_i}{z - x_i} = 2 \sum_{i=1}^4 \zeta(p_i) = -\frac{4}{A} \quad \left( \text{via the identity } \sum_{i=1}^4 \zeta(p_i) = -\frac{2}{A} \right)
    \\
    & \Longleftrightarrow P_2(z) := -\frac{4}{A} P(z) - S(z)= 0. \quad \Bigg(S(z) :=\sum_{i=1}^4 y_i \prod_{j \neq i} (z - x_j) \Bigg)
\end{split}
\]
A direct computation shows that 
\[
P_1(z) = Y(z) P_2(z).
\]
\end{example}

\subsection{Modular form}
When $(t,s) \in (\mathbb{C}/\mathbb{Z})^2 \setminus (\tfrac{1}{2} \mathbb{Z} / \mathbb{Z})^2$, the existence of solutions with prescribed monodromy data depends highly non-trivially on the modulus $\tau$. To solve this transcendental problem, the theory of pre-modular forms was developed in \cite{Lin_Wang_2017}.  In this subsection, we generalize the pre-modular form theory into $(\mathbf{n}, \mathbf{p})$-deformed version.

Formally, an analytic function $f(z;\tau)$ in $\mathbb{C} \times \mathbb{H}$ is called a pre-modular form of weight $k\in \mathbb{N}$ if it satisfies:
\begin{enumerate}
    \item For any fixed $\tau \in \mathbb{H}$, the function $f(z; \tau)$ is analytic in $z$ and depends only on $z \pmod{\Lambda_\tau} \in E_\tau$.
    \item For any fixed torsion type $z \pmod{\Lambda_\tau} \in E_\tau[N]$, the function $f(z; \tau)$ as a function of $\tau$ is a modular form of weight $k$ with respect to $\Gamma(N)$.
\end{enumerate}

The foundational example is the Hecke function. Let $z = \tilde{t}+\tilde{s}\tau$, it is defined as:
\[
Z(z;\tau) = Z_{\tilde{t},\tilde{s}}(\tau) := \zeta(\tilde{t}+\tilde{s}\tau; \tau) - \tilde{t} \eta_1(\tau) - \tilde{s} \eta_2(\tau).
\]
This function is double periodic in $\tilde{t}$ and $\tilde{s}$.
As a result, $Z_{t,s}(\tau)$ is well-defined and holomorphic on $t, s \in \mathbb{C}/\mathbb{Z}$, satisfying the criteria of a pre-modular form (of weight 1).

Define the $(\mathbf{n},\mathbf{p})$-deformed Hecke function:
\[
    Z_{\mathbf{n},\mathbf{p};t,s}(\tau):= \frac{1}{n} \zeta_{\mathbf{n},\mathbf{p}}(t+s\tau) - t \eta_1 -s\eta_2,
\]
which does not satisfy the criterion (2). Due to its modularity nature, we still call it a pre-modular form. See the remark below.
\begin{remark} \label{r:modular_form}
In the classical setting, for a fixed integer $N \geq 2$, one can construct a modular form from a pre-modular form $f(z;\tau)$ of weight $k$ by taking the product over primitive $N$-torsion points:
\[
    M_{f,N}(\tau) := \prod_{ \substack{0\leq k_1,k_2<N \\ \gcd(k_1,k_2,N)=1} } f \Big(\frac{k_1+k_2\tau}{N};\tau \Big).
\]

For the $(\mathbf{n},\mathbf{p})$-deformed Hecke function, the presence of a generic deformation parameter $\mathbf{p}$ breaks the Condition 2. Nevertheless, for a fixed $N \geq 2$, one can still associate a modular form to it by extending this product to simultaneously average over the $N$-torsions of $\mathbf{p}$. Because of this analogous construction, it justifies retaining the terminology ``pre-modular form'' for the deformed function.
\end{remark}

Define the $(\mathbf{n}, \mathbf{p})$-deformed fundamental rational function $\mathbf{z}$ as:
\[
\begin{split}
    \mathbf{z}_{\mathbf{n},\mathbf{p}}(*;\tau): \overline{\mathcal{Y}}_{\mathbf{n},\mathbf{p}}(\tau) \rightarrow \mathbb{C}, \quad
    [(\mathbf{a},h)]  \mapsto \frac{1}{n} \zeta_{\mathbf{n},\mathbf{p}}( \sigma_{\mathbf{n},\mathbf{p}}(\mathbf{a})) - h .
\end{split}
\]

We give a direct consequence of Proposition \ref{prop:monodromy}:
\begin{proposition} \label{p:GH}
For $\mathbf{n} \in \mathbb{Z}^r$, the existence of an ansatz solution corresponds to $[(\mathbf{a},h)] \in \overline{\mathcal{Y}}_{\mathbf{n},\mathbf{p}}(\tau)$ with monodromy data $(t,s) \in (\mathbb{C}/\mathbb{Z})^2$ is equivalent to
\begin{equation} \label{e:Green_Hecke}
    Z_{\mathbf{n},\mathbf{p};t,s} (\tau) = \mathbf{z}_{\mathbf{n},\mathbf{p}}([(\mathbf{a},h)];\tau).
\end{equation}
\end{proposition}
\begin{proof}
For $\mathbf{n} \in \mathbb{Z}^r$, the monodromy data is independent of the chosen fundamental periods $\gamma_i$. Furthermore, because the twist factor is trivial, $e^{\Theta(\gamma_i)} = 1$, the twisted monodromy data exactly coincides with the standard monodromy data.
\end{proof}

The construction of the pre-modular form works to general $\mathbf{n}$, although in the non-integer case it governs the twisted monodromy data. To construct it, we first study the characteristic polynomial of $\mathbf{z}_{\mathbf{n},\mathbf{p}}$. Denote 
\[
    \mathcal{K} := \mathbb{C}\Big(g_2, g_3, \{\wp(p_i), \wp'(p_i)\}_{i=1}^r, \wp(c), \wp'(c)\Big),
\]
and let $\pi: \overline{\mathcal{Y}}_{\mathbf{n},\mathbf{p}}(\tau) \rightarrow \overline{Y}_{\mathbf{n},\mathbf{p}}(\tau)$ be the natural projection. Recall that $\pi$ is an isomorphism over the interior and a finite morphism on the boundary. Because the following lemma computes the characteristic polynomial over a generic $c \in E_\tau$, working directly with $\overline{Y}_{\mathbf{n},\mathbf{p}}(\tau)$ does not lose any information. Consequently, we will abuse notation by writing $\mathbf{z}_{\mathbf{n},\mathbf{p}}(\mathbf{a};\tau)$ in place of $\mathbf{z}_{\mathbf{n},\mathbf{p}}([(\mathbf{a},h)];\tau)$.

\begin{lemma}[Characteristic polynomial of $\mathbf{z}_{\mathbf{n},\mathbf{p}}$] \label{l:char_poly}
Let $\sigma_{\mathbf{n},\mathbf{p}}: \overline{{Y}}_{\mathbf{n},\mathbf{p}}(\tau) \to E_{\tau}$ be the addition map of degree $N$. The variable $\mathbf{z}_{\mathbf{n},\mathbf{p}}$ generates a characteristic polynomial of degree $N$ over $\mathcal{K}$:
\[
\begin{split}
    W_{\mathbf{n},\mathbf{p}}(Z) &:= \prod_{k=1}^N \Big(Z - \mathbf{z}_{\mathbf{n},\mathbf{p}}(\mathbf{a}^{(k)}; \tau) \Big) \\
    &= Z^N + \sum_{j=1}^N (-1)^j W_j(c; \tau) Z^{N-j} \in \mathcal{K}[Z],
\end{split}
\]
where the product is taken over the $N$ generic preimages $\{[\mathbf{a}^{(k)}]\} \in \sigma_{\mathbf{n},\mathbf{p}}^{-1}(c)$. 
The coefficients $W_j(c; \tau)$ are elliptic functions in $c$ satisfying the following properties:
\begin{enumerate}
    \item $W_j(c; \tau)$ is weighted-homogeneous of degree $j$, where $\wp, \wp', g_2, g_3$ are of degree $2, 3, 4, 6$ respectively.
    \item $W_j$ is regular everywhere except at the divisor $D \subset E_{\tau}$, where
    \[
        D = \{[p_1], \dots, [p_r]\} \cup \left\{ \left[ \sum_{i=1}^r k_i p_i \right] \;\middle|\; \sum_{i=1}^r k_i = n, \text{ and } k_i \leq 2n_i \text{ whenever } n_i \in \tfrac{1}{2}\mathbb{N} \right\}.
    \]
\end{enumerate}
\end{lemma}
\begin{proof}
For the weight statement, note that $\mathbf{z}_{\mathbf{n},\mathbf{p}}$ is of weight 1 and $W_j$ is the elementary symmetric polynomial of $\mathbf{z}_{\mathbf{n},\mathbf{p}}$. 

For the divisor $D$, note that in the definition of $\mathbf{z}_{\mathbf{n},\mathbf{p}}$, the first term gives a pole at $[p_i]$, and the second term $h$ gives poles at 
\[
        \left\{ \left[ \sum_{i=1}^r k_i p_i \right] \;\middle|\; \sum_{i=1}^r k_i = n, \text{ and } k_i \leq 2n_i \text{ whenever } n_i \in \tfrac{1}{2}\mathbb{N} \right\}
\]
by the classification of Type-I boundary points, see Theorem~\ref{special-pt}. 
\end{proof}

\begin{remark}
Unlike the classical $\mathbf{p}=\mathbf{0}$ case, $\mathbf{z}_{\mathbf{n},\mathbf{p}}$ might not be a primitive element with respect to the field extension $\mathcal{K}(E_\tau) \hookrightarrow \mathcal{K}(\overline{Y}_{\mathbf{n},\mathbf{p}}(\tau))$. In other words, its minimal polynomial takes the form $Q^k$ for $k\geq 2$. 

For instance, in Example~\ref{eg:GLC_half_periods}, the generalized Lam\'e curve decomposes as $\overline{Y}_{\mathbf{n},\mathbf{p}}(\tau) = 2\Delta_0 \cup \Delta_1 \cup \Delta_2 \cup \Delta_3$. On these components, $\mathbf{z}_{\mathbf{n},\mathbf{p}}$ is not a primitive element, and the characteristic polynomial factors as:
\[
\begin{split}
    W_{\mathbf{n},\mathbf{p}}(Z) &= \prod_{i=1}^3 \left( Z - \frac{1}{4} \frac{\wp''(c)}{\wp'(c)} + \frac{1}{2} \frac{\wp'(c)}{\wp(c) - e_i} \right)^2
    \\
    & = \prod_{i=1}^3 \left( Z - \frac{1}{4} \frac{\wp''(c - \omega_i/2)}{\wp'(c - \omega_i/2)} \right)^2.
\end{split}
\]
\end{remark}

\begin{definition}[Pre-modular form] 
Fix twisted monodromy data $(t,s) \in (\mathbb{C}/\mathbb{Z})^2$, and let $c = t + s\tau$. The pre-modular form is defined as:
\[
    \Phi_{\mathbf{n},\mathbf{p}; t, s}(\tau) := W_{\mathbf{n},\mathbf{p}} \Big( Z_{\mathbf{n},\mathbf{p};t,s}(\tau) \Big).
\]
\end{definition}
This corollary follows directly from the degeneration theorem:
\begin{corollary}\label{cor:modular_decomp} 
Let $\omega = m_1\omega_1 + m_2 \omega_2 \in \Lambda_\tau$. Then
\[
    \lim_{\epsilon \rightarrow 0} \Phi_{\mathbf{n},\mathbf{p}_{12, \epsilon}^\omega;t,s}(\tau) = \prod_{k=0}^{n-1} \Big( \Phi_{\mathbf{n}_k,\mathbf{p}_{12};t',s'}(\tau) \Big)^{m_k},
\]
where $\mathbf{n}_{12}^{(k)}$ and $\mathbf{p}_{12}^{\omega}$ are defined in Theorem~\ref{thm:dege_generic}, $m_k$ is the multiplicity of the corresponding component and
$t' = t + m_1 n_1, \ s'= s + m_2 n_1.$

In particular, we have
\[
    \lim_{\mathbf{p} \rightarrow \mathbf{0}} \Phi_{\mathbf{n},\mathbf{p};t,s}(\tau) = \prod_{k=0}^{n-1} \Big( \Phi_{n-k,0;t,s}(\tau) \Big)^{m_k(\mathbf{n})},
\]
where $m_k(\mathbf{n})$ is defined in Corollary~\ref{Cor:special_fiber_p=0}.
\end{corollary}
\begin{proof}
Note that the product is up to $n-1$, since the degree of the addition map on the weight 0 (underlying) GLC is zero.

Recall that
\[
    \mathbf{p}_{12}^\omega = (p_{12} + \omega +\epsilon, p_{12}-\epsilon,p_3,\dots,p_r), \quad \mathbf{p}_{12} := \mathbf{p}_{12}^0.
\]
We have
\[
\begin{split}
    \widetilde{\sigma}_{\mathbf{n},\mathbf{p}_{12, \epsilon}}(\mathbf{a}) &= \widetilde{\sigma}_{\mathbf{n},\mathbf{p}_{12, \epsilon}^\omega}(\mathbf{a}) + n_1 \omega \\
    &= (\tilde{t} \omega_1 + \tilde{s} \omega_2) + n_1 (m_1 \omega_1 + m_2 \omega_2) \\
    &= (\tilde{t} + n_1 m_1) \omega_1 + (\tilde{s} + n_1 m_2) \omega_2.
\end{split}
\]
Similarly, applying the shift operator with $\eta(\omega) = m_1 \eta_1 + m_2 \eta_2$,
\[
\begin{split}
    h' &= S_{n_1 \eta(\omega)} h \\
    &= (\tilde{t} \eta_1 + \tilde{s} \eta_2) + n_1(m_1 \eta_1 + m_2 \eta_2) \\
    &= (\tilde{t} + n_1 m_1) \eta_1 + (\tilde{s} + n_1 m_2) \eta_2.
\end{split}
\]
We conclude
\[
    t' = t + n_1 m_1, \qquad
    s' = s + n_1 m_2.
\]
This completes the proof.
\end{proof}

\subsection{Twisted isomonodromic deformation}
The family of generalized Lam\'e curves provides the machinery to analyze twisted isomonodromic deformations.  Through the degeneration formula on special fibers, it relates different types of generalized Lam\'e equations within one continuous family.
\begin{remark}
The "twisted" terminology comes from two aspects. First, when $n_i \notin \mathbb{N}$, the data $(t,s)$ only recovers the monodromy matrices $M_{\gamma_1}$ and $M_{\gamma_2}$ up to a scaling factor. Second, in the degeneration formula, the $h$-shift in the special fiber results in the corresponding monodromy shift. 
\end{remark}

To our knowledge in the existing literature where isomonodromic deformation was studied, they are typically restricted to configurations with poles that remained distinct \cite{Kawai_2003}, or with small total weight, or with symmetric assumptions; see Example~\ref{eg:Hitchin}.

Our approach avoids these limitations entirely. The family of GLCs naturally allows for continuous deformations in a general setting, without relying on poles distinct or symmetric conditions. While a rigorous analysis—such as studying the associated modular forms or combinatorial models (via a correspondence between abelian differentials and flat surfaces)—is a clear direction for future research, it falls outside the scope of this paper. Consequently, we restrict ourselves to a general discussion and outline possible approaches.

To make this setup precise, we study the deformation of the following specific locus. For any fixed $(t,s) \in (\mathbb{C}/\mathbb{Z})^2 \setminus (\tfrac{1}{2} \mathbb{Z} / \mathbb{Z})^2$, we define
\[
\overline{\mathcal{Y}}^{t,s}_{\mathbf{n}, \mathbf{p}}(\tau) := 
\left\{
    [(\mathbf{a},h)] \in \overline{\mathcal{Y}}_{\mathbf{n},\mathbf{p}}(\tau)
    \;\middle|\;
    [(\mathbf{a},h)] \text{ has twisted monodromy } (t,s)
\right\}.
\]
Our objective is to fix the weight $\mathbf{n}$ and deform the configuration $\mathbf{p}$. Geometrically, this fits into a family over the parameter space $(\mathbf{p},\tau)$:
\[
\begin{tikzcd}
\overline{\mathcal{Y}}^{t,s}_{\mathbf{n},\mathbf{p}}(\tau) \arrow[r, hook] \arrow[d] & \overline{\mathcal{Y}}^{t,s}_{\mathbf{n}} \arrow[d] \\
(\mathbf{p}, \tau) \arrow[r, mapsto]                                                 & \mathbb{C}^r \times \mathbb{H}
\end{tikzcd}.
\]
For a fixed $\tau$ and generic $\mathbf{p}$, the fiber is $\overline{\mathcal{Y}}_{\mathbf{n},\mathbf{p}}^{t,s}(\tau)$. 

Conversely, when the points $\mathbf{p}$ collide, the corresponding special fiber is the union of components described in the degeneration formula. 

\begin{example}[Hitchin] \label{eg:Hitchin} Consider the case $\mathbf{n} = (\frac{1}{2}, \frac{1}{2})$ and, up to translation, $\mathbf{p} = (p, -p)$. The defining equation of $Y_{\mathbf{n},\mathbf{p}}(\tau)$ is simply
\[
    \frac{1}{2} \Big( \zeta(t+s\tau + p) + \zeta(t+s\tau -p) \Big) = t \eta_1(\tau) + s \eta_2(\tau).
\]
Using the identity $\zeta(u+v) - \zeta(u) - \zeta(v) = \frac{1}{2} \frac{\wp'(u) - \wp'(v)}{\wp(u)-\wp(v)}$, we have the relation between $p$ and $\tau$:
\[
    \wp(p) = \wp(t+s\tau) + \frac{\wp'(t+s\tau)}{2(\zeta(t+s\tau) - t \eta_1 - s\eta_2)}.
\]
Recognizing the term in the denominator as the Hecke function $Z_{t,s}(\tau)$, this provides a highly explicit algebraic characterization of the space. We can explicitly describe the total space over the base $\mathbb{C} \times \mathbb{H}$ as:
\[
\overline{\mathcal{Y}}^{t,s}_{\mathbf{n}} = \Bigg\{ (p, \tau, [ (\mathbf{a},h) ]) 
\;\Bigg|\; 2\Big( \wp(p) - \wp(t+s\tau) \Big) Z_{t,s}(\tau) = \wp'(t+s\tau) \Bigg\}.
\]
\end{example}

\vspace{5mm}

Beyond such exact cases, obtaining a closed algebraic relation between $\mathbf{p}$ and $\tau$ is structurally challenging in general. Nevertheless, their dependence is strictly governed by differential equations. Consider a smooth deformation path $\epsilon \mapsto \mathbf{p}(\epsilon; \tau) = (p_1(\epsilon; \tau), \dots, p_r(\epsilon; \tau))$. By applying the total derivative with respect to $\epsilon$, we formally refer to this governing equation, which describes the analytic continuation of $\tau$ along the path, as the twisted isomonodromy deformation:
\[\displaystyle
    \frac{d \tau}{d \epsilon} = - \frac{ \sum_{i=1}^r \frac{\partial \Phi_{\mathbf{n},\mathbf{p};t,s}}{\partial p_i} \frac{\partial p_i}{\partial \epsilon} }{ \frac{\partial \Phi_{\mathbf{n},\mathbf{p};t,s}}{\partial \tau} + \sum_{i=1}^r \frac{\partial \Phi_{\mathbf{n},\mathbf{p};t,s}}{\partial p_i} \frac{\partial p_i}{\partial \tau} }.
\]

To formulate another application of the degeneration theorem, we first analyze the boundary behavior at the cusps. The asymptotic expansion of the modular form of the classical Lam'e equation at the cusps $\tau \in\mathbb{Q}\cup \{\infty\}$ shows that the leading boundary constants vanish if and only if the (twisted) monodromy data lies in the $SL_2(\mathbb{Z})$-orbit of $s \in \frac{1}{2}\mathbb{Z}$. We define this locus as:$$\mathcal{B} := \left\{ (t,s) \in (\mathbb{R}/\mathbb{Z})^2 \;\middle|\; ct + ds \in \tfrac{1}{2}\mathbb{Z} / \mathbb{Z} \text{ for coprime integers } c,d \right\}.$$With this locus identified, we obtain the following corollary governing the existence of twisted isomonodromic deformations:

\begin{corollary}\label{cor:isomo_deform}
For $(t,s) \in (\mathbb{R}/\mathbb{Z})^2 \setminus \mathcal{B}$, fix $[(\mathbf{a}_0, h_0)] \in \overline{\mathcal{Y}}^{t,s}_{\mathbf{n}, \mathbf{p}_0}(\tau_0)$. A twisted isomonodromy deformation exists along a generic path $\mathbf{p}_0 \to \mathbf{0}$, analytically continuing to a limit $[(\mathbf{a}^*, h^*)] \in \overline{\mathcal{Y}}^{t,s}_{k,0}(\tau^*)$ for some $0 \leq k \leq n$ and $\tau^* \in \mathbb{H}$.
\end{corollary}

\begin{proof}
The total space $\overline{\mathcal{Y}}^{t,s}_{\mathbf{n}}$ is a quasi-projective variety over $\mathbf{p}$. The irreducible component containing the point $([(\mathbf{a}_0, h_0)], \tau_0)$ projects dominantly onto the base $\mathbf{p}$-space. 

By the definition of $\mathcal{B}$, the number of finite points of $\overline{\mathcal{Y}}^{t,s}_{\mathbf{n}}$ in the special fiber at $\mathbf{p} = \mathbf{0}$ exactly coincides with the number of sheets of the covering over a small open neighborhood $\mathbf{0} \in U$, since no solutions of $\tau$ will escape to the cusp at $\mathbf{p}=0$. Consequently, a generic path $\mathbf{p}_0 \to \mathbf{0}$ lifts to a continuous trajectory within this irreducible component. This lifted trajectory is the twisted isomonodromy deformation.

Furthermore, since the limit $p_i \to 0$ involves no lattice translations, $(t,s)$ remains constant. By the degeneration theorem, the special fiber at $\mathbf{p} = \mathbf{0}$ is the scheme-theoretic union of the classical and weight-zero spaces $\overline{\mathcal{Y}}^{t,s}_{k,0}(\tau^*)$ for $0 \leq k \leq n$ with corresponding multiplicity. The endpoint of the lifted path lies in one of these components. This finishes the proof.
\end{proof}

\begin{remark}
Consider the path $\mathbf{p}(\epsilon) = (0, \epsilon\omega_i, -\epsilon\omega_i)$ parameterized by $\epsilon$, with weights $\mathbf{n} = (n, \frac{1}{2}, \frac{1}{2})$. The degeneration formula explicitly characterizes the two special boundary fibers at $\epsilon=0$ and $\epsilon=\frac{1}{2}$; the limits at these two ends yield a continuous deformation that relates components of different weights.

\small
\begin{figure}[ht!]
\centering
\begin{tikzcd}[row sep=large, column sep=normal]
\begin{aligned} 
    &\overline{\mathcal{Y}}_{n+1, 0}^{t,s} \cup 2 \big(\{0\} \times \overline{\mathcal{Y}}_{n, 0}^{t,s}\big) \\ 
    &\quad \cup \big(\{0\}^2 \times \overline{\mathcal{Y}}_{n-1, 0}^{t,s}\big) 
\end{aligned} \arrow[r] \arrow[d]  
    & \overline{\mathcal{Y}}_{\mathbf{n}, \mathbf{p}(\epsilon)}^{t,s} \arrow[d]  
    & \begin{aligned} 
    &S_{-\frac{1}{2}\eta_i} \Big( \overline{\mathcal{Y}}_{(n,1), (0, \frac{1}{2}\omega_i)}^{t,s} \cup \\ 
    &\quad \big( \{\tfrac{1}{2}\omega_i\} \times \overline{\mathcal{Y}}_{n, 0}^{t,s} \big) \Big) 
\end{aligned} \arrow[d] \arrow[l] \\
\mathbf{p}(0) \arrow[r]
    & \mathbf{p}(\epsilon)  
    & \mathbf{p}(\tfrac{1}{2}) \arrow[l]
\end{tikzcd}
\end{figure}
\normalsize

Suppose we fix a monodromy data $(t,s)\notin \mathcal{B}$ that gives an ansatz solution on the boundary component $\overline{\mathcal{Y}}_{n+1, 0}^{t,s}$ at $\epsilon=0$. We deform this configuration into the generic fiber. As $\epsilon \to \frac{1}{2}$, there always exists a (possibly non-unique) choice of deformation direction such that the trajectory specializes to the $\overline{\mathcal{Y}}_{n, 0}^{t,s}$ component (with monodromy shift).

The existence of this continuous path connecting $\overline{\mathcal{Y}}_{n+1, 0}^{t,s}$ to $\mathcal{S}_v\Big(\overline{\mathcal{Y}}_{n, 0}^{t,s}\Big)$ for some explicit shift $v$ is analytically non-trivial, but it can be established using the combinatorial model of \cite{Wu_combinatorial}. This provides a continuous analytic foundation for the cell-attachment procedure discussed in \cite{Chou_Wang_Wu, Eremenko_Mondello_Panov_2023}. 

By iterating this process and applying Corollary~\ref{cor:isomo_deform}, we deduce the following remarkable statement: for any given $\mathbf{n}$ and fixed global monodromy $(t,s) \notin \mathcal{B}$, there always exists a twisted isomonodromic deformation to a limit $\overline{\mathcal{Y}}^{t,s}_{1,0}$ or $\overline{\mathcal{Y}}^{t,s}_{0,0}$.
\end{remark}

\section{Rational and elliptic symmetric equilibrium systems}
In this section, we describe a framework for counting solutions of higher-order symmetric elliptic equilibrium systems. The framework follows the argument used to compute the degree of the addition map and may be viewed as an elliptic Bézout-type counting method. We introduce generic constants into the equilibrium equations and study the asymptotic limit as these parameters tend to infinity along a suitable direction. After excluding other limiting configurations, the solutions cluster near the poles of the equations. Extracting the leading-order local behavior reduces the counting problem to symmetric rational equilibrium systems at generic values.

We then follow these solutions as the deformation parameters tend to zero. In general, some configurations may degenerate or escape to the boundary, so recovering the count at zero requires determining their local contributions with multiplicity. The nonlinear ODE method developed in the following subsection provides a way to analyze the rational systems arising in this process. In this paper, we carry out the argument for two moving pole pairs: no solutions escape to the boundary, and we obtain Treibich's count of 27, with multiplicity, on every smooth elliptic curve.
\subsection{ODE methods and the rational symmetric equilibrium system}
Consider a general set up involving $r$ variables $\alpha_1, \dots, \alpha_r$. Let $F_1$, $F_2$ and $U$ be rational functions. Consider the system:
\begin{equation} \label{eqn:rational_equilibrium}
    \sum_{\nu \neq \mu } \Big( F_1(\alpha_\mu - \alpha_\nu) + F_2(\alpha_\mu + \alpha_\nu) \Big) + U(\alpha_\mu), \text{ for }\mu=1,\dots,r.
\end{equation}
We describe a general method to solve this system. 

Let $q(z) = \prod_{i=1}^r (z- \alpha_i)$ be the polynomial with roots given by the system. 
Observe that 
\[
\begin{split}
    \sum_{\nu \neq \mu} (\alpha_\mu-\alpha_\nu)^l &= S_{1,l}(q',q'', \dots, q^{(l+1)}) (z = \alpha_\mu)
    \\
    \sum_{\nu \neq \mu} (\alpha_\mu + \alpha_\nu)^l &= S_{2,l}(q', q'',\dots,q^{(l+1)}) (z=\alpha_\mu),
\end{split}
\]
where $S_{1,l}$ and $S_{2,l}$ are polynomials with $l+1$-variables. $S_{1,l}$ and $S_{2,l}$ can be computed recursively:
\begin{lemma} Fix a root $\alpha_\mu$. Let 
\[
\begin{split}
c_j = \frac{(-1)^j}{(j+1)!} q^{(j+1)}(\alpha_\mu), \qquad
d_j = \frac{1}{j!} q^{(j)}(-\alpha_\mu), 
\end{split}
\]
with $c_j = d_j =0$ for $j<0$. 
\begin{itemize}
\item Inverse Power Sums ($k = m > 0$):
\[
\begin{split}
& m c_m + \sum_{i=1}^m c_{m-i} S_{1,i} = 0
\\
& m d_m + \sum_{i=1}^m d_{m-i} \tilde{S}_{2,i} = 0, \quad \text{where } S_{2,m} = \tilde{S}_{2,m} - \frac{1}{(2\alpha_\mu)^m}
\end{split}
\]
\item Positive Power Sums ($k = -m < 0$):$$m c_{r-1-m} + \sum_{i=1}^m c_{r-1-m+i} S_{1,-i} = 0$$$$m d_{r-m} + \sum_{i=1}^m d_{r-m+i} \tilde{S}_{2,-i} = 0, \quad \text{where } S_{2,-m} = \tilde{S}_{2,-m} - (2\alpha_\mu)^m$$    
\end{itemize}
\end{lemma}

The formula is independent of the choice of root $\alpha_\mu$. We rewrite the system~\eqref{eqn:rational_equilibrium} as the following non-linear ODE:
\[
    \mathcal{P}[q](z) \equiv 0 \quad \text{mod}(q(z)). 
\]
This reduces the complicated system into a simpler system of reminder terms. Precisely:
\[
    \mathcal{P}[q](z) = Q(z) q(z) + R(z),
\]
where $R(z) = \sum_{i=0}^{r-1} R_i z^i$ is the reminder polynomial that are forced to be zero.

Let $q(z) = z^r + \sum_{i=1}^r c_i z^{r-i}$. The problem reduces to the Gr\"obner basis computation:
\[
\begin{split}
    \mathcal{J} = \Big\langle R_0, \dots, R_{k-1}, t \mathcal{S}_q -1\Big\rangle \subset \mathbb{C}[c_0,\dots,c_{r-1},t],
\end{split}
\]
where $\mathcal{S}_q$ is the product of at most three terms (depending on the system):
\begin{itemize}
    \item $\Delta_q$, the discriminant of $q$. This is equivalent to the condition $\alpha_\mu \neq \alpha_\nu$ for $\mu \neq \nu$.
    \item $c_r$. This is equivalent to assume $\alpha_\mu \neq 0$;
    \item ${\rm Result}_z (q(z), q(-z))$, the resultant of $q(z)$ and $q(-z)$. This is equivalent to assume $\alpha_\mu \neq - \alpha_\nu$ for $\mu \neq \nu$.
\end{itemize}
For example, if $F_1$ and $U$ has negative term, we take $S_q = \Delta_q  c_r$.

We use the following example to explain the method. It will be useful in the proof of Theorem~\ref{thm:Treibich_conj}.
\begin{lemma} \label{lem:rational_boundary}
The system 
\[
\begin{cases}
     \displaystyle\frac{1}{(\alpha_1-\alpha_2)^3} + \frac{1}{(\alpha_1 + \alpha_2)^3} + \frac{x}{\alpha_1^3}=0,
    \\
    \displaystyle\frac{1}{(\alpha_2-\alpha_1)^3} + \frac{1}{(\alpha_2 + \alpha_1)^3} + \frac{x}{\alpha_2^3}=0.
\end{cases}
\]
has no solution for $x \in \mathbb{C} \setminus \{  \frac{1}{2} \}$. For $x = \frac{1}{2}$, the solution takes the form:
\[
    (\alpha_1,\alpha_2) = (c,\pm i \, c), \quad c \in \mathbb{C}^*.
\]
\end{lemma}
\begin{proof}
Write $q(z) = z^2 + c_1 z + c_2$. A direct computation gives
\[
\begin{split}
\mathcal P[q](z)
&=z^3(2z+c_1)^3-c_1^3z^3-xc_1^3(2z+c_1)^3\\
&=Q(z)q(z)+R(z), \quad R(z) = R_1 z + R_0,
\end{split}
\]
where 
\[
\begin{split}
    R_0&=-2c_2(c_1^4-3c_1^2c_2+4c_2^2)
    -xc_1^4(c_1^2-4c_2),\\
    R_1&=-2c_1\bigl((1+x)c_1^2(c_1^2-4c_2)+6c_2^2\bigr).
\end{split}
\]
Consider
\[
    \mathcal J=\langle R_0,R_1,t\mathcal S_q-1\rangle
    \subset\mathbb C[t,c_2,c_1,x], \quad \mathcal{S}_q = c_1c_2(c_1^2 - 4 c_2)
\]
A Gröbner basis computation with lexicographic order
\[
    t\succ c_2\succ c_1\succ x
\]
gives
\[
    \mathcal J\cap\mathbb C[x]=\langle(2x-1)^2\rangle.
\]
We conclude that the only possible value for $x$ is $\frac{1}{2}$. Further computation shows that $(\alpha_1, \alpha_2) = (c, \pm i \, c)$.
\end{proof}

\subsection{Finite gap KdV hierarchy}
We recall some results on elliptic finite-gap potentials. A potential $f(z)$ is called finite-gap if there exists an odd-order differential
operator $P_{2g+1}$ commuting with $\frac{d^2}{dz^2}-f(z)$. Adding a constant $B$ does not change this
property.

Airault, McKean, and Moser \cite{Airault_McKean_Moser_1977}
studied potentials of the form
\[
    I(z)=2\sum_{i=1}^N\wp(z-p_i),
\]
with $p_i \neq p_j$.
The locus that characterize the finite-gap potentials is characterize the following equations
\[
 L_N:=
 \left\{
 \mathbf p\in E_\tau^N
 \ \middle|\
 \sum_{\substack{j=1\\j\ne i}}^N\wp'(p_i-p_j)=0
 \quad(1\leq i\leq N)
 \right\}.
\]

Treibich and Verdier \cite{Treibich_Verdier_1990} subsequently
developed the geometry of elliptic solitons and the compactification
$\overline L_N$ of the AMM locus with colliding poles.
They prove that the potentials originally considered by Darboux,
\[
 I(z)=
 \sum_{i=0}^3 n_i(n_i+1)
 \wp\Big(z-\frac{w_i}{2}\Big),
\]
are finite-gap for every choice
$(n_0,n_1,n_2,n_3)\in\mathbb Z_{\geq0}^4$. The triangular numbers record the number of poles in each collision:
a pole with coefficient $n_i(n_i+1)$ accounts for
$T_{n_i}=n_i(n_i+1)/2$ original poles of coefficient $2$.
Thus the corresponding configuration belongs to $\overline L_N$ with
\[
 N=\sum_{i=0}^3T_{n_i}
   =\frac12\sum_{i=0}^3n_i(n_i+1).
\]

Later, Gesztesy and Weikard \cite{Gesztesy_Weikard_Picard_1996}
characterized elliptic finite-gap potentials through the
meromorphicity of solutions for every value of the spectral parameter.
Gesztesy, Unterkofler, and Weikard
\cite[Theorem~2.11]{Gesztesy_Unterkofler_Karl_2006}
gave an explicit characterization: up to an additive constant,
every elliptic finite-gap potential has the form
\[
 I(z)=\sum_{i=1}^r n_i(n_i+1)\wp(z-p_i),
 \qquad n_i\in\mathbb Z_{>0},
\]
with pairwise distinct poles satisfying
\begin{equation}\label{e:KdV}
 \sum_{\substack{j=1\\j\ne i}}^r
 n_j(n_j+1)\wp^{(2k-1)}(p_i-p_j)=0,
 \qquad
 1\leq k\leq n_i,\quad 1\leq i\leq r.
\end{equation}
Conversely, these conditions imply that $I(z)+B$ is a finite-gap
KdV potential.

\begin{remark}
The formal dimension count for \eqref{e:KdV} is
\[
 r-\sum_{i=1}^r n_i+\min_i n_i.
\]
Further dependencies may occur, so this count determines neither the actual dimension nor nonemptiness. Treibich's symmetric-pair family below automatically satisfies the half-period constraints and reduces the remaining conditions to $r$ equations in $r$ moving-pole parameters, giving a natural counting problem.
\end{remark}

\subsection{Completion of Treibich's count}
Very recently, Treibich \cite{Treibich_2025} studied a counting problem for
even elliptic finite-gap potentials of the form
\[
 I_{\mathbf p}(z)=
 \sum_{i=0}^3 n_i(n_i+1)\wp\Big(z-\frac{w_i}{2}\Big)
 +2\sum_{\mu=1}^r
 \bigl(\wp(z-p_\mu)+\wp(z+p_\mu)\bigr).
\]
We assume that the moving poles $\{\pm p_1,\ldots,\pm p_r\}$
are pairwise distinct and are not the half-periods.
By \eqref{e:KdV}, this is a finite-gap KdV potential if and only if
\begin{equation}\label{Eqn_finite_gap}
 8\sum_{\substack{\nu=1\\\nu\ne\mu}}^r
 \bigl(\wp'(p_\mu-p_\nu)+\wp'(p_\mu+p_\nu)\bigr)
 +\sum_{i=0}^3(2n_i+1)^2
 \wp'\Big(p_\mu-\frac{w_i}{2}\Big)=0,
\end{equation}
The system has expected dimension zero. It is reasonable to count the number of the set $\{\pm p_1,\dots,\pm p_r\}$ (or equivalently, the potential $I_{\mathbf{p}}(z)$) satisfying the system with multiplicity. Let $T(n_0,n_1,n_2,n_3;r)$ be the corresponding number. 
\begin{example}\label{e:Treivich_r_1} For $r=1$, the equation is
\[
    \sum_{i=0}^4 (2n_i+1)^2 \wp'(p - \frac{w_i}{2}) =0.
\]
$T(n_0,\dots,n_3;1)= T(0,0,0,0;1) =6$. 
\begin{proof}
The equation has $12$ poles and hence $12$ zeros. If $p$ is a solution, which is not a half period, then so is $-p$. Hence there are $12/2$ sets of solutions $\{ \pm p \}$. This argument does not depend on $n_i$. 
\end{proof}
\end{example}

\begin{theorem}[Complete Treibich counting]
For every smooth elliptic curve and every
$(n_0,n_1,n_2,n_3)\in\mathbb Z_{\geq0}^4$,
\[
 T(n_0,n_1,n_2,n_3;2)=T(0,0,0,0;2)=27,
\]
where $T$ counts unordered signed pole configurations with multiplicity.
\end{theorem}

\begin{remark}
Treibich \cite[Theorem~5.15]{Treibich_2025} proved that there
are $27$ distinct potentials for a generic smooth elliptic curve, independently of the weights $(n_0,n_1,n_2,n_3)$.
The theorem above extends this counting statement to every smooth elliptic curve by counting with multiplicities.
\end{remark}

We prove the theorem by studying the rational systems arising from collisions near the half-periods. When $a_\mu \rightarrow \frac{1}{2} w_i$, for some $\mu$ and $i$, the principal part of the system takes the following form:
\begin{equation} \label{eqn_finite_gap_principal}
    F_l(\alpha):=\sum_{\nu=1, \neq \mu}^r \Big(\frac{1}{(\alpha_\mu - \alpha_\nu)^l} + \frac{1}{(\alpha_\mu + \alpha_\nu)^l} \Big) + \frac{x}{\alpha_{\mu}^l} = \rho_{\mu}, \text{ for }\mu=1,\dots,r.
\end{equation}
We will take $l=3$ and $x = (2n_i+1)^2/8$ for $n_i \in \mathbb{N}$ and $r=2$. The discussion below works for general case. 

\begin{theorem}\label{thm:rational_counting} Given $x \neq 0$, for generic $\rho_1,\dots, \rho_r$, the system~\eqref{eqn_finite_gap_principal} has $l^r(2r-1)!!$ solutions. Furthermore, for a smaller and still generic open set of $\rho_1,\dots,\rho_r$, on all the $l^r (2r-1)!!$ solutions, their Jacobian matrices are nonsingular.
\end{theorem}
\begin{proof}
We start with the proof for $l=1$ and general case following from a transversality argument and a Bézout type counting. The $l=1$ case splits into the following steps. We formulate each step as a lemma. 
\begin{lemma} For $l=1$, $\rho_\mu=0$, and $x \in \mathbb{C} \setminus \{ -(r-1) \}$, the system has no solution.
For $x= -(r-1)$, then there are infinitely many solutions satisfying
\[
    [a_1^2: \dots :a_k^2] = C[1,\zeta,\dots, \zeta^{r-1}].
\]
\end{lemma}
\begin{proof}
Let $u_\mu := a_\mu^2$. Using the identity 
\[
    \frac{1}{a_\mu - a_\nu} + \frac{1}{a_\mu+a_\nu} = \frac{2a_\mu}{a_\mu^2 - a_\nu^2} = \frac{2a_\mu}{u_\mu - u_\nu},
\]
we divide the system by $a_\mu$ to obtain
\[
    \sum_{\nu=1,\neq \mu}^r\frac{2}{u_\mu -u_\nu} + \frac{x}{u_\mu} =0, \text{ for }\mu=1,\dots,r.
\]
The system has been studied in Lemma~\ref{Lemma_sol_x}. This proof the lemma.
\end{proof}

\begin{lemma}
For $\rho_1=1$, $\rho_2=\dots=\rho_r = 0$, and $x \in \mathbb{C} \setminus \{ 0, -1, \dots, -(2r-2) \}$, the system
\[
    \sum_{\nu=1, \neq \mu}^r \Big( \frac{1}{a_\mu - a_\nu} + \frac{1}{a_\mu+a_\nu} \Big) + \frac{x}{a_\mu} = \rho_\mu, \text{ for }\mu=1,\dots,r
\]
has $(2r-2)!!$ solutions. Furthermore, for a possibly smaller open set of $x$, the Jacobian matrices for all $(2r-2)!!$ solutions are nonsingular.
\end{lemma}
\begin{proof}
First, note that $(2r-2)!! = 2^{r-1} (r-1)!$. We claim that there is a unique solution to the system $\{ u_{\mu} = a_{\mu}^2 \}_{\mu=2}^{r}$ with $u_{\mu} \neq 0$. The counting then follows from the symmetry and the choice of square roots.

Let $a_1=1$. The remaining $r-1$ equations in the system reduce to
\[
    \sum_{\nu=2, \neq \mu}^r \frac{1}{u_\mu-u_\nu} + \frac{1}{u_\mu-1} + \frac{x}{u_\mu} = 0.
\]
If a solution exists, we define the polynomial $q(z) := \prod_{\mu=2}^r (z- u_\mu)$. It satisfies the following differential equation:
\[
    \Big( z(z-1) \frac{d^2}{dz^2} + 2(z + zx -x)\frac{d}{dz} + \lambda \Big) q(z) = 0.
\]
Solving this equation yields $\lambda = -(r-1)(r + 2x)$, and
\[
    q(z) = \sum_{j=0}^{r-1} \binom{r+2x-2}{r-1-j} \binom{r}{j} (-z)^j (1-z)^{r-1-j};
\]
which can be written in a different form:
\[
    q(z) = \binom{r+2x-2}{r-1} \cdot {}_2F_1\left(-r, r+2x-1; r; z^2\right).
\]
Note that $q(0) \neq 0$ implies 
\[
    x \notin \{ 0, -1, -2, -3, \dots, -(r-2) \}.
\]
and that $q(z)$ is a polynomial of degree $r-1$ implies 
\[
    x \notin \{ -r, -(r+1), \dots, -(2r-2) \}.
\]    
Finally, we rule out the case $x = -(r-1)$, since otherwise $c_1$ would be zero.

For the nonsingularity of Jacobian, note that each $a_\mu$ is a polynomial of $x$ and the determinant of the Jacobian of the solution is a rational function of $x$. In particular, it is nonzero since when $x>0$, the Jacobian is diagonally dominant. 
\end{proof}

\begin{lemma} For $l=1$, a fixed $x \neq 0$, and generic $\rho_\mu$, the system has $(2r-1)!!$ solutions. Furthermore, for a small and still generic open set of $\rho_\mu$, the Jacobian matrices for all $(2r-1)!!$ solutions are nonsingular.
\end{lemma}
\begin{proof}
We explicitly construct $(2r-1)!!$ distinct branches of solutions for the specific choice $\rho_\mu = \epsilon^\mu$ with $0 < |\epsilon| \ll 1$ and show that there are no other possibilities. 

Note first that the number of perfect pairings of $2r$ elements $\{\pm 1,\dots, \pm r\}$ is equal to $(2r-1)!!$. 
For generic $x$, we construct solutions parameterized by $t$ and stratify them using set partitions of $\{1, 2, \dots, r\}$.

Given a solution vector $(a_1, \dots, a_r)$, we group the indices based on their asymptotic growth rates as $\epsilon \to 0$. We say two indices $\mu$ and $\nu$ are connected (denoted $\mu \sim \nu$) if $|a_\mu / a_\nu|$ is bounded both below and above by non-zero constants independent of $t$. This equivalence relation yields a set partition $\mathcal{P} = \{B_1, \dots, B_m\}$ of $\{1, \dots, r\}$.

If a block in the partition is given by $B = \{l_1, l_2, \dots, l_s\}$, ordered such that $l_1 = \min(B)$, the corresponding coordinates share the asymptotic behavior:
\[
    a_{l_1} \sim a_{l_2} \sim \dots \sim a_{l_s} \sim \epsilon^{-l_1}.
\]
We parameterize these solutions as Laurent series in $\epsilon$:
\[
    a_{\mu} = \sum_{i \geq 0} \alpha_{\mu, i} \epsilon^{-l_1+i}, \quad \text{for } \mu \in B.
\]
Substituting this expansion into the $\mu$-th equation of the system for $\mu \in B$ and multiplying by $\epsilon^{-l_1}$, we isolate the leading coefficients $\alpha_{\mu, 0}$. Because variables in other blocks have different growth rates, their cross-terms are of strictly higher order in $\epsilon$. The leading coefficient system completely decouples into:
\[
    \sum_{\nu \in B, \neq \mu} \Big( \frac{1}{\alpha_{\mu, 0} - \alpha_{\nu, 0}} + \frac{1}{\alpha_{\mu, 0} + \alpha_{\nu, 0}} \Big) + \frac{x}{\alpha_{\mu, 0}} = \delta_{\mu, l_1}, \quad \text{for } \mu \in B.
\]
The Kronecker delta $\delta_{\mu, l_1}$ appears on the right because $d_{l_1} = \epsilon^{l_1}$ exactly matches the $\epsilon^{l_1}$ order of the left-hand side (yielding $1$), while for any larger index $\mu > l_1$ in the block, $d_\mu = \epsilon^\mu = o(\epsilon^{l_1})$ vanishes at the leading order.

This specific subsystem of size $s$ has been studied in the previous lemma, which guarantees exactly $(2s-2)!!$ solutions. For generic $x > 0$, the Jacobian of this leading-order system is non-singular, allowing us to solve for the higher-order coefficients $\alpha_{\mu, i}$ ($i \ge 1$) uniquely by the Implicit Function Theorem. Hence, the solvability of the leading coefficient system strictly corresponds to valid formal series solutions.

Since the number of solutions $(2s-2)!!$ matches the number of perfect pairings of $\{\pm l_1, \dots, \pm l_s\}$ that form a single connected component, summing over all possible set partitions exactly recovers the total number of unconstrained pairings:
\[
    \sum_{\mathcal{P} \vdash \{1, \dots, r\}} \prod_{B \in \mathcal{P}} (2|B|-2)!! = (2r-1)!!,
\]
where $\mathcal{P} \vdash \{1, \dots, r\}$ denotes a set partition.

If there were any other valid asymptotic growth rates, we could always find a sub-block whose leading coefficient system reduces to Lemma 3.1 with $d_\mu = 0$ for all $\mu$, which admits no solution. Therefore, this construction explicitly exhausts all $(2r-1)!!$ branches. This competes the proof for the number of solution.

For the Jacobian part, note that on each partition, we reduce the equation to previous lemma. The nonsingularity follows directly from the previous lemma.
\end{proof}

For general $l \geq 2$, let $\mathcal{U} = \{ \alpha \in \mathbb{C}^r \mid \alpha_\mu \neq 0 \text{ and } \alpha_\mu \neq \pm \alpha_\nu \text{ for } \mu \neq \nu \}$ be the open configuration space of valid roots. We compute the degree of the morphism $F_l: \mathcal{U} \to \mathbb{C}^r$ via intersection theory. We factor the system as the composition $F_l = L_x \circ \Pi_l \circ \Psi$ through the ambient space $\mathbb{C}^N$, where $N = 2r(r-1) + r$: 
\begin{enumerate}
    \item The rational embedding $\Psi: \mathcal{U} \hookrightarrow \mathbb{C}^N$ mapping the roots to their respective pole coordinates:
    \[
        \Psi(\alpha) = \left( \dots, \frac{1}{\alpha_\mu - \alpha_\nu}, \dots, \frac{1}{\alpha_\mu + \alpha_\nu}, \dots, \frac{1}{\alpha_\mu} \right).
    \]
    Let $M_r = \Psi(\mathcal{U})$ denote the $r$-dimensional image of this embedding.
    \item The $l$-th power map $\Pi_l(z) = (z_1^l, \dots, z_N^l)$.
    \item The linear projection $L_x: \mathbb{C}^N \to \mathbb{C}^r$ corresponding to the summations in the rational system.
\end{enumerate}

Solving $F_l(\alpha) = \rho_l$ is geometrically equivalent to intersecting $M_r$ with the fiber $\mathcal{F}_{\boldsymbol{\rho}} = \Pi_l^{-1}(L_x^{-1}(\boldsymbol{\rho}))$.

Note that the Jacobian determinant is not identically zero, since as $\alpha_1 \rightarrow 0$, it is strictly dominated by the diagonal term $-lx/ \alpha_1^{l+1}$. Hence, for a generic $\rho \in \mathbb{C}^r$, all preimages $F_l^{-1}(\rho_l)$ have strictly nonsingular Jacobians. This guarantees that $M_r$ intersects the fiber $\mathcal{F}_\rho$ transversely.

Since $L_x$ is linear, $L_x^{-1}(\boldsymbol{\rho})$ is the intersection of $r$ generic hyperplanes, $H_1 \cap \dots \cap H_r$. Its pullback $\mathcal{F}_{\boldsymbol{\rho}}$ is therefore the intersection of $r$ degree-$l$ hypersurfaces. 
Because the intersection is transverse, the intersection number exactly counts the distinct solutions. We obtain:
\[
    (l[H_1]) \cdot \dots \cdot (l[H_r]) \cdot [M_r] = l^r ([H_1]) \cdot \dots \cdot ([H_r]) \cdot [M_r]  = l^r (2r-1)!!.
\]
The last equality reduces to $l=1$ computation. This completes the proof.
\end{proof}

\begin{theorem} \label{thm:Treibich_conj}
The system 
\[
\begin{cases}
     \displaystyle F_1 =  8 \Big( \wp'(p_1-p_2) + \wp'(p_1+p_2) \Big) + \sum_{i=0}^3 (2n_i+1)^2  \wp'(p_1 - \frac{w_i}{2}) =0, 
    \\
     \displaystyle F_2 = 8 \Big( \wp'(p_2-p_1) + \wp'(p_2+p_1) \Big) + \sum_{i=0}^3 (2n_i+1)^2  \wp'(p_2 - \frac{w_i}{2}) =0,
\end{cases}
\]
has $216$ solutions.
\end{theorem}
\begin{proof}
The technique is the same as the proof in Theorem~\ref{thm_deg}. First we deform the system to generic values $F_\mu = \rho_\mu$, for $\mu=1,2$. 
We claim that by choosing a generic infinity $(\rho_1,\rho_2) \rightarrow \infty$. The solutions cluster around the poles of equations (the half periods $\frac{\omega_i}{2}$).  
A diagonal or antidiagonal collision away from the half-periods would force, respectively,
\[
    \rho_1+\rho_2=O(1)
    \qquad\text{or}\qquad
    \rho_1-\rho_2=O(1).
\]
This would not happen for generic choice of $(\rho_1,\rho_2)$. Thus both points approach half-periods.

By Theorem~\ref{thm:rational_counting}, applied with
$l=3$ and $x=(2n_i+1)^2/8$, a cluster of size one
contributes $3$ solutions, while a cluster of size two
contributes $3^2(2\cdot2-1)!!=27$ solutions.
Distinct half-periods therefore contribute $3^2=9$ solutions
for each of the $4\cdot3=12$ ordered choices.
A common half-period contributes $27$ solutions for each
of the four choices. Hence the generic deformed system has
\[
 4\cdot3\cdot9+4\cdot27=216
\]
solutions.

Now let $(\rho_1,\rho_2)\to(0,0)$.
If a solution approaches the boundary, where
$a_\mu=\omega_i/2$ or $a_1=\pm a_2$, boundedness of $(\rho_1,\rho_2)$ forces both poles to approach a common half-period. Rescaling then yields a solution of the system in Lemma~\ref{lem:rational_boundary}, which requires
\[
 \frac{(2n_i+1)^2}{8}=\frac12.
\]
This is impossible for integer weights.
Hence no solution escapes, and the count remains $216$. 
\end{proof}

\begin{corollary}[Treibich counting] \label{cor:Treibich}  Given $(n_0,n_1,n_2,n_3) \in \mathbb{Z}_{\geq 0}^4$, then 
\[
\begin{split}
    T(n_0,n_1,n_2,n_3;2) &= T(0,0,0,0,2) =: \frac{216}{2^2 2!} = 27.
\end{split}
\]
\end{corollary}
\begin{proof}
We explain the denominator $2^2 2!$. Any permutation and sign choice of a solution 
\[
    \{ \mu_1 p_{\sigma(1)},  \mu_2 p_{\sigma(2)}   \}, \text{ where }\mu_2 \in \{ \pm 1\}, \text{ and } \sigma\in S_2,
\]
will give the same set $\{ \pm p_1,\pm p_2 \}$.
\end{proof}

\begin{remark}
The rational-limit method provides a framework for counting configurations with any number of pole pairs. For $r\geq3$, determining the elliptic count requires further analysis of
the additional collision configurations and their contributions. We leave this extension to a future work.
\end{remark}
\appendix

\bibliographystyle{plain}
    
\bibliography{zbib}
\end{document}